\documentclass{amsart}

\usepackage{lcad}
\usepackage{amsfonts}
\usepackage{amssymb}
\usepackage{enumerate}
\usepackage{amsmath}
\usepackage{amsthm}
\usepackage{graphicx}

\newtheorem{theorem}{Theorem}[section]
\newtheorem{lemma}[theorem]{Lemma}
\newtheorem{corollary}[theorem]{Corollary}

\newtheorem{fact}[theorem]{Fact}
\newtheorem{problem}[theorem]{Problem}
\newtheorem*{x}{Theorem \ref{<}}
\newtheorem*{abc1}{Theorem \ref{t:CCD}}
\newtheorem*{abc2}{Theorem \ref{t:lower}}
\newtheorem*{y}{Theorem \ref{sierszth}}
\newtheorem*{z}{Theorem \ref{tkak}}
\newtheorem*{w}{Problem \ref{p:Brown}}
\newtheorem*{uu}{Corollary \ref{c:ft}}
\newtheorem*{uuu}{Corollary \ref{c:self}}

\newtheorem*{uuuuu}{Theorem \ref{t:tdecomp}}
\newtheorem*{uuuuuu}{Theorem \ref{t:tHdecomp}}
\newtheorem*{abc3}{Corollary \ref{c:tH<H}}

\theoremstyle{definition}
\newtheorem{example}[theorem]{Example}
\newtheorem{remark}[theorem]{Remark}
\newtheorem{definition}[theorem]{Definition}
\newtheorem{notation}[theorem]{Notation}

\numberwithin{equation}{section}

\DeclareMathOperator{\pr}{pr}
\DeclareMathOperator{\diam}{diam}
\DeclareMathOperator{\inter}{int}
\DeclareMathOperator{\dist}{dist}
\DeclareMathOperator{\Lip}{Lip}
\DeclareMathOperator{\cl}{cl}
\DeclareMathOperator{\ran}{Ran}

\begin{document}

\title{A new fractal dimension: The topological Hausdorff dimension}

\author{Rich\'ard Balka}

\address{Department of Mathematics, University of Washington, Box 354350, Seattle, WA 98195-4350, USA and
Alfr\'ed R\'enyi Institute of Mathematics, Hungarian Academy of Sciences, PO Box 127, 1364 Budapest, Hungary
and Institute of Mathematics and Informatics, Eszterh\'azy K\'aroly College, Le\'anyka u. 4., 3300 Eger, Hungary}

\email{balka@math.washington.edu}

\thanks{The first author was supported by the
Hungarian Scientific Research Fund grant no.~72655.}

\author{Zolt\'an Buczolich}

\address{Institute of Mathematics, E\"otv\"os Lor\'and University,
P\'azm\'any P\'eter s. 1/c, 1117 Budapest, Hungary}

\email{buczo@cs.elte.hu}

\thanks{The second author was supported by the Hungarian
Scientific Research Fund grants no. K075242 and K104178.}

\author{M\'arton Elekes}

\address{Alfr\'ed R\'enyi Institute of Mathematics, Hungarian Academy of Sciences,
PO Box 127, 1364 Budapest, Hungary and Institute of Mathematics, E\"otv\"os Lor\'and
University, P\'azm\'any P\'eter s. 1/c,
1117 Budapest, Hungary}

\email{elekes.marton@renyi.mta.hu}

\thanks{The third author was supported by the Hungarian
Scientific Research Fund grants no.~72655, 61600, and 83726.}



\begin{abstract}
We introduce a new concept of dimension for metric spaces, the so-called
\emph{topological Hausdorff dimension}. It is defined by a very natural
combination of the definitions of the topological dimension and the Hausdorff
dimension. The value of the topological Hausdorff dimension is always between
the topological dimension and the Hausdorff dimension, in particular, this new
dimension is a non-trivial lower estimate for the Hausdorff dimension.

We examine the basic properties of this new notion of dimension, compare it to other well-known notions,
determine its value for some classical fractals such as the Sierpinski carpet, the von Koch snowflake curve,
Kakeya sets, the trail of the Brownian motion, etc.

As our first application, we generalize the celebrated result of Chayes,
Chayes and Durrett about the phase transition of the connectedness of the
limit set of Mandelbrot's fractal percolation process. They proved that
certain curves show up in the limit set when passing a critical probability,
and we prove that actually `thick' families of curves show up, where roughly
speaking the word thick means that the curves can be parametrized in a natural
way by a set of large Hausdorff dimension. The proof of this is basically a
lower estimate of the topological Hausdorff dimension of the limit set.
For the sake of completeness, we also give an upper estimate and
conclude that in the non-trivial cases the topological Hausdorff dimension is
almost surely strictly below the Hausdorff dimension.

Finally, as our second application, we show that the topological Hausdorff dimension is precisely the right notion to describe the Hausdorff dimension of the level sets of the generic continuous function (in the sense of Baire category) defined on a compact metric space.
\end{abstract}

\maketitle

\tableofcontents

\section{Introduction}

The term `fractal' was introduced by Mandelbrot in his celebrated book
\cite{Man}. He formally defined a subset of a Euclidean space to be a fractal if its topological
dimension is strictly smaller than its Hausdorff dimension. This is just one example
to illustrate the fundamental role these two notions of dimension play in
the study of fractal sets. To mention another such example, let us recall that
the topological dimension of a metric space $X$ is the infimum of the Hausdorff
dimensions of the metric spaces homeomorphic to $X$, see \cite{HW}.

The main goal of this paper is to introduce a new concept of dimension, the
so-called \emph{topological Hausdorff dimension}, that interpolates the two
above mentioned dimensions in a very natural way. Let us recall the
definition of the (small inductive) topological dimension (see e.g. \cite{Eng,HW}).

\begin{definition} Set $\dim _{t} \emptyset = -1$. The \emph{topological dimension}
of a non-empty metric space $X$ is defined by induction as
\[
\dim_{t} X=\inf\{d: X \textrm{ has a basis } \mathcal{U} \textrm{
such that } \dim_{t} \partial {U} \leq d-1 \textrm{ for every } U\in
\mathcal{U} \}.
\]
\end{definition}

Our new dimension will be defined analogously, however, note that this second
definition will not be inductive, and also that it can attain non-integer values as well.
The Hausdorff dimension of a metric space $X$
is denoted by $\dim_{H} X$, see e.g.~\cite{F} or \cite{Ma}.
In this paper we adopt the convention that $\dim_{H}\emptyset = -1$.

\begin{definition}\label{deftoph}
Set $\dim _{tH} \emptyset=-1$. The \emph{topological Hausdorff dimension}
of a non-empty metric space $X$ is defined as
\[
\dim_{tH} X=\inf\{d:
 X \textrm{ has a basis } \mathcal{U} \textrm{ such that } \dim_{H} \partial {U} \leq d-1 \textrm{ for every } U\in \mathcal{U} \}.
\]
\end{definition}
(Both notions of dimension can attain the value $\infty$ as well, actually we
use the convention $\infty - 1 = \infty$, hence $d = \infty$ is a member of the
above set.)

\bigskip

From now on \emph{generic} is always understood in the sense of Baire category.

\bigskip

It was not this analogy that initiated the study of this new concept.
Over the last $30$ years there has been a large interest in studying dimensions of
various sets related to generic continuous maps. Mauldin and Williams \cite{MW} proved that the Hausdorff dimension of the graph of the generic $f\in C[0,1]$ is $1$, while Humke and Petruska \cite{HP} showed that its packing dimension is $2$. The box dimensions of graphs of generic continuous functions were investigated by Hyde, Laschos, Olsen, Petrykiewicz, and Shaw \cite{HLOPS}.
Balka, Farkas, Fraser and Hyde \cite{BFFH} proved that for the generic continuous map $f\colon K\to \mathbb{R}^n$ we have $\dim_H f(K)=\min\{\dim_t K,n\}$ for all compact metric spaces $K$. Bruckner and Garg \cite{BG} described the level set structure of the generic $f\in C[0,1]$ from the topological point of view. From the metric point of view, it is well-known that each non-empty level set of the generic $f\in C[0,1]$ is of Hausdorff dimension $0$. Kirchheim \cite{BK} considered the level sets of the generic continuous function $f\colon [0,1]^d\to \mathbb{R}$.
He proved that for every $y\in \inter f([0,1]^d)$ we have $\dim_{H}f^{-1}(y)=d-1$, that is, as one would expect, `most' level sets are of Hausdorff dimension $d-1$. The next problem is about generalizations of this result to fractal sets in place of $[0,1]^d$,
this is where our original motivation came from.

\begin{problem}\label{p:Buczo}
Describe the Hausdorff dimension of the level sets of the generic continuous
function defined on a compact metric space.
\end{problem}

It has turned out that the topological Hausdorff dimension is the right
concept to deal with this problem. We will essentially prove that the value $d-1$ in
Kirchheim's result has
to be replaced by $\dim_{tH} K - 1$, see the end of this introduction or Section~\ref{s:application} for
the details.

\bigskip

We would also like to mention another potentially very interesting motivation
of this new concept. Unlike most well-known notions of dimension, such as
packing or box-counting dimensions, the topological Hausdorff dimension is
smaller than the Hausdorff dimension.  As it is often an important and
difficult task to estimate the Hausdorff dimension from below, this gives
another reason to study the topological Hausdorff dimension.

\bigskip

It is also worth mentioning that there is another recent approach
by Urba\'nski \cite{Ur} to combine the topological dimension and the Hausdorff
dimension. However, his new concept, called the transfinite Hausdorff dimension
is quite different in nature from ours, e.g. it takes ordinal numbers as values.

\bigskip

Moreover, the first listed author (using ideas of U. B. Darji and the
third listed author) recently generalized the results of the paper for maps
taking values in $\mathbb{R}^n$ instead of $\mathbb{R}$.
The new concept of dimension needed
to describe the Hausdorff dimension of the fibers of the generic
continuous map is called the $n$th inductive topological Hausdorff
dimension, see \cite{B}.

\bigskip

Next we say a few words about the main results and the organization of the paper.

\bigskip

In Section~\ref{s:equivalent} we discuss some alternative definitions of the
topological Hausdorff dimension yielding the same concept. Recall that the
following classical theorem in fact describes an alternative recursive
definition of the topological dimension.

\begin{uuuuu} If $X$ is a non-empty separable metric space then
\[
\dim_t X = \min \{ d : \exists A \subseteq X \textrm{ such that } \dim_t
A\leq d-1 \textrm{ and } \dim_t (X \setminus A)\leq 0\}.
\]
\end{uuuuu}

The next result shows that by replacing one instance of $\dim_t
A$ by $\dim_{H} A$ we again obtain
the notion of topological Hausdorff dimension.

\begin{uuuuuu} If $X$ is a non-empty separable metric space then
\[
\dim_{tH} X = \min \{ d : \exists A \subseteq X \textrm{ such that }
\dim_H A\leq d-1 \textrm{ and } \dim_t (X \setminus A)\leq 0\}.
\]
\end{uuuuuu}

As a corollary we also obtain that actually $\inf = \min$ in our original
definition of the topological Hausdorff dimension, which is not only
interesting, but will also be used in one of the applications. We discuss the
analogues of the other definitions of the topological dimension as well, such as
the large inductive dimension and the Lebesgue covering dimension.

\bigskip

In Section~\ref{s:properties} we investigate the basic properties of the topological Hausdorff dimension.
Among others, we prove the following.

\begin{x}
$\dim_{t} X \le \dim_{tH} X  \le \dim_{H} X$.
\end{x}

We also verify that $\dim_{tH} X$ satisfies some standard properties
of a dimension, such as monotonicity, bi-Lipschitz invariance and
countable stability for closed sets. We discuss the existence of $G_\delta$
hulls and $F_\sigma$ subsets with the same topological Hausdorff dimension, as well.
Moreover, we check that this concept is genuinely new in that
$\dim_{tH} X$ cannot be expressed as a function of $\dim_{t}X$ and $\dim_{H} X$.

\bigskip

In Section~\ref{s:examples} we compute $\dim_{tH}X$ for some classical fractals, like the
Sierpi\'nski triangle and carpet, the von Koch curve, etc. For example

\begin{y}
Let $T$ be the Sierpi\'nski carpet. Then $\dim_{tH} T = \frac{\log 6}{\log 3}=\frac{\log 2}{\log 3}+1 $.
\end{y}
(Note that $\dim_t T = 1$ and $\dim_H T =\frac{\log 8}{\log 3}$
while the Hausdorff dimension of the triadic Cantor set equals $\frac{\log 2}{\log 3}$.)

\bigskip

We also consider Kakeya sets (see \cite{F} or \cite{Ma}). Unfortunately, our methods do not give any useful information concerning the Kakeya Conjecture.

\begin{z}
For every $d \in \mathbb{N}^+$ there exist a compact Kakeya set of
topological Hausdorff dimension 1 in $\mathbb{R}^d$.
\end{z}

Following K\"orner~\cite{Korner} we  prove somewhat more, since we essentially show that the generic element of a carefully chosen space is a Kakeya set of topological Hausdorff dimension 1.

\bigskip

We show that the range of the Brownian motion almost surely (i.e. with probability $1$)
has topological Hausdorff dimension $1$ in every dimension except perhaps $2$ and $3$.
These two cases remain the most intriguing open problems of the paper.

\begin{w}
Determine the almost sure topological Hausdorff dimension of the range of the $d$-dimensional Brownian motion for $d=2$ or $3$.
Equivalently, determine the smallest $c \ge 0$ such that the range can be decomposed into a totally disconnected set and a set of
Hausdorff dimension at most $c-1$ almost surely.
\end{w}

We also relate
the planar case to a well-known open problem of W. Werner and solve the dual
version of this problem, that is, the version in which the notion of
Wiener measure is replaced by Baire category. In a similar vein, we also show
that the range of the generic continuous map
$f \colon [0,1] \to \mathbb{R}^d$ is of topological Hausdorff dimension $1$ for every $d$.

\bigskip

As our first application in Section~\ref{s:Mandelbrot} we generalize a result of Chayes,
Chayes and Durrett about the phase transition of the connectedness of the
limit set of Mandelbrot's fractal percolation process. This limit set $M = M^{(p,n)}$ is a
random Cantor set, which is constructed by dividing the unit square into $n \times n$ equal sub-squares and
keeping each of them independently with probability $p$, and then repeating the same procedure recursively
for every sub-square. (See Section~\ref{s:Mandelbrot} for more details.)

\begin{abc1}[Chayes-Chayes-Durrett, \cite{CH}]
There exists a critical probability $p_c = p_c^{(n)} \in (0,1)$ such that if
$p < p_c$ then $M$ is totally disconnected almost surely, and if $p > p_c$ then
$M$ contains a nontrivial connected component with positive probability.
\end{abc1}

It will be easy to see that this theorem is a special case of our next result.

\begin{abc2}
For every $d \in [0,2)$ there exists a critical probability $p_c^{(d)} = p_c^{(d, n)}  \in
(0,1)$ such that if $p < p_c^{(d)}$ then $\dim_{tH} M \le d$ almost surely,
and if $p > p_c^{(d)}$ then $\dim_{tH} M > d$ almost surely (provided $M \neq \emptyset$).
\end{abc2}

Theorem \ref{t:CCD} essentially says that certain curves show up at the
critical probability, and our proof will show that even `thick' families of curves
show up, which roughly speaking means a `Lipschitz-like copy' of $C \times [0,1]$ with $\dim_H C > d-1$.

We also give a numerical upper bound for $\dim_{tH} M$ which implies the following.

\begin{abc3}
Almost surely
\[
\dim_{tH} M < \dim_{H} M \textrm{ or } M = \emptyset.
\]
\end{abc3}

\bigskip

In Section~\ref{s:application} we answer Problem \ref{p:Buczo} as follows.

\begin{uu} If $K$ is a compact metric space with $\dim_t K > 0$ then for the generic $f \in C(K)$
$$\sup \left\{ \dim_{H}f^{-1}(y) : y \in \mathbb{R} \right\} = \dim_{tH} K - 1.$$
\end{uu}

(It is well-known that if
$\dim_t K = 0$ then the generic $f \in C(K)$ is one-to-one, thus every non-empty level set is of Hausdorff dimension 0, see
\cite[Lemma~2.6]{BBE2} for a proof.)

If $K$ is also sufficiently homogeneous, e.g.
self-similar then we can actually say more.

\begin{uuu} Let $K$ be a self-similar compact metric space with
$\dim_{t}K>0$. Then
for the generic $f\in C(K)$ for the generic $y\in f(K)$
$$\dim_{H} f^{-1}(y)=\dim_{tH}K-1.$$
\end{uuu}

It can actually also be shown that the supremum is attained in Corollary~\ref{c:ft}. On the other hand,
one cannot say more in a sense, since  there is a $K\subseteq \mathbb{R}^2$ such that for
the generic $f \in C(K)$ there is a \emph{unique} $y \in \mathbb{R}$
for which $\dim_{H}f^{-1}(y) = \dim_{tH} K - 1$. Moreover, in
certain situations we can replace `the generic $y \in f(K)$' with
`for every $y\in \inter f(K)$' as in Kirchheim's theorem. The
results of this last paragraph appeared elsewhere, see
\cite{BBE2}.

\bigskip

Finally, in Section~\ref{c:problems} we list some  open problems.

\section{Preliminaries}

Let $(X,d)$ be a metric space.
For $A,B \subseteq X$ let us define $\dist(A,B) = \inf\{d(x,y) : x\in A,~y\in B\}$.
Let $B(x,r)$ and $U(x,r)$ stand for the closed and open ball of
radius $r$ centered at $x$, respectively. Set $B(A,r) = \{x \in X : \dist(\{x\},A) \le r \}$ and
$U(A,r) = \{x \in X : \dist(\{x\},A) < r \}$. We denote by $\cl A$, $\inter A$ and
$\partial A$ the closure, interior and boundary of $A$, respectively.
The diameter of a set $A$ is denoted by $\diam A$. We use the convention $\diam \emptyset = 0$.
For two metric spaces $(X,d_{X})$ and $(Y,d_{Y})$ a function
$f\colon X\to Y$ is \emph{Lipschitz} if there exists a constant $C
\in \mathbb{R}$ such that $d_{Y}(f(x_{1}),f(x_{2}))\leq C \cdot
d_{X}(x_{1},x_{2})$ for all $x_{1},x_{2}\in X$. The smallest such
constant $C$ is called the Lipschitz constant of $f$ and denoted by
$\Lip(f)$. A function $f\colon X\to Y$ is called \emph{bi-Lipschitz}
if $f$ is a bijection and both $f$ and $f^{-1}$ are Lipschitz.
For a metric space $X$ and $s \ge 0$ the \emph{$s$-dimensional Hausdorff measure} is defined as
\begin{align*}
\mathcal{H}^{s}(X)&=\lim_{\delta\to 0+}\mathcal{H}^{s}_{\delta}(X)
\mbox{, where}\\
\mathcal{H}^{s}_{\delta}(X)&=\inf \left\{ \sum_{i=1}^\infty (\diam
U_{i})^{s}: X \subseteq \bigcup_{i=1}^{\infty} U_{i},~
\forall i \diam U_i \le \delta \right\}.
\end{align*}
Recall that
$$\mathcal{H}^{s}_{\infty}(X)=\inf \left\{ \sum_{i=1}^\infty (\diam
U_{i})^{s}: X \subseteq \bigcup_{i=1}^{\infty} U_{i}\right\}$$
is the \emph{$s$-dimensional Hausdorff content}.

The \emph{Hausdorff dimension of $X$} is defined as
\[
\dim_{H} X = \inf\{s \ge 0: \mathcal{H}^{s}(X) =0\}.
\]
It is not difficult to see using the regularity of $\mathcal{H}^{s}_{\delta}$ that
every set is contained in a $G_\delta$ set of the same Hausdorff
dimension. For more information on these concepts see \cite{F} or \cite{Ma}.

Let $X$ be a \emph{complete} metric space. A set is \emph{somewhere dense} if
it is dense in a non-empty open set, and otherwise it is called \emph{nowhere
  dense}. We say that $M \subseteq X$ is
\emph{meager} if it is a countable union of nowhere dense sets, and
a set is called \emph{co-meager} if its complement is meager. By
Baire's Category Theorem co-meager sets are dense. It is not
difficult to show that a set is co-meager iff it contains a dense
$G_\delta$ set. We say that the generic element $x \in X$ has
property $\mathcal{P}$ if $\{x \in X : x \textrm{ has property }
\mathcal{P} \}$ is co-meager. The term `typical' is also used
instead of `generic'. See e.g. \cite{Ke} for more on these concepts.

A topological space $X$ is called \emph{totally disconnected} if $X$ is
empty or every connected component of $X$ is a singleton. If $\dim_t X\leq 0$
then $X$ is clearly totally disconnected.  If $X$ is locally compact then
$\dim_t X\leq 0$ iff $X$ is totally disconnected, see \cite[1.4.5.]{Eng}. If $X$ is a $\sigma$-compact metric space
then $\dim_t X\leq 0$ iff $X$ is totally disconnected, see the
countable stability of topological dimension zero for closed sets \cite[1.3.1.]{Eng}.

\section{Equivalent definitions of the topological Hausdorff dimension}
\label{s:equivalent}

The goal of this section is to prove some equivalent definitions which will
play an important role later. Perhaps the main point here is that while the
original definition is of local nature, we can find an equivalent global
definition.

Let us recall the following decomposition theorem for the topological
dimension, see \cite[1.5.7.]{Eng}, which can be regarded as an equivalent
definition.

\begin{theorem}\label{t:tdecomp} If $X$ is a non-empty separable metric space then
$$\dim_t X=\min\{d: \exists A\subseteq X \textrm{ such that } \dim_t A\leq d-1 \textrm{ and } \dim_t (X\setminus A)\leq 0\}.$$
\end{theorem}

(Note that as above, $\infty$ is assumed to be a member of the above set,
moreover we use the convention $\min \{ \infty \} = \infty$.)

\bigskip

The main goal of this section is to prove Theorem \ref{t:tHdecomp},
an analogous decomposition theorem for the topological Hausdorff dimension,
which yields an equivalent definition of the topological Hausdorff dimension.
As a by-product, we obtained that in the definition of $\dim_{tH} X$ the infimum
is attained, see Corollary~\ref{c:inf=min}.

\begin{remark}
One can actually check that, as usual in dimension theory, the assumption of
separability cannot be dropped.
\end{remark}

Let $X$ be a given non-empty separable metric space. For the sake of
notational simplicity we use the following notation.

\begin{notation}
\begin{align*}
P_{tH}&=\{d: X \textrm{ has a basis } \mathcal{U} \textrm{ such that } \dim_{H} \partial {U} \leq d-1 \textrm{ for every } U\in \mathcal{U} \},\\
P_{dH}&=\{d: \exists A\subseteq X \textrm{ such that } \dim_H A\leq d-1 \textrm{ and } \dim_t (X\setminus A)\leq 0\}.
\end{align*}
We assume $\infty \in P_{tH},P_{dH}$.
\end{notation}

The following lemmas are the heart of the section.

\begin{lemma}\label{l:P=P} $P_{tH}=P_{dH}$.
\end{lemma}

\begin{proof} First we prove $P_{tH}\subseteq P_{dH}$. Assume $d\in P_{tH}$ and $d<\infty$. Then there exists a countable basis $\mathcal{U}$ of $X$ such that
$\dim_H \partial U\leq d-1$ for all $U\in \mathcal{U}$. Let $A=\bigcup_{U\in \mathcal{U}} \partial U$, then the countable stability of the Hausdorff dimension yields
$\dim_H A\leq d-1$, and the definition of $A$ clearly implies $\dim_t (X\setminus A)\leq 0$. Hence $d\in P_{dH}$.

Now we prove $P_{dH}\subseteq P_{tH}$. Assume $d\in P_{dH}$ and $d<\infty$. Let us fix $x\in X$ and $r>0$. To verify
$d\in P_{tH}$ we need to find an open set $U\subseteq X$ such that $x\in U \subseteq U(x,r)$ and $\dim_{H} \partial U\leq d-1$. Since $d\in P_{dH}$,
there is a set $A\subseteq X$ such that $\dim_H A\leq d-1$ and $\dim_t(X\setminus A)\leq 0$.
As $X\setminus A$ is a separable subspace of $X$ with topological
dimension $0$, by the separation theorem for topological dimension
zero \cite[1.2.11.]{Eng} there is a so-called partition between $x$
and $X\setminus U(x,r)$ disjoint from $X\setminus A$. This means that
there exist disjoint open sets $U,U'\subseteq X$ such that $ x\in
U$, $X\setminus U(x,r)\subseteq U'$ and $\left(X\setminus (U\cup
U')\right)\cap (X\setminus A)=\emptyset$. In particular, $x \in U
\subseteq U(x,r)$. Moreover, $\partial U\cap (X\setminus A)=\emptyset$,
so $\partial U\subseteq A$, thus $\dim_{H} \partial U\leq
\dim_{H} A\leq d-1$. Hence $d\in P_{tH}$.
\end{proof}

\begin{lemma}\label{l:inf}  $\inf P_{dH}\in P_{dH}$.
\end{lemma}

\begin{proof} Let $d=\inf P_{dH}$, we may assume $d<\infty$. Set $d_n=d+1/n$
for all $n\in \mathbb{N}^+$.  As $d_n\in P_{dH}$, there exist sets $A_n\subseteq X$
such that $\dim_H A_n\leq d_n-1$ and $\dim_t(X\setminus A_n)\leq 0$. We may
assume that the sets $A_n$ are $G_{\delta}$, since we can take $G_{\delta}$
hulls with the same Hausdorff dimension. Let $A=\bigcap_{n=1}^{\infty} A_n$,
then clearly $\dim_H A\leq d-1$.  As $X\setminus A_n$ are $F_\sigma$ sets such
that $\dim_t(X\setminus A_n) \leq 0$ and $X\setminus A =
\bigcup_{n=1}^{\infty} (X\setminus A_n)$, countable stability of topological dimension zero for $F_{\sigma}$ sets
\cite[1.3.3.~Corollary]{Eng} yields $\dim_t (X\setminus A)\leq 0$. Hence $d\in
P_{dH}$.
\end{proof}

Now the main result of the section is an easy consequence of Lemmas~\ref{l:P=P}, ~\ref{l:inf} and the definition of the topological Hausdorff
dimension.

\begin{theorem}\label{t:tHdecomp} If $X$ is a non-empty separable metric space then
$$\dim_{tH} X=\min\{d: \exists A\subseteq X \textrm{ such that } \dim_H A\leq d-1 \textrm{ and } \dim_t (X\setminus A)\leq 0\}.$$
\end{theorem}

Since every set is contained in a $G_\delta$ set of the same Hausdorff
dimension, and a $\sigma$-compact metric space is totally disconnected iff it has topological dimension zero, the above theorem yields
the following equivalent definition.

\begin{theorem}\label{t:tHdecomp2}
For a non-empty $\sigma$-compact metric space $X$
\begin{align*}
\dim_{tH} X=\min\{&d: \exists A\subseteq X \textrm{ such that } \dim_H A\leq d-1 \\
&\textrm{and } X\setminus A \textrm{ is totally disconnected}\}.
\end{align*}
\end{theorem}

Moreover, as a by-product of Lemma~\ref{l:P=P} and Lemma~\ref{l:inf} we obtain that the infimum is attained in the
original definition, which will play a role in one of the applications.

\begin{corollary}\label{c:inf=min} If $X$ is a non-empty separable metric space then
$$\dim_{tH} X=\min\{d:
 X \textrm{ has a basis } \mathcal{U} \textrm{ such that } \dim_{H} \partial {U} \leq d-1 \textrm{ for every } U\in \mathcal{U} \}.$$
\end{corollary}

\begin{remark}
There are even more equivalent definitions of the topological dimension, and
one can prove that the appropriate analogues result in the same notion as well,
but since these statements will not be used in the sequel we only state the
results here. For some details consult \cite{B}.

Moreover, Section~\ref{s:application} contains further equivalent definitions
in the compact case in connection with the level sets of generic continuous
functions.

The next classical theorem shows that the notion of large inductive dimension
coincides with the other definitions of topological dimension for separable
metric spaces.

\begin{theorem} For every non-empty separable metric space $X$
\begin{align*} \dim_{t} X =\min\{&d: \forall F \subseteq X \textrm{ closed and } \forall V \textrm{ open with } F \subseteq V,  \ \exists U \textrm{ open} \\ &\textrm{such that } F \subseteq U \subseteq V \textrm{ and } \dim_{t} \partial U \leq d-1 \}.
\end{align*}
\end{theorem}

The natural analogue yields the same concept again.

\begin{theorem} For every non-empty separable metric space $X$
\begin{align*} \dim_{tH} X=\min\{&d: \forall F \subseteq X \textrm{ closed and } \forall V \textrm{ open with } F \subseteq V,  \ \exists U \textrm{ open} \\ &\textrm{such that } F \subseteq U \subseteq V \textrm{ and } \dim_{H} \partial U \leq d-1 \}.
\end{align*}
\end{theorem}

Next we take up the Lebesgue covering dimension.

For a family $\mathcal{A}$ of sets and $m\in \mathbb{N}^+$ let $T_{m}(\mathcal{A})$
denote the set of points covered by at least $m$ members of $\mathcal{A}$.

\begin{theorem} For every non-empty separable metric space $X$
\begin{align*}
\dim_{t} X=\min \{&d: \forall \, \textrm{finite open cover } \mathcal{U} \textrm{ of } X ~ \exists \, \textrm{a finite open refinement } \\
&\mathcal{V}  \textrm{ of } \mathcal{U} \textrm{ such that } T_{d+2}(\mathcal{V}) = \emptyset\}.
\end{align*}
\end{theorem}

The analogous result is the following.

\begin{theorem} For every separable metric space $X$ with $\dim_t X > 0$
\begin{align*}
\dim_{tH} X=\min \{&d: \forall  \varepsilon>0 \ \forall \, \textrm{finite open cover } \mathcal{U} \textrm{ of } X ~ \exists \, \textrm{a finite open refinement } \\
&\mathcal{V}  \textrm{ of } \mathcal{U} \textrm{ such that }\mathcal{H}_{\infty}^{d-1+\varepsilon}(T_{2}(\mathcal{V}))\leq \varepsilon\}.
\end{align*}
\end{theorem}

\end{remark}

\section{Basic properties of the topological Hausdorff dimension}
\label{s:properties}

The main goal of this section is to investigate the basic properties of the
topological Hausdorff dimension. It will turn out that some properties of this dimension are metric in nature,
while others are topological. This enhances the philosophy that our new notion is
really `between' the topological and Hausdorff dimensions.

Let $X$ be a metric space.
Since $\dim_{t} X = -1 \iff X = \emptyset \iff \dim_{H} X = -1$, we easily
obtain

\begin{fact} \label{<=>}
$\dim_{tH} X=0 \Longleftrightarrow \dim_{t} X=0$.
\end{fact}

As $\dim_H X$ is either $-1$ or at least $0$, we obtain

\begin{fact}\label{jump}
The topological Hausdorff dimension of a non-empty space is either $0$ or at least $1$.
\end{fact}

These two facts easily yield

\begin{corollary}\label{conn}
Every metric space with a non-trivial connected component
has topological Hausdorff dimension at least one.
\end{corollary}

The next theorem states that the topological Hausdorff dimension is
 between the topological and the Hausdorff dimension.

\begin{theorem}\label{<}
For every metric space $X$
\[
\dim_{t}X\leq \dim_{tH} X \leq
\dim _{H} X.
\]
\end{theorem}

\begin{proof}
We can clearly assume that $X$ is non-empty. It is well-known that $\dim_{t} X
\leq \dim_{H} X$ (see e.g. \cite{HW}), which easily implies
$\dim_{t} X \leq \dim_{tH} X$ using the definitions.
The second inequality is obvious if $\dim_{H}
X=\infty$. If $\dim_{H} X<1$ then $\dim_{t} X=0$ (since $\dim_{t} X \leq \dim_{H} X$ and $\dim_{t} X$ only takes integer values) and by Fact~\ref{<=>} we obtain $\dim _{tH} X=0$, hence the second inequality holds.
Therefore we may assume that $1\leq \dim_{H}
 X < \infty$. The following lemma is basically \cite[Thm.~7.7]{Ma}.
 It is only stated there in the special case $X = A \subseteq \mathbb{R}^n$, but the proof works verbatim for all metric spaces $X$.

\begin{lemma} \label{lip}
Let $X$ be a metric space and $f\colon X \to \mathbb{R}^{m}$ be
Lipschitz. If $s\geq m$, then
 \begin{equation} \int ^{\star} \mathcal{H}^{s-m}\left(f^{-1}(y)\right)\, \mathrm{d} \mathcal{H}^{m}(y)\leq c(m)
 \Lip(f)^{m} \mathcal{H}^{s}(X), \end{equation}
where $\int ^{\star}$ denotes the upper integral and $c(m)$ is a
finite constant depending only on $m$.
\end{lemma}
Now we return to the proof of Theorem \ref{<}. We fix $x_{0}\in X$
and define $f\colon X\to \mathbb{R}$ by $f(x)=d_{X}(x,x_{0})$. Using
the triangle inequality it is easy to see that $f$ is Lipschitz with
$\Lip(f)\leq 1$. We fix $n\in \mathbb{N}^{+}$ and apply Lemma~\ref{lip} for $f$ and $s=\dim_{H} X+\frac 1n>1 = m$. Hence
$$\int ^{\star} \mathcal{H}^{s-1}(f^{-1}(y))\, \mathrm{d} \mathcal{H}^{1}(y)\leq c(1) \mathcal{H}^{s}(X)=0.$$
 Thus $\mathcal{H}^{s-1}(f^{-1}(y))=\mathcal{H}^{\dim_{H} X+\frac
 1n-1}(f^{-1}(y))=0$ holds for a.e. $y\in \mathbb{R}$.
 Since this is true for all $n\in \mathbb{N}^{+}$, we obtain that
  $\dim_{H} f^{-1}(y) \leq \dim_{H} X-1$ for a.e.
$y\in \mathbb{R}$. From the definition of $f$ it follows that
$\partial U(x_{0},y) \subseteq f^{-1}(y)$. Hence there is a
neighborhood basis of $x_{0}$ with boundaries of Hausdorff dimension
at most $\dim_{H} X-1$, and this is true for all $x_{0}\in X$, so
there is a basis with
 boundaries of Hausdorff dimension at most $\dim_{H}X-1$. By
 the definition of the topological Hausdorff dimension this implies $\dim_{tH}
 X\leq \dim_{H} X$.
\end{proof}

\bigskip

There are some elementary properties one expects from a notion of dimension.
Now we verify some of these for the topological Hausdorff dimension.

\bigskip

\textbf{Extension of the classical dimension.}  Theorem \ref{<}
implies that the topological Hausdorff dimension of a countable set
equals zero, moreover, for open subspaces of $\mathbb{R}^{d}$ and
for smooth $d$-dimensional manifolds the topological Hausdorff
dimension equals $d$.

\textbf{Monotonicity.} Let $X \subseteq Y$. If $\mathcal{U}$ is a
basis in $Y$ then $\mathcal{U}_{X}=\{U\cap X: U\in \mathcal{U}\}$ is
a basis in $X$, and $\partial _{X} (U\cap X)\subseteq \partial_{Y}
U$ holds for all $U\in \mathcal{U}$. This yields

\begin{fact}[Monotonicity] \label{mon}
If $X\subseteq Y$ are metric
spaces then $\dim_{tH} X \leq \dim_{tH} Y$.
\end{fact}

\textbf{Bi-Lipschitz invariance.}  First we prove that the topological Hausdorff
dimension
does not increase under Lipschitz homeomorphisms. An easy
consequence of this that our dimension is bi-Lipschitz invariant, and
does not increase under an injective Lipschitz map on a compact
space. After obtaining corollaries of Theorem \ref{lipinv} we give some
examples illustrating the necessity of  certain conditions in this theorem
and its corollaries.

\begin{theorem} \label{lipinv}
Let $X,Y$ be metric spaces. If  $f\colon X\to Y$ is
a Lipschitz homeomorphism then $\dim_{tH}Y\leq \dim_{tH}X$.
\end{theorem}

\begin{proof} Since $f$ is a homeomorphism, if $\mathcal{U}$ is a basis in $X$ then
$\mathcal{V}=\{f(U): U\in \mathcal{U}\}$ is a basis in $Y$, and
 $\partial f(U)=f(\partial U)$ for all $U\in \mathcal{U}$.
The Lipschitz property of $f$ implies that $\dim_{H} \partial
V=\dim_{H}\partial f(U)=\dim_{H}f(\partial U)\leq \dim_{H} \partial
U$ for all $V=f(U)\in \mathcal{V}$.  Thus $\dim_{tH} Y\leq \dim_{tH}
X$.
\end{proof}

Note that every bi-Lipschitz map and each injective continuous map defined on a compact space is a homeomorphism.
Therefore we obtain the following corollaries.

\begin{corollary}[Bi-Lipschitz invariance]
\label{c:bi-Lip}
Let $X,Y$ be metric
spaces. If $f\colon X\to Y$ is a bi-Lipschitz onto map, then
$\dim_{tH}X=\dim_{tH}Y$. \end{corollary}

\begin{corollary} \label{inj} If $K$ is a compact metric space, and $f\colon K\to Y$
is \emph{one-to-one} Lipschitz then $\dim_{tH}f(K)\leq
\dim_{tH}K$. \end{corollary}

 The following example shows that we cannot drop injectivity here. First
 we need a well-known lemma.

\begin{lemma} \label{lipsurj}
Let $M\subseteq \mathbb{R}$ be measurable with positive Lebesgue
measure. Then there exists a Lipschitz onto map $f\colon M\to
[0,1]$.
\end{lemma}

\begin{proof} Let us choose a compact set $C \subseteq M$ of positive Lebesgue
  measure. Define $f\colon
M\to [0,1]$ by
$$f(x)=\frac{\lambda\left((-\infty,x)\cap C\right)}{\lambda (C)},$$
where $\lambda$ denotes the one-dimensional Lebesgue measure. Then it is not difficult to see that $f$ is Lipschitz
(with $\Lip(f) \le \frac{1}{\lambda (C)}$) and $f(C) = f(M) = [0,1]$.
\end{proof}

\begin{example} \label{ex2.15} Let $K\subseteq \mathbb{R}$ be a Cantor set (that is,
  a set homeomorphic to the middle-thirds Cantor set) of
positive Lebesgue measure. By Fact~\ref{<=>}, $\dim_{tH}
 K =\dim_{t} K =0$. Using Lemma~\ref{lipsurj} there is a Lipschitz
map $f\colon K\to [0,1]$ such that $f(K)=[0,1]$. By Theorem
\ref{<}, $\dim_{tH}[0,1]=1$, hence
$\dim_{tH} K =0<1=\dim_{tH}[0,1]=\dim_{tH} f(K)$.
 \end{example}

The next example shows that Corollary~\ref{inj} does not hold without
the assumption of compactness. We even have a separable metric
counterexample.

\begin{example} \label{ex2.16} Let $C$ be the middle-thirds Cantor set, and $f\colon
C\times C\to [0,2]$ be defined by $f(x,y)=x+y$. It is well-known and
easy to see that
$f$ is Lipschitz and $f(C\times C)=[0,2]$. Therefore one can
select
 a subset $X\subseteq C\times C$ such that
$f|_{X}$ is a bijection from $X$ onto $[0,2]$.
Then $X$ is separable metric. Monotonicity and $\dim_{t}
 (C\times C) =0$ imply $\dim_{tH}X \leq \dim_{tH}  (C\times C) =0$.
Therefore, $f$ is one-to-one and Lipschitz on $X$ but $\dim_{tH}X=0<1=\dim_{tH}[0,2]=\dim_{tH} f(X)$.
\end{example}

Our last example shows that the topological Hausdorff dimension is not
invariant under homeomorphisms. Not even for compact metric spaces.

\begin{example} Let $C_1, C_2 \subseteq \mathbb{R}$ be Cantor sets such that $\dim_H
  C_1 \neq \dim_H C_2$. We will see in Theorem \ref{prod} that $\dim_{tH} (C_i
  \times [0,1]) = \dim_H C_i + 1$ for $i=1,2$. Hence $C_1 \times [0,1]$ and
  $C_2 \times [0,1]$ are homeomorphic compact metric spaces whose topological
  Hausdorff dimensions disagree.
 \end{example}

\textbf{Stability and countable stability.}
As the following  example shows, similarly to
the case of topological dimension, stability does not hold for non-closed sets.
That is, $X=\bigcup_{n=1}^{k} X_{n}$
does not imply $\dim_{tH}X=\max_{1 \leq n \leq k}\dim_{tH}X_{n}.$

\begin{example}
Theorem \ref{<} implies $\dim_{tH} \mathbb{R}=1$, and Fact~\ref{<=>} yields $\dim_{tH} \mathbb{Q}=\dim_{t} \mathbb{Q}=0$ and
$\dim_{tH}(\mathbb{R}\setminus
\mathbb{Q})=\dim_{t}(\mathbb{R}\setminus \mathbb{Q})=0$. Thus
$\dim_{tH}
\mathbb{R}=1>0=\max\{\dim_{tH} \mathbb{Q},\dim_{tH}(\mathbb{R}\setminus
\mathbb{Q})\}$, and therefore stability fails.
\end{example}

As a corollary, we now show that as opposed to the case of Hausdorff (and
packing) dimension,
there is no reasonable family of measures inducing the topological Hausdorff
dimension. Let us say that a 1-parameter family of measures $\{ \mu^s \}_{s \ge
  0}$ is \emph{monotone} if $\mu^s(A) = 0$, $s<t$ implies $\mu^t(A) = 0$.
The family of Hausdorff (or packing) measures certainly satisfies this
criterion. It is not difficult to see that monotonicity implies that the induced
notion of dimension, that is,
$\dim A = \inf\{ s : \mu^s (A) =0 \}$ is countably stable. Hence we obtain

\begin{corollary}
There is no monotone 1-parameter family of measures $\{ \mu^s \}_{s \ge 0}$
such that $\dim_{tH} A = \inf\{ s : \mu^s (A) =0 \}$.
\end{corollary}

However, just like in the case of topological dimension, even countable
stability holds for \emph{closed} sets.

\begin{theorem}[Countable stability for closed sets] Let $X$ be a separable metric
space and $X=\bigcup_{n\in \mathbb{N}} \ X_{n}$, where $X_{n}$
$(n\in \mathbb{N})$ are closed subsets of $X$. Then $\dim_{tH}
X=\sup_{n\in \mathbb{N}} \dim_{tH}X_{n}$.
\end{theorem}

\begin{proof} Monotonicity clearly implies $\dim_{tH}X \geq \sup_{n\in \mathbb{N}} \dim_{tH} X_{n}$.

For the opposite inequality let $d=\sup_{n\in \mathbb{N}}\dim_{tH}X_{n}$, we may assume $d<\infty$. Theorem \ref{t:tHdecomp} yields that there are
sets $A_n\subseteq X_n$ such that $\dim_H A_n\leq d-1$ and $\dim_t (X_n\setminus A_n)\leq 0$. We may assume that the sets $A_n$ are $G_{\delta}$,
since we can take $G_{\delta}$ hulls with the same Hausdorff dimension. Let $A=\bigcup_{n=0}^{\infty} A_n\subseteq X$,
the countable stability of the Hausdorff dimension implies $\dim_H A\leq d-1$. As the sets $X_n\setminus A_n$ are $F_\sigma$ in $X$ with $\dim_t(X_n\setminus A_n) \leq 0$ and $X\setminus A\subseteq \bigcup_{n=0}^{\infty} (X_n\setminus A_n)$,
monotonicity and countable stability of topological dimension zero for $F_{\sigma}$ sets \cite[1.3.3.~Corollary]{Eng} yield $\dim_t (X\setminus A)\leq 0$.
Finally, $\dim_H A\leq d-1$ and $\dim_t (X\setminus A)\leq 0$ together with Theorem \ref{t:tHdecomp} imply $\dim_{tH} X\leq d$, and the proof is complete.
\end{proof}

\begin{corollary}
Countable stability holds for $F_{\sigma}$ sets, as well.
\end{corollary}

\textbf{Regularity.} Next we investigate the existence of $G_{\delta}$ hulls
and $F_{\sigma}$ subsets with the same topological Hausdorff dimension.

\begin{theorem}[Enlargement theorem for topological Hausdorff dimension] If
$X$ is a metric space and $Y\subseteq X$ is a separable subspace then there exists
a $G_{\delta}$ set $G\subseteq X$ such that $Y\subseteq G$ and $\dim_{tH}
G=\dim_{tH} Y$.
\end{theorem}

\begin{proof} We may assume $Y\neq \emptyset$. Theorem \ref{t:tHdecomp}
implies that there exists $Z\subseteq Y$ such that $\dim_H Z=\dim_{tH} Y-1$
and $\dim_{t}(Y\setminus Z)\leq 0$. Let $A\subseteq X$ be a $G_{\delta}$ set
such that $Z\subseteq A$ and $\dim_H A=\dim_H Z=\dim_{tH} Y-1$. By
\cite[1.2.14.]{Eng} there exists a
$G_{\delta}$ set $B\subseteq X$ such that $Y\setminus Z\subseteq B$ and
$\dim_{t} B\leq 0$.  Let $G = A\cup B$, then $G$ is a
$G_{\delta}$ subset of $X$ with $Y \subseteq G$. Since $\dim_{t} (G\setminus
A)\leq \dim_t B\leq 0$, Theorem \ref{t:tHdecomp} implies that $\dim_{tH} G\leq
\dim_{tH} A+1=\dim_{tH} Y$, and monotonicity yields $\dim_{tH} G=\dim_{tH} Y$.
\end{proof}

The following example shows that inner regularity does not hold even for
$G_{\delta}$ subsets of Euclidean spaces.

\begin{example}\label{ex:Maz} For every $d\in \mathbb{N}^+$ Mazurkiewicz \cite{Maz}
constructed a $G_{\delta}$ set $G\subseteq [0,1]^{d+1}$ such that $\dim_{t}
G=d$ and $G$ is totally disconnected
(see \cite[Thm.~3.9.3.]{Mi} for a proof in English).
Theorem \ref{<} implies that $\dim_{tH}G\geq
\dim_t G=d$. If $F\subseteq G$ is closed in $\mathbb{R}^{d+1}$ then $F$ is compact
and totally disconnected, thus $\dim_t F\leq 0$, so Fact~\ref{<=>} implies
$\dim_{tH}F\leq 0$. Therefore countable stability of the topological
Hausdorff dimension for closed sets yields that every $F_{\sigma}$ subset of
$G$ has topological Hausdorff dimension at most $0$.
\end{example}

\textbf{Products.}
Now we investigate products from the point of view of topological
Hausdorff dimension. By the product of two metric spaces we will always mean the
$l^2$-product, that is,
\[
d_{X\times Y}((x_{1},y_{1}),(x_{2},y_{2}))=
\sqrt{d^{2}_{X}(x_1,x_2)+d^{2}_{Y}(y_1,y_2)}.
\]
First we recall a well-known statement, see \cite[Chapter~3]{F} and \cite[Product formula~7.3]{F} for
the definition of the upper box-counting dimension and the proof, respectively. In fact, \cite{F} works
in Euclidean spaces only, but the proof goes through verbatim to general metric spaces.

\begin{lemma} \label{box}
Let $X,Y$ be non-empty metric spaces such that
$\dim_{H}Y=\overline{\dim}_{B} Y$, where $\overline{\dim}_{B}$ is
the upper box-counting dimension. Then
$$\dim_{H}(X\times Y) \leq \dim_{H}X+\dim_{H}Y.$$
\end{lemma}

Now we prove our next theorem which provides a large class of sets for which the
topological Hausdorff dimension and the Hausdorff dimension coincide.

\begin{theorem} \label{prod} Let $X$ be a non-empty separable metric space. Then
$$\dim_{tH}\left(X \times [0,1]\right)=\dim_{H} \left(X\times [0,1]\right)=\dim_{H} X+1.$$
\end{theorem}

\begin{proof}
From Theorem \ref{<} it follows that $\dim_{tH} \left(X\times
[0,1]\right) \leq \dim_{H} \left(X\times [0,1]\right)$.

Applying  Lemma~\ref{box} for $Y=[0,1]$ we deduce that
$$\dim_{H} \left(X\times [0,1]\right)\leq \dim_{H} X+\dim_{H} [0,1]=\dim_{H} X+1.$$

Finally, we prove that $\dim_H X+1\leq \dim_{tH} \left(X\times [0,1]\right)$. Let us define $\pr_X\colon
X\times [0,1]\to X$ as $\pr_X (x,y)=x$ and let $Z=X\times [0,1]$. Theorem \ref{t:tHdecomp} implies that there is a set $A\subseteq Z$ such that $\dim_H A\leq \dim_{tH}Z-1$ and $\dim_t (Z\setminus A)\leq 0$. Since $Z\setminus A$ is totally disconnected, $A$ intersects $\{x\}\times [0,1]$ for all $x\in X$, thus $\pr_X (A)=X$.
Projections do not increase the Hausdorff dimension, thus
$$\dim_{tH}Z-1\geq \dim_H A\geq \dim_H \pr_X (A)=\dim_H X.$$
Hence $\dim_{tH} \left(X\times [0,1]\right) \geq \dim_H X+1$, and the proof is complete.
\end{proof}

\begin{remark}
We cannot drop separability here.
Indeed, if $X$ is an uncountable discrete metric space then it is not difficult to
see that $\dim_{tH} (X\times [0,1])=1$ and $\dim_{H} (X\times
[0,1])=\dim_{H} X=\infty$.

Separability is a rather natural assumption throughout the paper. First, the
Hausdorff dimension is only meaningful in this context (it is always infinite
for non-separable spaces), secondly for the theory of topological dimension
this is the most usual framework.
\end{remark}

\begin{corollary} \label{prodn} If $X$ is a non-empty  separable metric space
  then
\[
\dim_{tH} \left(X\times [0,1]^{d}\right)=\dim_{H} \left(X\times
  [0,1]^{d}\right)=\dim_{H} X+d.
\]
\end{corollary}

\textbf{The possible values of $(\dim_t X, \dim_{tH} X, \dim_H X)$.}
Now we provide a complete description of the possible values of
the triple $(\dim_t X, \dim_{tH} X, \dim_H X)$. Moreover, all possible values
can be realized by compact spaces as well.

\begin{theorem} \label{trip} For a triple $(d,s,t) \in [0, \infty]^3$ the
  following are equivalent.
\begin{enumerate}[(i)]
\item
There exists a compact metric space $K$ such that $\dim_t K = d$, $\dim_{tH} K
= s$, and $\dim_H K = t$.
\item
There exists a separable metric space $X$ such that $\dim_t X = d$, $\dim_{tH} X
= s$, and $\dim_H X = t$.
\item
There exists a metric space $X$ such that $\dim_t X = d$, $\dim_{tH} X
= s$, and $\dim_H X = t$.
\item
$d=s=t=-1$, or $d=s=0$, $t \in [0, \infty]$, or $d\in \mathbb{N}^{+}
\cup \{ \infty \}, \ s,t \in [1, \infty], \ d\leq s \leq t$.
\end{enumerate}
\end{theorem}

\begin{proof} The implications
$(i) \implies (ii)$ and $(ii) \implies (iii)$ are obvious, and $(iii) \implies (iv)$ can easily be
  checked using Fact~\ref{<=>} and Theorem \ref{<}.

It remains to prove that $(iv) \implies (i)$.
First, the empty set takes care of the case $d=s=t=-1$. Let now $d=s=0$, $t \in [0,
  \infty]$.
For $t \in [0,\infty]$ let $K_{t}$ be a Cantor set with $\dim_{H}
K_{t}=t$. Such sets are well-known to exist already in $[0,1]^n$ for
large enough $n$ in case $t< \infty$, whereas if $C$ is the
middle-thirds Cantor set then $C^\mathbb{N}$ is such a set for
$t=\infty$. Then clearly $\dim_{t} K_{t} =\dim_{tH} K_{t} =0$ and
$\dim_{H} K_{t}=t$, so we are done with this case.

Finally, let $d\in \mathbb{N}^{+} \cup \{ \infty \}, \ s,t \in [1,
\infty], \ d\leq s \leq t$. We may assume $d < \infty$, otherwise
the Hilbert cube provides a suitable example. (Indeed, clearly
$\dim_{t} [0,1]^\mathbb{N} = \dim_{tH} [0,1]^\mathbb{N} = \dim_H
[0,1]^\mathbb{N} = \infty$.) Define $K_{d, s
,t}=(K_{s-d}\times[0,1]^{d})\cup K_{t}$ (this can be understood as
the disjoint sum of metric spaces, but we may also assume that all
these spaces are in the Hilbert cube, so the union is well defined).
Since $\dim_t (X\times Y) \le \dim_t X + \dim_t Y$ for non-empty
spaces (see e.g. \cite{Eng}), we obtain $\dim_t
(K_{s-d}\times[0,1]^{d}) =0+d = d$. Hence, by the stability of the
topological dimension for closed sets, $\dim_{t} K_{d,s,t} = \max
\left\{ \dim_{t} \left(K_{s-d} \times [0,1]^{d} \right) , \dim_{t}
K_{t} \right\} = \max\{d,0\}=d$. Using Corollary~\ref{prodn} and the
stability of the topological Hausdorff dimension for closed sets we
infer that
$\dim_{tH}K_{d,s,t}=\max\left\{\dim_{tH}\left(K_{s-d}\times
[0,1]^{d}\right),\dim_{tH} K_{t}\right\}=\max \{s-d+d,0\}=s$. Again
by Corollary~\ref{prodn} and by the stability of the Hausdorff
dimension we obtain that
$\dim_{H}K_{d,s,t}=\max\left\{\dim_{H}\left(K_{s-d}\times
[0,1]^{d}\right),\dim_{H} K_{t}\right\}=\max
\{s-d+d,t\}=\max\{s,t\}=t$. This completes the proof.
\end{proof}

\textbf{The topological Hausdorff dimension is not a function of the
topological and the Hausdorff dimension.}
As a particular case of the above theorem we obtain that there are
compact metric spaces $X$ and $Y$ such that $\dim_{t} X = \dim_{t}
Y$ and $\dim_{H} X = \dim_{H} Y$ but $\dim_{tH} X \neq \dim_{tH} Y$.
This immediately implies the following, which shows that the
topological Hausdorff dimension is indeed a genuinely new concept.

\begin{corollary} \label{funindep} $\dim_{tH} X$ cannot be
  calculated from $\dim_{t} X$ and $\dim_{H} X$, even for compact metric spaces.
\end{corollary}

\section{Calculating the topological Hausdorff dimension}\label{s:examples}

\subsection{Some classical fractals}

First we present certain natural examples of compact sets $K$ with
$\dim_t K = \dim_{tH} K < \dim_H K$. Let $S$ be the Sierpi\'nski
triangle, then it is well-known that $\dim_t S = 1$ and $\dim_H S
=\frac{\log 3}{\log 2}$.

 \begin{theorem} \label{sierhth} Let $S$ be the Sierpi\'nski triangle. Then
   $\dim_{tH}(S)=1$.
\end{theorem}

\begin{proof} Let $\varphi _{i}\colon  \mathbb{R}^{2}\rightarrow \mathbb{R}^{2}$ ($i=1,2,3$) be
the three similitudes with ratio $1/2$ for which $S=\bigcup
_{i=1}^{3} \varphi _{i} (S)$. Sets of the form $\varphi _{i_{n}}
\circ \dots \circ \varphi _{i_{1}}(S)$, $ n\in \mathbb{N}, \,  j \in
\{1,\dots,n\}, \,  i_{j} \in \{1,2,3\}$ are called the elementary
pieces of $S$. It is not difficult to see that
\[
\mathcal{U} = \left\{ \inter_S H : H \textrm{ is a finite union of
elementary pieces of } S \right\}
\]
is a basis of $S$ such that $\#\partial_S U$ is finite for every
$U\in \mathcal{U}$.  Therefore $\dim_{H} \partial_S U \leq 0$, and
hence
 $\dim_{tH} S\leq 1$. On the other hand, $S$ contains a line
segment, therefore $\dim_{tH} S\geq \dim_{tH}[0,1] = 1$ by monotonicity.
 \end{proof}

Now we turn to the von Koch snowflake curve $D$. Recall that $\dim_t D = 1$
and $\dim_H D=\frac{\log 4}{\log 3}$.

\begin{fact} \label{interval} If $K$ is homeomorphic to $[0,1]$ then $\dim
_{tH} K=1$. \end{fact}

\begin{proof} By  Corollary~\ref{conn} we obtain
that $\dim_{tH} K \geq 1$. On the other hand, since $K$ is
homeomorphic to $[0,1]$, there is a basis in $K$ such that
$\#\partial U\leq 2$ for every $U\in \mathcal{U}$. Thus $\dim_{tH} K
\leq 1$.
\end{proof}

\begin{corollary}
Let $D$ be the von Koch curve. Then $\dim_{tH} D= 1$.
\end{corollary}

Next we take up a natural example of a compact
set $K$ with $\dim_t K < \dim_{tH} K < \dim_H K$. Let $T$ be the Sierpi\'nski carpet,
then it is well-known that $\dim_t T = 1$ and $\dim_H T
=\frac{\log 8}{\log 3}$.

\begin{theorem} \label{sierszth}
Let $T$ be the Sierpi\'nski carpet. Then $\dim_{tH} T = \frac{\log 2}{\log 3}
+ 1 = \frac{\log 6}{\log 3}$.
\end{theorem}

\begin{proof}
Let $C$ denote the middle-thirds Cantor set. Observe that $C \times [0,1]
\subseteq T$. Then monotonicity and Theorem
\ref{prod} yield $\dim_{tH} T\geq \dim_{tH} (C\times [0,1]) = \dim_{H} C + 1
 = \frac{\log 2}{\log 3} + 1$.

Let us now prove the opposite inequality. For $n\in \mathbb{N}$ and
$i=1,\dots, 3^{n}$ let $z_{i}^{n}=\frac{2i-1}{2(3^{n})}$. Then
clearly
\[
\left\{z_{i}^{n} : n \in \mathbb{N}, \, i \in
\left\{1,\dots,3^{n}\right\} \right\}
\]
is dense in $[0,1]$. Let $L$ be a horizontal line defined by an
equation of the form $y = z_i^n$ or a vertical line defined by $x =
z_i^n$. It is easy to see that $L \cap T$  consists of finitely many
sets geometrically similar to the middle-thirds Cantor set. Using
these lines it is not difficult to construct a rectangular basis
$\mathcal{U}$ of $T$ such that $\dim_{H} \partial_T U = \frac{\log
2}{\log 3}$ for every $U \in \mathcal{U}$, and hence $\dim_{tH}
T\leq \frac{\log 2}{\log 3} + 1$.
\end{proof}

Finally we remark that, by Theorem \ref{prod}, $K = C \times [0,1]$ (where $C$
is the middle-thirds Cantor set) is a natural example of a compact
set with $\dim_t K < \dim_{tH} K = \dim_H K$.

\subsection{Kakeya sets}

\begin{definition}
A subset of $\mathbb{R}^{d}$ is called a \emph{Kakeya set} if it
contains a non-degenerate line segment in every direction (some
authors call these sets \emph{Besicovitch sets}).
\end{definition}

According to a surprising classical result, Kakeya sets of Lebesgue
measure zero exist. However, one of the most famous conjectures in
analysis is the Kakeya Conjecture stating that every Kakeya set in
$\mathbb{R}^{d}$ has Hausdorff dimension $d$. This is known to hold
only in dimension at most 2 so far, and a solution already in
$\mathbb{R}^3$ would have a huge impact on numerous areas of
mathematics.

It would be tempting to attack the Kakeya Conjecture using $\dim_{tH} K \le \dim_H K$, but the following theorem,
the main theorem of this subsection will show that
unfortunately we cannot get anything non-trivial this way.

\begin{theorem}
\label{tkak} There exists a Kakeya set $K\subseteq \mathbb{R}^{d}$
of topological Hausdorff dimension 1 for every integer $d \ge 1$.
\end{theorem}

This result is of course sharp, since if a set contains a line segment then its topological Hausdorff dimension is at least 1.

We will actually prove somewhat more, since we will essentially show that the generic element of a
carefully chosen space is a Kakeya set of topological Hausdorff dimension $1$. This idea, as well
as most of the others in this subsection are already present in \cite{Korner} by T. W. K\"orner.
However, he only works in the plane and his space slightly differs from ours. For the sake of completeness we provide the rather short proof
in detail.

\bigskip

Let $(\mathcal{K},d_{H})$ be the set of non-empty compact subsets of
$\mathbb{R}^{d-1}\times [0,1]$ endowed with the Hausdorff metric,
that is, for each $K_1,K_2\in \mathcal{K}$
$$d_{H}(K_1,K_2)=\min \left\{r: K_1\subseteq B(K_2,r) \textrm{ and } K_2\subseteq B(K_1,r)\right\}.$$
It is well-known that
$(\mathcal{K},d_{H})$ is a complete metric space, see
e.g.~\cite{Ke}.

Let
$$\Gamma=\left\{(x_{1},\dots ,x_{d-1},1): 1/2 \leq x_i \leq 1, \quad i=1,\dots ,d-1\right\}$$
denote a subset of directions in $\mathbb{R}^{d}$.
A closed line segment $w$ connecting $\mathbb{R}^{d-1}\times \{0\}$
and $\mathbb{R}^{d-1}\times \{1\}$ is called a standard segment.

Let us denote by $\mathcal{F}\subseteq \mathcal{K}$ the system of
those compact sets in $\mathbb{R}^{d-1}\times [0,1]$ in which for
each $v\in \Gamma$ we can find a standard segment $w$ parallel to
$v$.
First we show that $\mathcal{F}$ is closed in $\mathcal{K}$. Let us
assume that $F_{n}\in \mathcal{F}$, $K\in \mathcal{K}$ and $F_{n}\to
K$ with respect to $d_H$. We have to show that $K \in \mathcal{F}$.
Let $v \in \Gamma$ be arbitrary. Since $F_{n}\in \mathcal{F}$, there
exists a $w_{n}\subseteq F_{n}$ parallel to $v$ for every $n$. It is
easy to see that $\bigcup_{n\in \mathbb{N}}F_{n}$ is bounded, hence
we can choose a subsequence $n_k$ such that $w_{n_k}$ is convergent
with respect to $d_H$. But then clearly $w_{n_k} \to w$ for some
standard segment $w\subseteq K$, and $w$ is parallel to $v$. Hence
$K \in \mathcal{F}$ indeed.

Therefore, $(\mathcal{F},d_{H})$ is a complete metric space and
hence we can use Baire category arguments.

The next lemma is based on \cite[Thm.~3.6.]{Korner}.

\begin{lemma}
\label{lkak} The generic set in $(\mathcal{F},d_H)$ is of
topological Hausdorff dimension 1.
\end{lemma}

\begin{proof} The rational cubes form a basis of $\mathbb{R}^d$, and their boundaries are covered by the rational
hyperplanes orthogonal to one of the usual basis vectors of
$\mathbb{R}^d$. Therefore, it suffices to show that if $S$ is a
fixed hyperplane orthogonal to one of the usual basis vectors then
$\{ F \in \mathcal{F} : \dim_H (F \cap S) = 0 \}$ is co-meager.

For $n\in \mathbb{N}^{+}$ define
\[
\mathcal{F}_{n}=\left\{F\in \mathcal{F}: \mathcal{H}^{\frac
1n}_{\frac 1n}(F\cap S)<\frac 1n\right\}.
\]
In order to show that $\{ F \in \mathcal{F} : \dim_H (F \cap S) = 0
\} = \bigcap_{n\in \mathbb{N}^{+}}\mathcal{F}_{n}$ is co-meager, it
is enough to prove that each $\mathcal{F}_{n}$ contains a dense open
set.

For $p\in \mathbb{R}^{d}$, $v\in \Gamma$ and $0<\alpha<\pi/2$ we
denote by $C(p,v,\alpha)$ the following doubly infinite closed cone
\[
C(p,v,\alpha) = \{ x \in \mathbb{R}^d : \textrm{ the angle between
the lines of $v$ and $x-p$ is at most } \alpha \}.
\]
We denote by $V(C(p,v,\alpha))$ the set of those vectors
$u=(u_{1},...,u_{d-1},1)$ for which there is a line in
$\inter(C(p,v,\alpha))\cup \{ p \}$ parallel to $u$. Then
$V(C(p,v,\alpha))$ is relatively open in $\mathbb{R}^{d-1}\times \{
1 \}$.

The sets of the form $C'(p,v,\alpha)= C(p,v,\alpha)\cap
\left(\mathbb{R}^{d-1}\times [0,1]\right)$ will be called truncated
cones, and the system of truncated cones will be denoted by
${\mathcal{C}}'$.
A truncated cone $C'(p,v,\alpha)$ is $S$-compatible if either
$C'(p,v,\alpha)\cap S=\{ p \}$, or $C'(p,v,\alpha)\cap S=\emptyset.$
The set of $S$-compatible truncated cones is denoted by ${\mathcal{
C}}'_S$. Define ${\mathcal{F}}_{S}$ as the set of those $F\in
\mathcal{F}$ that can be written as the union of finitely many
$S$-compatible truncated cones and  finitely many points in
$\mathbb{R}^{d-1}\times [0,1]$.

Next we check that ${\mathcal{F}}_{S}$ is dense in $\mathcal{F}$.

Suppose $F\in \mathcal{F}$ is arbitrary and $\varepsilon >0$ is
given. First choose finitely many points $\{ y_{i} \}_{i=1}^t$ in
$F$ such that $F\subseteq B\left(\{ y_{i} \}_{i=1}^t,\varepsilon
\right)$.
Let $v \in \Gamma$ be arbitrary, then there exists a standard
segment $w_v \subseteq F$ parallel to $v$. By the choice of $S$ and
$\Gamma$, clearly $w_v \nsubseteq S$, hence we can choose $p_v$ and
$\alpha_v$ such that $C'(p_v,v,\alpha_v)\in {\mathcal{C}}'_S$ and
$d_H(C'(p_v,v,\alpha_v), w_v) \le \varepsilon$. Obviously $v\in
V(C(p_v,v,\alpha_v))$, so $\{ V(C(p_v, v, \alpha_v)) \}_{v\in
\Gamma}$ is an open cover of the compact set $\Gamma$. Therefore,
there are $\{ C'(p_{v_i},v_{i},\alpha_{v_i}) \}_{i=1}^m$ in
${\mathcal{C}}'_S$ such that $ \Gamma \subseteq \bigcup_{i=1}^{m}
V(C(p_{v_i},v_{i},\alpha_{v_i}))$. Put
$F'=\bigcup_{i=1}^{m}C'(p_{v_i},v_{i},\alpha_{v_i})\cup \{ y_{1},
\dots, y_t \}$, then $F' \in {\mathcal{F}}_{S}$. It is easy to see
that $\bigcup_{i=1}^{m}C'(p_{v_i},v_{i},\alpha_{v_i})\subseteq
B(F,\varepsilon)$, and combining this with $\{ y_i \}_{i=1}^t
\subseteq F$ we obtain that $F'\subseteq B(F,\varepsilon)$. By the
choice of $\{ y_i \}_{i=1}^t$ we also have $F\subseteq
B(F',\varepsilon)$. Thus $d_{H}(F,F') \le \varepsilon$.

Now using our dense set ${\mathcal{F}}_{S}$ we verify that
${\mathcal{F}}_{n}$ contains a dense open set $\mathcal{U}$. We
construct for all $F_{0}\in {\mathcal{F}}_{S}$ a ball in
$\mathcal{F}_{n}$ centered at $F_{0}$. By the definition of
$S$-compatibility $F_{0}\cap S$ is finite. Hence we can easily
choose a relatively open set $U_0\subseteq S$ such that $F_{0}\cap S
\subseteq U_0$ and $\mathcal{H}^{\frac 1n}_{\frac 1n}(U_0)<\frac
1n$. Let us define
$$\mathcal{U}= \left\{F\in \mathcal{F}: F\cap S\subseteq U_{0}\right\}.$$
Clearly $F_{0}\in \mathcal{U}$, $\mathcal{U}\subseteq
\mathcal{F}_{n}$ and it is easy to see that $\mathcal{U}$ is open in
$\mathcal{F}$. This completes the proof.
\end{proof}

From this we obtain the main theorem of the subsection as follows.

\begin{proof}[Proof of Theorem \ref{tkak}]
By the above lemma we can choose $F_{0}\in \mathcal{F}$ such that
$\dim_{tH} F_{0}=1$. Then $F_{0}$ contains a line segment in every
direction of $\Gamma$, hence we can choose finitely many isometric
copies of it, $\{F_i\}_{i=1}^n$ such that the compact set
$K=\bigcup_{i=0}^{n}F_{i}$ contains a line segment in every
direction. By the Lipschitz invariance of the topological Hausdorff
dimension $\dim_{tH} F_{i} = \dim_{tH} F_{0}$ for all $i$, and by
the stability of the topological Hausdorff dimension for closed sets
$\dim_{tH}K=1$.
\end{proof}

\subsection{Brownian motion}

One of the most important stochastic processes is the Brownian motion (see
e.g. \cite{MP}). Its range
and graph also serve as important examples of fractal sets in geometric
measure theory. Since the graph is always homeomorphic to $[0, \infty)$, Fact~\ref{interval} and countable stability for closed sets yield that its
topological Hausdorff dimension is 1. Hence we focus on the range only.

Each statement in this paragraph is to be understood to hold with probability
1 (almost surely). Clearly, in dimension 1 the range is a non-degenerate interval, so it has
topological Hausdorff dimension 1. Moreover, if the dimension is at least $4$
then the range has no multiple points (\cite{MP}), so it is homeomorphic to
$[0, \infty)$, which in turn implies as above
that the range has topological Hausdorff dimension 1
again.

\begin{problem}
\label{p:Brown}
Let $d = 2$ or $3$. Determine the almost sure topological Hausdorff
dimension of the range of the $d$-dimensional Brownian motion. Equivalently, determine the smallest $c \ge 0$ such that the range can be decomposed into a totally disconnected set and a set of Hausdorff dimension at most $c-1$ almost surely.
\end{problem}

Indeed, the two formulations of the problem are equivalent by Theorem \ref{t:tHdecomp2}.
The following open problem of W. Werner \cite[p. 384.]{MP} is closely related to Problem~\ref{p:Brown} in the case $d=2$.

\begin{notation}
By a \emph{curve} we mean a continuous map $\gamma \colon [0,1] \to \mathbb{R}^d$.
Let us denote by $\ran(\gamma)$ the \emph{range} of $\gamma$.
If $\gamma \colon [0,1] \to \mathbb{R}^2$ is a closed curve and $p\in \mathbb{R}^2\setminus \ran(\gamma)$
then let us denote by $WN(\gamma ,p)$ the \emph{winding number} of
$\gamma$ with respect to $p$, see \cite{Fu} for the definition.
\end{notation}

\begin{problem}[W. Werner] \label{p:Werner} Is it true almost surely for the planar Brownian motion $B\colon  [0,\infty) \to \mathbb{R}^2$
that for every $x,y\in \mathbb{R}^2 \setminus \ran(B)$ there exists a curve
$\gamma \colon [0,1]\to \mathbb{R}^2$ such that $\gamma(0)=x$, $\gamma(1)=y$, and $\ran(\gamma) \cap \ran(B)$ is finite?
\end{problem}

\begin{remark}
An affirmative answer to Problem \ref{p:Werner} would probably also solve
Problem \ref{p:Brown}. More precisely, if there exists any function $b \colon
(0, \infty) \to (0, \infty)$ with $\lim_{x\to 0+} b(x) = 0$ such that the above curve $\gamma$ can be constructed
in the disc $U(x, b(|x-y|))$ then one can build a basis
of $\ran (B)$ as follows. Let $x \in \ran (B)$ and $\varepsilon > 0$ be given, and pick
$\delta > 0$ such that $|x-y| < \delta$ implies $b(|x-y|) < \frac{\varepsilon}{2}$. Select a sequence of points $x_0, x_1, \dots, x_n = x_0$ on the circle of radius $\frac{\varepsilon}{2}$ centered at $x$ such that $x_i \notin \ran (B)$ and $|x_i - x_{i+1}| < \delta$ for every $i = 0, \dots, n-1$.
Moreover, we can also assume that the argument of the vectors $x_i - x$ $(i = 0, \dots, n-1)$ is an increasing sequence in $[0, 2\pi)$. This implies that if for every $i = 0, \dots, n-1$ we construct a curve $\gamma_i$ from $x_i$ to $x_{i+1}$ in $U(x_i, \frac{\varepsilon}{2})$ intersecting $\ran (B)$ finitely many times, and glue these curves together to obtain a closed curve $\gamma$ then the winding number $WN(\gamma, x) = 1$, and hence $x$ is in a bounded component $U_{x, \varepsilon}$ of $\mathbb{R}^2 \setminus \ran (\gamma)$ see \cite[Proposition~3.16.]{Fu}. Then it is easy to see that $\partial U_{x, \varepsilon} \subseteq \ran (\gamma)$, hence $\partial U_{x, \varepsilon} \cap \ran (B)$ is finite, and $x \in U_{x, \varepsilon} \subseteq U(x, \varepsilon)$. Therefore, the sets of the form $U_{x, \varepsilon} \cap \ran(B)$ form a basis of $\ran(B)$ such that the boundary of $U_{x, \varepsilon}$ relative to $\ran(B)$ is finite for all $x$ and $\varepsilon$, thus the topological Hausdorff dimension of $\ran (B)$ is $1$ almost surely.
\end{remark}

\bigskip

While we can solve neither Problem \ref{p:Brown} nor Problem \ref{p:Werner}, we are able to solve
their Baire category duals. First we need some preparation.

\begin{definition}
Let us denote by $\mathcal {C}^d$ the space of curves from $[0,1]$ to $\mathbb{R}^d$ endowed with
the supremum metric. As this is a complete metric space, we can use Baire category arguments.

A \emph{standard line segment} in $\mathbb{R}^d$ is a closed non-degenerate line segment parallel to one of the standard basis vectors
$e_1=(1,0,\dots,0),\dots, e_d=(0,0,\dots,1)$.
\end{definition}

The following theorem gives an affirmative answer to the category dual of Problem~\ref{p:Brown}.

\begin{theorem} The range of the generic $f\in \mathcal{C}^d$ has topological Hausdorff dimension $1$ for every $d\in \mathbb{N}^+$.
\end{theorem}

\begin{proof} If $d=1$ then the statement is straightforward, thus we may assume $d>1$.
By Corollary~\ref{conn} the range of every non-constant $f \in \mathcal{C}^d$ has topological Hausdorff dimension at least 1,
thus we need to show that the generic $f \in \mathcal{C}^d$
has topological Hausdorff dimension at most 1. The rational cubes form a basis of $\mathbb{R}^d$,
and their boundaries are covered by countably many hyperplanes,
therefore it is enough to show that if $S$ is a fixed hyperplane then $\{ f \in \mathcal{C}^d : \dim_H \left(\ran(f) \cap S\right) = 0 \}$
is co-meager. By rotating our coordinate system we may suppose that $S$ is not parallel to any vector in the standard basis $\{e_1,\dots, e_d\}$.
For $n\in \mathbb{N}^{+}$ define
$$\mathcal{F}_{n}=\left\{ f \in \mathcal{C}^d: \mathcal{H}^{\frac
1n}_{\frac 1n}\left(\ran(f) \cap S\right)<\frac 1n\right\}.$$
In order to show that $\{ f \in \mathcal{C}^d : \dim_H \left(\ran(f) \cap S\right) = 0\}= \bigcap_{n\in \mathbb{N}^{+}}\mathcal{F}_{n}$ is co-meager, it
is enough to prove that each $\mathcal{F}_{n}$ is a dense open
set. Let us fix $n \in \mathbb{N}^+$. The regularity of $\mathcal{H}^{\frac 1n}_{\frac 1n}$ implies that $\mathcal{F}_{n}$ is open.
Let $\mathcal{G}$ be the set of curves $f\in\mathcal{C}^d$ such that $\ran(f)$ is a union of finitely many standard line segments.
It is easy to see that $\mathcal{G}$ is dense in $\mathcal{C}^d$. As $\ran(f)\cap S$ is finite for every $f \in \mathcal{G}$, we have $\mathcal{G}\subseteq  \mathcal{F}_n$. Thus $\mathcal{F}_{n}$ is dense in $\mathcal{C}^d$, and the proof is complete.
\end{proof}

\begin{definition} Let $C \subseteq \mathbb{R}^2$ and let $p,q\in \mathbb{R}^2\setminus
C$. We say that $C$ \emph{separates $p$ and $q$} if these points belong to
different connected components of $\mathbb{R}^2\setminus C$.  A metric space is
called a \emph{continuum} if it is compact and connected.  A set $V\subseteq
\mathbb{R}^2$ is a \emph{standard open set} if it is a union of finitely many
axis-parallel open squares.

We say that $f\in \mathcal{C}^2$ is a \emph{standard curve} if there exist
$0=x_0<\dots<x_m=1$ such that $f$ maps the interval $I_k=[x_{k-1},x_k]$
bijectively to a standard line segment, and the standard line segments $f(I_k)$ and $f(I_{k+1})$ are orthogonal
for all $k\in \{1,\dots,m\}$, where we use the notation $I_{m+1} = I_1$.
Notice that the points $x_k$ are uniquely determined. We say
that $f(x_k)$ and $I_k$ are the \emph{turning points} and \emph{edge
intervals} of $f$, respectively.  A simple closed standard curve $\gamma\in
\mathcal{C}^2$ is a \emph{standard polygon}.

Let $\alpha>0$. Let us define $w\in \mathcal{C}^2$ to be an \emph{$\alpha$-wire} if
there exists a standard polygon $\gamma \in \mathcal{C}^2$ with edge intervals
$I_1,\dots,I_m$ such that the lengths of $\gamma(I_k)$ are at most $\alpha$,
and if $E_k$ denotes the segment we obtain from the edge $\gamma(I_k)$ by expanding it by length $\alpha$
in both directions, then $w(I_k)=E_k$ for all $k\in \{1,\dots,m\}$, see Figure \ref{fig:curve1}.
We say that the $E_k$ are the \emph{edges} of the $\alpha$-wire $w$, and $E_k$ and $E_{k+1}$
are \emph{adjacent} if $k\in \{1,\dots,m\}$, where  $E_{m+1}=E_1$.
Of course, $w$ passes through points of $w(I_{k})\setminus \gamma(I_{k})$ more than once.
\end{definition}

\placedrawing[ht]{
\begin{picture}(90,52)
\Thicklines
\drawpath{6.0}{24.0}{6.0}{20.0}
\drawpath{6.0}{20.0}{10.0}{20.0}
\drawpath{10.0}{20.0}{10.0}{14.0}
\drawpath{10.0}{14.0}{16.0}{14.0}
\drawpath{16.0}{14.0}{16.0}{12.0}
\drawpath{16.0}{12.0}{22.0}{12.0}
\drawpath{22.0}{12.0}{22.0}{18.0}
\drawpath{22.0}{18.0}{24.0}{18.0}
\drawpath{24.0}{18.0}{24.0}{24.0}
\drawpath{24.0}{24.0}{26.0}{24.0}
\drawpath{26.0}{24.0}{26.0}{26.0}
\drawpath{26.0}{26.0}{28.0}{26.0}
\drawpath{28.0}{26.0}{28.0}{28.0}
\drawpath{28.0}{28.0}{30.0}{28.0}
\drawpath{30.0}{28.0}{30.0}{32.0}
\drawpath{30.0}{32.0}{24.0}{32.0}
\drawpath{24.0}{32.0}{24.0}{36.0}
\drawpath{24.0}{36.0}{20.0}{36.0}
\drawpath{20.0}{36.0}{20.0}{40.0}
\drawpath{20.0}{40.0}{18.0}{40.0}
\drawpath{18.0}{40.0}{14.0}{40.0}
\drawpath{14.0}{40.0}{14.0}{38.0}
\drawpath{14.0}{38.0}{12.0}{38.0}
\drawpath{12.0}{38.0}{12.0}{36.0}
\drawpath{12.0}{36.0}{10.0}{36.0}
\drawpath{10.0}{36.0}{10.0}{32.0}
\drawpath{10.0}{32.0}{4.0}{32.0}
\drawpath{4.0}{32.0}{4.0}{28.0}
\drawpath{4.0}{28.0}{8.0}{28.0}
\drawpath{8.0}{28.0}{8.0}{24.0}
\drawpath{8.0}{24.0}{6.0}{24.0}
\drawpath{62.0}{40.0}{68.0}{40.0}
\drawpath{68.0}{40.0}{68.0}{36.0}
\drawpath{68.0}{36.0}{72.0}{36.0}
\drawpath{72.0}{36.0}{72.0}{32.0}
\drawpath{72.0}{32.0}{78.0}{32.0}
\drawpath{78.0}{32.0}{78.0}{28.0}
\drawpath{78.0}{28.0}{76.0}{28.0}
\drawpath{76.0}{28.0}{76.0}{26.0}
\drawpath{76.0}{26.0}{74.0}{26.0}
\drawpath{74.0}{26.0}{74.0}{24.0}
\drawpath{74.0}{24.0}{72.0}{24.0}
\drawpath{72.0}{24.0}{72.0}{18.0}
\drawpath{72.0}{18.0}{70.0}{18.0}
\drawpath{70.0}{18.0}{70.0}{12.0}
\drawpath{70.0}{12.0}{64.0}{12.0}
\drawpath{64.0}{12.0}{64.0}{14.0}
\drawpath{64.0}{14.0}{58.0}{14.0}
\drawpath{58.0}{14.0}{58.0}{20.0}
\drawpath{58.0}{20.0}{54.0}{20.0}
\drawpath{54.0}{20.0}{54.0}{24.0}
\drawpath{54.0}{24.0}{56.0}{24.0}
\drawpath{56.0}{24.0}{56.0}{28.0}
\drawpath{56.0}{28.0}{52.0}{28.0}
\drawpath{52.0}{28.0}{52.0}{32.0}
\drawpath{52.0}{32.0}{58.0}{32.0}
\drawpath{58.0}{32.0}{58.0}{36.0}
\drawpath{58.0}{32.0}{58.0}{36.0}
\drawpath{58.0}{36.0}{60.0}{36.0}
\drawpath{60.0}{36.0}{60.0}{38.0}
\drawpath{60.0}{38.0}{62.0}{38.0}
\drawpath{62.0}{38.0}{62.0}{40.0}
\thinlines
\drawpath{68.0}{40.0}{74.0}{40.0}
\drawpath{72.0}{36.0}{72.0}{42.0}
\drawpath{72.0}{36.0}{78.0}{36.0}
\drawpath{78.0}{32.0}{78.0}{38.0}
\drawpath{78.0}{32.0}{84.0}{32.0}
\drawpath{72.0}{32.0}{66.0}{32.0}
\drawpath{68.0}{36.0}{62.0}{36.0}
\drawpath{68.0}{40.0}{68.0}{46.0}
\drawpath{68.0}{36.0}{68.0}{30.0}
\drawpath{62.0}{40.0}{56.0}{40.0}
\drawpath{62.0}{40.0}{62.0}{46.0}
\drawpath{62.0}{38.0}{62.0}{32.0}
\drawpath{62.0}{38.0}{68.0}{38.0}
\drawpath{60.0}{36.0}{66.0}{36.0}
\drawpath{60.0}{38.0}{54.0}{38.0}
\drawpath{58.0}{36.0}{52.0}{36.0}
\drawpath{52.0}{32.0}{46.0}{32.0}
\drawpath{58.0}{32.0}{64.0}{32.0}
\drawpath{58.0}{36.0}{58.0}{42.0}
\drawpath{58.0}{32.0}{58.0}{26.0}
\drawpath{60.0}{38.0}{60.0}{44.0}
\drawpath{60.0}{36.0}{60.0}{30.0}
\drawpath{72.0}{32.0}{72.0}{26.0}
\drawpath{76.0}{28.0}{70.0}{28.0}
\drawpath{78.0}{28.0}{84.0}{28.0}
\drawpath{78.0}{28.0}{78.0}{22.0}
\drawpath{76.0}{26.0}{82.0}{26.0}
\drawpath{76.0}{26.0}{76.0}{20.0}
\drawpath{74.0}{26.0}{68.0}{26.0}
\drawpath{76.0}{28.0}{76.0}{34.0}
\drawpath{74.0}{26.0}{74.0}{32.0}
\drawpath{74.0}{24.0}{80.0}{24.0}
\drawpath{74.0}{24.0}{74.0}{18.0}
\drawpath{72.0}{24.0}{66.0}{24.0}
\drawpath{72.0}{24.0}{72.0}{30.0}
\drawpath{72.0}{18.0}{72.0}{12.0}
\drawpath{72.0}{18.0}{78.0}{18.0}
\drawpath{70.0}{18.0}{64.0}{18.0}
\drawpath{70.0}{12.0}{70.0}{6.0}
\drawpath{70.0}{18.0}{70.0}{24.0}
\drawpath{70.0}{12.0}{76.0}{12.0}
\drawpath{64.0}{12.0}{58.0}{12.0}
\drawpath{64.0}{14.0}{64.0}{20.0}
\drawpath{64.0}{12.0}{64.0}{6.0}
\drawpath{64.0}{12.0}{58.0}{12.0}
\drawpath{58.0}{20.0}{64.0}{20.0}
\drawpath{54.0}{20.0}{48.0}{20.0}
\drawpath{54.0}{20.0}{54.0}{14.0}
\drawpath{58.0}{14.0}{58.0}{8.0}
\drawpath{58.0}{14.0}{52.0}{14.0}
\drawpath{58.0}{20.0}{58.0}{26.0}
\drawpath{58.0}{24.0}{52.0}{24.0}
\drawpath{52.0}{28.0}{46.0}{28.0}
\drawpath{56.0}{28.0}{62.0}{28.0}
\drawpath{56.0}{28.0}{56.0}{34.0}
\drawpath{56.0}{24.0}{62.0}{24.0}
\drawpath{56.0}{24.0}{56.0}{18.0}
\drawpath{54.0}{24.0}{48.0}{24.0}
\drawpath{54.0}{24.0}{54.0}{30.0}
\drawpath{52.0}{28.0}{52.0}{22.0}
\drawpath{52.0}{32.0}{52.0}{38.0}
\drawpath{64.0}{14.0}{70.0}{14.0}
\drawcenteredtext{16.0}{22.0}{$\ran(\gamma)$}
\drawcenteredtext{64.0}{22.0}{$\ran(w)$}
\drawoverbrace{81.0}{33.0}{6.0}
\drawcenteredtext{81.0}{35.0}{$\alpha$}
\end{picture}
}
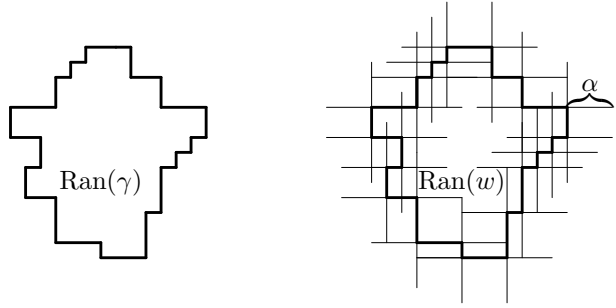
{A standard polygon $\gamma$ and the corresponding $\alpha$-wire $w$}{fig:curve1}

The above definitions easily imply the following facts.

\begin{fact} \label{f:standard} The standard curves form a dense set in $\mathcal{C}^2$.
\end{fact}

\begin{fact} \label{f:wire}
Let $E,E'\subseteq \mathbb{R}^2$ be two adjacent edges of an $\alpha$-wire. Assume that $f\in \mathcal{C}^2$ and $I,I'\subseteq [0,1]$ are disjoint closed intervals such that
$f(I)=E$ and $f(I')=E'$. Then for all $g\in U(f,\alpha/2)$ we obtain that $g(I)\cap g(I')\neq \emptyset$.
\end{fact}

The next theorem answers the category dual of Problem \ref{p:Werner} in the negative.

\begin{theorem} \label{t:br2} For the generic $f \in \mathcal{C}^2$ if $p,q\in \mathbb{R}^2\setminus \ran(f)$ and $\ran(f)$ separates
$p$ and $q$ then for every continuum $C\subseteq \mathbb{R}^2$ with $p,q\in C$ the intersection $C \cap \ran(f)$ has cardinality continuum.
\end{theorem}

Before proving Theorem \ref{t:br2} we need some lemmas.

\begin{lemma} \label{l:opensep} Let $f\in \mathcal{C}^2$ be a standard curve and let $p,q\in \mathbb{R}^2\setminus \ran (f)$. If $V\subseteq \mathbb{R}^2$ is a standard open set
such that $V\cap \ran(f)$ separates $p$ and $q$ then there is a standard
polygon $\gamma\in \mathcal{C}^2$ such that $\ran(\gamma) \subseteq V\cap \ran(f)$ separates $p$ and $q$.
\end{lemma}

\begin{proof}
We say that an open line segment $S$ is a \emph{marginal segment} if $S\subseteq V\cap \ran(f)$ and one of its endpoints is in $\partial V$
and the other endpoint is either in $\partial V$, or is a turning point, or is
a point attained by $f$ more than once. Let $S_1,\dots,S_n\subseteq V\cap \ran(f)$ be the marginal
segments. Clearly, they are pairwise disjoint.
Set $K=(V\cap \ran(f))\setminus \left(\bigcup_{i=1}^{n} S_i\right)$.
Then $K\subseteq V\cap \ran(f)$ is compact.

First we prove that $K$ separates $p$ and $q$. Assume to the contrary that this is not the case,
then there is a standard curve $g\in \mathcal{C}^2$ with $g(0)=p$ and $g(1)=q$ such that $K\cap \ran(g)=\emptyset$.
Let $y_1$ be an endpoint of $S_1$ in $\partial V$. If $g$ meets $S_1$ then we can modify $g$ in a small neighborhood of $S_1$ such that the modified curve
passes through $y_1$ and avoids $S_1\cup K$. Continuing this procedure we obtain a curve $\widetilde{g}\in \mathcal{C}^2$ with $\widetilde{g}(0)=p$ and $\widetilde{g}(1)=q$ such that $(V\cap \ran(f))\cap \ran(\widetilde{g})=\emptyset$, but this contradicts the fact that $V\cap \ran(f)$ separates $p$ and $q$.

Now elementary considerations show that there is a standard
polygon $\gamma \in \mathcal{C}^2$ such that $\ran(\gamma) \subseteq K$. For the sake of completeness we mention that \cite[Thm.~14.3.]{Ne}
implies that $K$ contains a component $C$ separating $p$ and $q$. Then $C$ is a locally connected continuum,
so \cite[(2.41)]{Wh} yields that there exists a simple closed curve
$\gamma \in \mathcal{C}^2$ such that $\ran(\gamma) \subseteq C$ separates $p$ and $q$,
and after reparametrization $\gamma$ will be a standard polygon.
\end{proof}

\begin{lemma} \label{l:br3} Let $p,q\in \mathbb{R}^2$ and $\varepsilon>0$. Assume that
$f\in \mathcal{C}^2$ is a standard curve and $V\subseteq \mathbb{R}^2$ is a standard open set such that
$V\cap \ran(f)$ separates $p$ and $q$. Then there exist a standard curve
$f_0\in \mathcal{C}^2$, standard open sets $V_0,V_1\subseteq \mathbb{R}^2$ and $\delta>0$ such that
\begin{enumerate}[(i)]
\item \label{l:iii}  $V_0,V_1\subseteq V$ and $\dist(V_0,V_1)>0$,
\item \label{l:iv}  $V_j\cap \ran(g)$ separates $p$ and $q$ if $j\in \{0,1\}$ and $g\in U(f_0,\delta)$,
\item  \label{l:ii}  $U(f_0,\delta)\subseteq U(f,\varepsilon)$,
\item \label{l:i}  $f_0=f$ on $[0,1]\setminus f^{-1}(V)$.
\end{enumerate}
\end{lemma}

\begin{proof}
It is enough to construct a not necessarily standard $f_0\in \mathcal{C}^2$ with the
above properties, since by Fact~\ref{f:standard} we can replace it with a standard one.

By Lemma~\ref{l:opensep}
there exists a standard
polygon $\gamma \in \mathcal{C}^2$ such that $\ran(\gamma) \subseteq V\cap \ran(f)$ separates $p$ and $q$.
Let $\Gamma=\ran(\gamma)$, we may assume by decreasing $\varepsilon$ if necessary
that $U(\Gamma,\varepsilon)\subseteq V$ and $p,q \notin U(\Gamma,\varepsilon)$. Let us suppose that $q$ is in the bounded component of $\mathbb{R}^2\setminus \Gamma$.
Then the winding numbers $WN(\gamma,p)=0$ and $WN(\gamma,q)=\pm 1$, see
\cite[Proposition~3.16.]{Fu} and \cite[Proposition~5.20.]{Fu}.
It is easy to see that we can fix a small enough $\alpha\in (0,\varepsilon/9)$ such that
there exist non-intersecting $\alpha$-wires $w_0, w_1\in U(\gamma,\varepsilon/3)$, see Figure \ref{fig:curve2}.

\placedrawing[ht]{
\begin{picture}(58,58)
\Thicklines
\drawpath{11.0}{47.0}{47.0}{47.0}
\drawpath{47.0}{47.0}{47.0}{11.0}
\drawpath{47.0}{11.0}{23.0}{11.0}
\drawpath{23.0}{11.0}{23.0}{23.0}
\drawpath{23.0}{23.0}{11.0}{23.0}
\drawpath{11.0}{23.0}{11.0}{47.0}
\drawpath{16.0}{44.0}{16.0}{42.0}
\drawpath{16.0}{42.0}{14.0}{42.0}
\drawpath{14.0}{42.0}{14.0}{40.0}
\drawpath{14.0}{40.0}{16.0}{40.0}
\drawpath{16.0}{40.0}{16.0}{38.0}
\drawpath{16.0}{38.0}{14.0}{38.0}
\drawpath{14.0}{38.0}{14.0}{36.0}
\drawpath{14.0}{36.0}{16.0}{36.0}
\drawpath{16.0}{36.0}{16.0}{34.0}
\drawpath{16.0}{34.0}{14.0}{34.0}
\drawpath{14.0}{34.0}{14.0}{32.0}
\drawpath{14.0}{32.0}{16.0}{32.0}
\drawpath{16.0}{32.0}{16.0}{30.0}
\drawpath{16.0}{30.0}{14.0}{30.0}
\drawpath{14.0}{30.0}{14.0}{28.0}
\drawpath{14.0}{28.0}{16.0}{28.0}
\drawpath{16.0}{28.0}{16.0}{26.0}
\drawpath{16.0}{26.0}{18.0}{26.0}
\drawpath{18.0}{26.0}{18.0}{28.0}
\drawpath{18.0}{28.0}{20.0}{28.0}
\drawpath{20.0}{28.0}{20.0}{26.0}
\drawpath{20.0}{26.0}{22.0}{26.0}
\drawpath{22.0}{26.0}{22.0}{28.0}
\drawpath{22.0}{28.0}{24.0}{28.0}
\drawpath{24.0}{28.0}{24.0}{26.0}
\drawpath{24.0}{26.0}{26.0}{26.0}
\drawpath{26.0}{26.0}{26.0}{24.0}
\drawpath{26.0}{24.0}{28.0}{24.0}
\drawpath{28.0}{24.0}{28.0}{22.0}
\drawpath{28.0}{22.0}{26.0}{22.0}
\drawpath{26.0}{22.0}{26.0}{20.0}
\drawpath{26.0}{20.0}{28.0}{20.0}
\drawpath{28.0}{20.0}{28.0}{18.0}
\drawpath{28.0}{18.0}{26.0}{18.0}
\drawpath{26.0}{18.0}{26.0}{16.0}
\drawpath{26.0}{16.0}{28.0}{16.0}
\drawpath{28.0}{16.0}{28.0}{14.0}
\drawpath{28.0}{14.0}{30.0}{14.0}
\drawpath{30.0}{14.0}{30.0}{16.0}
\drawpath{30.0}{16.0}{32.0}{16.0}
\drawpath{32.0}{16.0}{32.0}{14.0}
\drawpath{32.0}{14.0}{34.0}{14.0}
\drawpath{34.0}{14.0}{34.0}{16.0}
\drawpath{34.0}{16.0}{36.0}{16.0}
\drawpath{36.0}{16.0}{36.0}{14.0}
\drawpath{36.0}{14.0}{38.0}{14.0}
\drawpath{38.0}{14.0}{38.0}{16.0}
\drawpath{38.0}{16.0}{40.0}{16.0}
\drawpath{40.0}{16.0}{40.0}{14.0}
\drawpath{40.0}{14.0}{42.0}{14.0}
\drawpath{42.0}{14.0}{42.0}{16.0}
\drawpath{42.0}{16.0}{44.0}{16.0}
\drawpath{44.0}{16.0}{44.0}{18.0}
\drawpath{44.0}{18.0}{42.0}{18.0}
\drawpath{42.0}{18.0}{42.0}{20.0}
\drawpath{42.0}{20.0}{44.0}{20.0}
\drawpath{44.0}{20.0}{44.0}{22.0}
\drawpath{44.0}{22.0}{42.0}{22.0}
\drawpath{42.0}{22.0}{42.0}{24.0}
\drawpath{42.0}{24.0}{44.0}{24.0}
\drawpath{44.0}{24.0}{44.0}{26.0}
\drawpath{44.0}{26.0}{42.0}{26.0}
\drawpath{42.0}{26.0}{42.0}{28.0}
\drawpath{42.0}{28.0}{44.0}{28.0}
\drawpath{42.0}{30.0}{42.0}{30.0}
\drawpath{44.0}{28.0}{44.0}{30.0}
\drawpath{44.0}{30.0}{42.0}{30.0}
\drawpath{42.0}{30.0}{42.0}{32.0}
\drawpath{42.0}{32.0}{44.0}{32.0}
\drawpath{44.0}{32.0}{44.0}{34.0}
\drawpath{44.0}{34.0}{42.0}{34.0}
\drawpath{42.0}{34.0}{42.0}{36.0}
\drawpath{42.0}{36.0}{44.0}{36.0}
\drawpath{44.0}{36.0}{44.0}{38.0}
\drawpath{44.0}{38.0}{42.0}{38.0}
\drawpath{42.0}{38.0}{42.0}{40.0}
\drawpath{42.0}{40.0}{44.0}{40.0}
\drawpath{44.0}{40.0}{44.0}{42.0}
\drawpath{44.0}{42.0}{42.0}{42.0}
\drawpath{42.0}{42.0}{42.0}{44.0}
\drawpath{42.0}{44.0}{40.0}{44.0}
\drawpath{40.0}{44.0}{40.0}{42.0}
\drawpath{40.0}{42.0}{38.0}{42.0}
\drawpath{38.0}{42.0}{38.0}{44.0}
\drawpath{38.0}{44.0}{36.0}{44.0}
\drawpath{36.0}{44.0}{36.0}{42.0}
\drawpath{36.0}{42.0}{34.0}{42.0}
\drawpath{34.0}{42.0}{34.0}{44.0}
\drawpath{34.0}{44.0}{32.0}{44.0}
\drawpath{32.0}{44.0}{32.0}{42.0}
\drawpath{32.0}{42.0}{30.0}{42.0}
\drawpath{30.0}{42.0}{30.0}{44.0}
\drawpath{30.0}{44.0}{28.0}{44.0}
\drawpath{28.0}{44.0}{28.0}{42.0}
\drawpath{28.0}{42.0}{26.0}{42.0}
\drawpath{26.0}{42.0}{26.0}{44.0}
\drawpath{26.0}{44.0}{24.0}{44.0}
\drawpath{24.0}{44.0}{24.0}{42.0}
\drawpath{24.0}{42.0}{22.0}{42.0}
\drawpath{22.0}{42.0}{22.0}{44.0}
\drawpath{22.0}{44.0}{20.0}{44.0}
\drawpath{20.0}{44.0}{20.0}{42.0}
\drawpath{20.0}{42.0}{18.0}{42.0}
\drawpath{18.0}{42.0}{18.0}{44.0}
\drawpath{18.0}{44.0}{18.0}{44.0}
\drawpath{18.0}{44.0}{16.0}{44.0}
\thinlines
\drawpath{14.0}{44.0}{14.0}{26.0}
\drawpath{14.0}{26.0}{28.0}{26.0}
\drawpath{28.0}{26.0}{28.0}{12.0}
\drawpath{26.0}{14.0}{44.0}{14.0}
\drawpath{44.0}{14.0}{44.0}{44.0}
\drawpath{44.0}{44.0}{14.0}{44.0}
\drawpath{26.0}{14.0}{26.0}{28.0}
\drawpath{26.0}{28.0}{12.0}{28.0}
\drawpath{16.0}{24.0}{16.0}{46.0}
\drawpath{12.0}{42.0}{46.0}{42.0}
\drawpath{42.0}{12.0}{42.0}{46.0}
\drawpath{18.0}{40.0}{18.0}{46.0}
\drawpath{20.0}{40.0}{20.0}{46.0}
\drawpath{22.0}{40.0}{22.0}{46.0}
\drawpath{24.0}{40.0}{24.0}{46.0}
\drawpath{26.0}{40.0}{26.0}{46.0}
\drawpath{28.0}{40.0}{28.0}{46.0}
\drawpath{30.0}{40.0}{30.0}{46.0}
\drawpath{32.0}{40.0}{32.0}{46.0}
\drawpath{34.0}{40.0}{34.0}{46.0}
\drawpath{36.0}{40.0}{36.0}{46.0}
\drawpath{38.0}{40.0}{38.0}{46.0}
\drawpath{40.0}{40.0}{40.0}{46.0}
\drawpath{40.0}{40.0}{46.0}{40.0}
\drawpath{40.0}{38.0}{46.0}{38.0}
\drawpath{40.0}{36.0}{46.0}{36.0}
\drawpath{40.0}{34.0}{46.0}{34.0}
\drawpath{40.0}{32.0}{46.0}{32.0}
\drawpath{40.0}{30.0}{46.0}{30.0}
\drawpath{40.0}{28.0}{46.0}{28.0}
\drawpath{40.0}{26.0}{46.0}{26.0}
\drawpath{40.0}{24.0}{46.0}{24.0}
\drawpath{40.0}{22.0}{46.0}{22.0}
\drawpath{40.0}{20.0}{46.0}{20.0}
\drawpath{40.0}{22.0}{46.0}{22.0}
\drawpath{40.0}{24.0}{46.0}{24.0}
\drawpath{12.0}{40.0}{18.0}{40.0}
\drawpath{12.0}{38.0}{18.0}{38.0}
\drawpath{12.0}{36.0}{18.0}{36.0}
\drawpath{12.0}{34.0}{18.0}{34.0}
\drawpath{12.0}{32.0}{18.0}{32.0}
\drawpath{12.0}{30.0}{18.0}{30.0}
\drawpath{18.0}{24.0}{18.0}{30.0}
\drawpath{20.0}{24.0}{20.0}{30.0}
\drawpath{22.0}{24.0}{22.0}{30.0}
\drawpath{24.0}{24.0}{24.0}{30.0}
\drawpath{24.0}{24.0}{30.0}{24.0}
\drawpath{24.0}{22.0}{30.0}{22.0}
\drawpath{24.0}{20.0}{30.0}{20.0}
\drawpath{24.0}{18.0}{30.0}{18.0}
\drawpath{24.0}{16.0}{30.0}{16.0}
\drawpath{30.0}{18.0}{30.0}{12.0}
\drawpath{32.0}{18.0}{32.0}{12.0}
\drawpath{34.0}{18.0}{34.0}{12.0}
\drawpath{36.0}{18.0}{36.0}{12.0}
\drawpath{38.0}{18.0}{38.0}{12.0}
\drawpath{40.0}{18.0}{40.0}{12.0}
\drawpath{40.0}{18.0}{46.0}{18.0}
\drawpath{40.0}{16.0}{46.0}{16.0}
\Thicklines
\drawpath{14.0}{20.0}{14.0}{18.0}
\drawpath{14.0}{18.0}{16.0}{18.0}
\drawpath{16.0}{18.0}{16.0}{16.0}
\drawpath{16.0}{16.0}{18.0}{16.0}
\drawpath{18.0}{16.0}{18.0}{14.0}
\drawpath{18.0}{14.0}{20.0}{14.0}
\drawpath{20.0}{14.0}{20.0}{12.0}
\drawpath{20.0}{12.0}{18.0}{12.0}
\drawpath{18.0}{12.0}{18.0}{10.0}
\drawpath{18.0}{10.0}{20.0}{10.0}
\drawpath{20.0}{10.0}{20.0}{8.0}
\drawpath{20.0}{8.0}{22.0}{8.0}
\drawpath{22.0}{8.0}{22.0}{6.0}
\drawpath{22.0}{6.0}{24.0}{6.0}
\drawpath{24.0}{6.0}{24.0}{8.0}
\drawpath{24.0}{8.0}{26.0}{8.0}
\drawpath{26.0}{8.0}{26.0}{6.0}
\drawpath{26.0}{6.0}{28.0}{6.0}
\drawpath{28.0}{6.0}{28.0}{8.0}
\drawpath{28.0}{8.0}{30.0}{8.0}
\drawpath{30.0}{8.0}{30.0}{6.0}
\drawpath{30.0}{6.0}{32.0}{6.0}
\drawpath{32.0}{6.0}{32.0}{8.0}
\drawpath{32.0}{8.0}{34.0}{8.0}
\drawpath{34.0}{8.0}{34.0}{6.0}
\drawpath{34.0}{6.0}{36.0}{6.0}
\drawpath{36.0}{6.0}{36.0}{8.0}
\drawpath{36.0}{8.0}{38.0}{8.0}
\drawpath{38.0}{8.0}{38.0}{6.0}
\drawpath{38.0}{6.0}{40.0}{6.0}
\drawpath{40.0}{6.0}{40.0}{8.0}
\drawpath{40.0}{8.0}{42.0}{8.0}
\drawpath{42.0}{8.0}{42.0}{6.0}
\drawpath{42.0}{6.0}{44.0}{6.0}
\drawpath{44.0}{6.0}{44.0}{8.0}
\drawpath{44.0}{8.0}{46.0}{8.0}
\drawpath{46.0}{8.0}{46.0}{6.0}
\drawpath{46.0}{6.0}{48.0}{6.0}
\drawpath{48.0}{6.0}{48.0}{8.0}
\drawpath{48.0}{8.0}{50.0}{8.0}
\drawpath{50.0}{8.0}{50.0}{10.0}
\drawpath{50.0}{10.0}{52.0}{10.0}
\drawpath{52.0}{10.0}{52.0}{12.0}
\drawpath{52.0}{12.0}{50.0}{12.0}
\drawpath{50.0}{12.0}{50.0}{14.0}
\drawpath{50.0}{14.0}{52.0}{14.0}
\drawpath{52.0}{14.0}{52.0}{16.0}
\drawpath{52.0}{16.0}{50.0}{16.0}
\drawpath{50.0}{16.0}{50.0}{18.0}
\drawpath{50.0}{18.0}{52.0}{18.0}
\drawpath{52.0}{18.0}{52.0}{20.0}
\drawpath{52.0}{20.0}{50.0}{20.0}
\drawpath{50.0}{20.0}{50.0}{22.0}
\drawpath{50.0}{22.0}{52.0}{22.0}
\drawpath{52.0}{22.0}{52.0}{24.0}
\drawpath{52.0}{24.0}{50.0}{24.0}
\drawpath{50.0}{24.0}{50.0}{26.0}
\drawpath{50.0}{26.0}{52.0}{26.0}
\drawpath{52.0}{26.0}{52.0}{28.0}
\drawpath{52.0}{28.0}{50.0}{28.0}
\drawpath{50.0}{28.0}{50.0}{30.0}
\drawpath{50.0}{30.0}{52.0}{30.0}
\drawpath{52.0}{30.0}{52.0}{32.0}
\drawpath{52.0}{32.0}{50.0}{32.0}
\drawpath{50.0}{32.0}{50.0}{34.0}
\drawpath{50.0}{34.0}{52.0}{34.0}
\drawpath{52.0}{34.0}{52.0}{36.0}
\drawpath{52.0}{36.0}{50.0}{36.0}
\drawpath{50.0}{36.0}{50.0}{38.0}
\drawpath{50.0}{38.0}{52.0}{38.0}
\drawpath{52.0}{38.0}{52.0}{40.0}
\drawpath{52.0}{40.0}{50.0}{40.0}
\drawpath{50.0}{40.0}{50.0}{42.0}
\drawpath{50.0}{42.0}{52.0}{42.0}
\drawpath{52.0}{42.0}{52.0}{42.0}
\drawpath{52.0}{42.0}{52.0}{44.0}
\drawpath{52.0}{44.0}{50.0}{44.0}
\drawpath{50.0}{44.0}{50.0}{46.0}
\drawpath{50.0}{46.0}{52.0}{46.0}
\drawpath{52.0}{46.0}{52.0}{48.0}
\drawpath{52.0}{48.0}{50.0}{48.0}
\drawpath{50.0}{48.0}{50.0}{50.0}
\drawpath{50.0}{50.0}{48.0}{50.0}
\drawpath{48.0}{50.0}{48.0}{52.0}
\drawpath{48.0}{52.0}{46.0}{52.0}
\drawpath{46.0}{52.0}{46.0}{50.0}
\drawpath{46.0}{50.0}{44.0}{50.0}
\drawpath{44.0}{50.0}{44.0}{52.0}
\drawpath{44.0}{52.0}{42.0}{52.0}
\drawpath{42.0}{52.0}{42.0}{50.0}
\drawpath{42.0}{50.0}{40.0}{50.0}
\drawpath{40.0}{50.0}{40.0}{52.0}
\drawpath{40.0}{52.0}{38.0}{52.0}
\drawpath{38.0}{52.0}{38.0}{50.0}
\drawpath{38.0}{50.0}{36.0}{50.0}
\drawpath{36.0}{50.0}{36.0}{52.0}
\drawpath{36.0}{52.0}{34.0}{52.0}
\drawpath{34.0}{52.0}{34.0}{50.0}
\drawpath{34.0}{50.0}{32.0}{50.0}
\drawpath{32.0}{50.0}{32.0}{52.0}
\drawpath{32.0}{52.0}{30.0}{52.0}
\drawpath{30.0}{52.0}{30.0}{50.0}
\drawpath{30.0}{50.0}{28.0}{50.0}
\drawpath{28.0}{50.0}{28.0}{52.0}
\drawpath{28.0}{52.0}{26.0}{52.0}
\drawpath{26.0}{52.0}{26.0}{50.0}
\drawpath{26.0}{50.0}{24.0}{50.0}
\drawpath{24.0}{50.0}{24.0}{52.0}
\drawpath{22.0}{50.0}{22.0}{52.0}
\drawpath{22.0}{50.0}{20.0}{50.0}
\drawpath{20.0}{50.0}{20.0}{52.0}
\drawpath{20.0}{52.0}{18.0}{52.0}
\drawpath{18.0}{52.0}{18.0}{50.0}
\drawpath{18.0}{50.0}{16.0}{50.0}
\drawpath{16.0}{50.0}{16.0}{52.0}
\drawpath{16.0}{52.0}{14.0}{52.0}
\drawpath{14.0}{52.0}{14.0}{50.0}
\drawpath{14.0}{50.0}{12.0}{50.0}
\drawpath{12.0}{50.0}{12.0}{52.0}
\drawpath{12.0}{52.0}{10.0}{52.0}
\drawpath{10.0}{52.0}{10.0}{50.0}
\drawpath{10.0}{50.0}{8.0}{50.0}
\drawpath{8.0}{50.0}{8.0}{48.0}
\drawpath{8.0}{48.0}{6.0}{48.0}
\drawpath{6.0}{48.0}{6.0}{46.0}
\drawpath{6.0}{46.0}{8.0}{46.0}
\drawpath{8.0}{46.0}{8.0}{44.0}
\drawpath{8.0}{44.0}{6.0}{44.0}
\drawpath{6.0}{44.0}{6.0}{42.0}
\drawpath{6.0}{42.0}{8.0}{42.0}
\drawpath{8.0}{42.0}{8.0}{40.0}
\drawpath{8.0}{40.0}{6.0}{40.0}
\drawpath{6.0}{40.0}{6.0}{38.0}
\drawpath{6.0}{38.0}{8.0}{38.0}
\drawpath{8.0}{38.0}{8.0}{36.0}
\drawpath{8.0}{36.0}{6.0}{36.0}
\drawpath{6.0}{36.0}{6.0}{34.0}
\drawpath{6.0}{34.0}{8.0}{34.0}
\drawpath{8.0}{34.0}{8.0}{32.0}
\drawpath{8.0}{32.0}{6.0}{32.0}
\drawpath{6.0}{32.0}{6.0}{30.0}
\drawpath{6.0}{30.0}{8.0}{30.0}
\drawpath{8.0}{30.0}{8.0}{28.0}
\drawpath{8.0}{28.0}{6.0}{28.0}
\drawpath{6.0}{28.0}{6.0}{26.0}
\drawpath{6.0}{26.0}{8.0}{26.0}
\drawpath{8.0}{26.0}{8.0}{24.0}
\drawpath{8.0}{24.0}{6.0}{24.0}
\drawpath{6.0}{24.0}{6.0}{22.0}
\drawpath{6.0}{22.0}{8.0}{22.0}
\drawpath{8.0}{22.0}{8.0}{20.0}
\drawpath{8.0}{20.0}{10.0}{20.0}
\drawpath{10.0}{20.0}{10.0}{18.0}
\drawpath{10.0}{18.0}{12.0}{18.0}
\drawpath{12.0}{18.0}{12.0}{20.0}
\drawpath{12.0}{20.0}{14.0}{20.0}
\drawpath{22.0}{52.0}{24.0}{52.0}
\thinlines
\drawpath{8.0}{52.0}{50.0}{52.0}
\drawpath{50.0}{52.0}{50.0}{6.0}
\drawpath{50.0}{6.0}{20.0}{6.0}
\drawpath{20.0}{6.0}{20.0}{16.0}
\drawpath{20.0}{16.0}{14.0}{16.0}
\drawpath{14.0}{16.0}{14.0}{22.0}
\drawpath{16.0}{20.0}{6.0}{20.0}
\drawpath{6.0}{20.0}{6.0}{50.0}
\drawpath{6.0}{50.0}{52.0}{50.0}
\drawpath{52.0}{50.0}{52.0}{8.0}
\drawpath{52.0}{8.0}{18.0}{8.0}
\drawpath{18.0}{8.0}{18.0}{18.0}
\drawpath{18.0}{18.0}{8.0}{18.0}
\drawpath{8.0}{18.0}{8.0}{52.0}
\drawpath{10.0}{54.0}{10.0}{48.0}
\drawpath{12.0}{54.0}{12.0}{48.0}
\drawpath{14.0}{54.0}{14.0}{48.0}
\drawpath{16.0}{54.0}{16.0}{48.0}
\drawpath{16.0}{54.0}{16.0}{48.0}
\drawpath{18.0}{54.0}{18.0}{48.0}
\drawpath{20.0}{54.0}{20.0}{48.0}
\drawpath{22.0}{54.0}{22.0}{48.0}
\drawpath{24.0}{54.0}{24.0}{48.0}
\drawpath{26.0}{54.0}{26.0}{48.0}
\drawpath{28.0}{54.0}{28.0}{48.0}
\drawpath{30.0}{54.0}{30.0}{48.0}
\drawpath{32.0}{54.0}{32.0}{48.0}
\drawpath{34.0}{54.0}{34.0}{48.0}
\drawpath{36.0}{54.0}{36.0}{48.0}
\drawpath{38.0}{54.0}{38.0}{48.0}
\drawpath{40.0}{54.0}{40.0}{48.0}
\drawpath{42.0}{54.0}{42.0}{48.0}
\drawpath{44.0}{54.0}{44.0}{48.0}
\drawpath{46.0}{54.0}{46.0}{48.0}
\drawpath{48.0}{54.0}{48.0}{48.0}
\drawpath{48.0}{48.0}{54.0}{48.0}
\drawpath{48.0}{46.0}{54.0}{46.0}
\drawpath{48.0}{44.0}{54.0}{44.0}
\drawpath{48.0}{42.0}{54.0}{42.0}
\drawpath{48.0}{40.0}{54.0}{40.0}
\drawpath{48.0}{38.0}{54.0}{38.0}
\drawpath{48.0}{36.0}{54.0}{36.0}
\drawpath{48.0}{34.0}{54.0}{34.0}
\drawpath{48.0}{32.0}{54.0}{32.0}
\drawpath{48.0}{30.0}{54.0}{30.0}
\drawpath{48.0}{28.0}{54.0}{28.0}
\drawpath{48.0}{26.0}{54.0}{26.0}
\drawpath{48.0}{24.0}{54.0}{24.0}
\drawpath{48.0}{22.0}{54.0}{22.0}
\drawpath{48.0}{20.0}{54.0}{20.0}
\drawpath{48.0}{18.0}{54.0}{18.0}
\drawpath{48.0}{16.0}{54.0}{16.0}
\drawpath{48.0}{14.0}{54.0}{14.0}
\drawpath{54.0}{16.0}{52.0}{16.0}
\drawpath{48.0}{12.0}{54.0}{12.0}
\drawpath{48.0}{10.0}{54.0}{10.0}
\drawpath{48.0}{10.0}{48.0}{4.0}
\drawpath{46.0}{4.0}{46.0}{10.0}
\drawpath{44.0}{10.0}{44.0}{4.0}
\drawpath{42.0}{4.0}{42.0}{10.0}
\drawpath{40.0}{10.0}{40.0}{4.0}
\drawpath{38.0}{4.0}{38.0}{10.0}
\drawpath{36.0}{10.0}{36.0}{4.0}
\drawpath{34.0}{4.0}{34.0}{10.0}
\drawpath{32.0}{10.0}{32.0}{4.0}
\drawpath{30.0}{4.0}{30.0}{10.0}
\drawpath{28.0}{10.0}{28.0}{4.0}
\drawpath{26.0}{4.0}{26.0}{10.0}
\drawpath{24.0}{10.0}{24.0}{4.0}
\drawpath{22.0}{4.0}{22.0}{10.0}
\drawpath{16.0}{10.0}{22.0}{10.0}
\drawpath{22.0}{12.0}{16.0}{12.0}
\drawpath{16.0}{14.0}{22.0}{14.0}
\drawpath{16.0}{14.0}{16.0}{20.0}
\drawpath{12.0}{16.0}{12.0}{22.0}
\drawpath{10.0}{16.0}{10.0}{22.0}
\drawpath{4.0}{22.0}{10.0}{22.0}
\drawpath{4.0}{24.0}{10.0}{24.0}
\drawpath{4.0}{26.0}{10.0}{26.0}
\drawpath{10.0}{28.0}{4.0}{28.0}
\drawpath{4.0}{30.0}{10.0}{30.0}
\drawpath{10.0}{32.0}{4.0}{32.0}
\drawpath{4.0}{34.0}{10.0}{34.0}
\drawpath{10.0}{36.0}{4.0}{36.0}
\drawpath{4.0}{38.0}{10.0}{38.0}
\drawpath{10.0}{40.0}{4.0}{40.0}
\drawpath{4.0}{42.0}{10.0}{42.0}
\drawpath{10.0}{44.0}{4.0}{44.0}
\drawpath{4.0}{46.0}{10.0}{46.0}
\drawpath{10.0}{48.0}{4.0}{48.0}
\drawcenteredtext{20.0}{20.0}{$\Gamma$}
\drawcenteredtext{30.0}{34.0}{$\ran(w_0)$}
\drawcenteredtext{8.0}{10.0}{$\ran(w_1)$}
\end{picture}
}
{Illustration to Lemma~\ref{l:br3}}{fig:curve2}

Let $W_j=\ran(w_j)$ for $j\in \{0,1\}$, then clearly $W_0\cap W_1=\emptyset$ and $W_0\cup W_1\subseteq U(\Gamma,\varepsilon/3)$.
Let us denote the edges of $w_0$ and $w_1$ by $E_1,\dots,E_{n_0}$ and $E_{n_0+1},\dots, E_{n_0+n_1}$, respectively.
Set $n=n_0+n_1$. Since $W_0\cup W_1\subseteq U(\Gamma,\varepsilon/3)\subseteq U(\ran(f),\varepsilon/3)$
and $\diam E_k\leq 3\alpha<\varepsilon/3$, there exist distinct points $z_1,\dots,z_n\in [0,1]$
such that $E_k\subseteq U(f(z_k),2\varepsilon/3)$ for all $k\in \{1,\dots,n\}$.
By the continuity of $f$ there are pairwise disjoint closed non-degenerate intervals
$I_1,\dots, I_{n}\subseteq [0,1]$ such that $E_k\subseteq U(f(x),2\varepsilon/3)$ for all $x\in I_k$ and $k\in \{1,\dots,n\}$.
Notice that each $x\in \bigcup_{k=1}^{n} I_k$ satisfies $f(x)\in U(W_0\cup
W_1, 2\varepsilon/3)\subseteq U(\Gamma,\varepsilon)\subseteq V$, thus
$\bigcup_{k=1}^{n} I_k\subseteq f^{-1}(V)$.
Now for all $k\in \{1,\dots,n\}$ define $f_0|_{I_k}$ such that $f_0(I_k)=E_k$
and let $f_0=f$ on $[0,1]\setminus f^{-1}(V)$.
It is easy to check that $|f(x) - f_0 (x)| < 2\varepsilon/3$ for every $x \in
\bigcup_{k=1}^{n} I_k \cup \left([0,1] \setminus f^{-1}(V)\right)$.
Then Tietze's Extension Theorem implies that we can extend $f_0$ continuously
to $[0,1]$ such that $f_0\in B(f,2\varepsilon/3)$. Property \eqref{l:i}
follows from the definition of $f_0$.

Now let $\delta=\min\left \{\alpha/2,\dist(W_0,W_1)/5 \right\}>0$. As $f_0\in B(f,2\varepsilon/3)$ and
$\delta\leq \alpha/2<\varepsilon/3$, we obtain $U(f_0,\delta)\subseteq U(f,\varepsilon)$, thus property \eqref{l:ii} holds.

Let us pick standard open sets $V_0$ and $V_1$ such that $U(W_j,\delta)\subseteq V_j \subseteq U(W_j,2\delta)$ for $j\in \{0,1\}$, then
$$V_j\subseteq U(W_j,2\delta)\subseteq U(\Gamma,2\delta+\varepsilon/3)\subseteq U(\Gamma ,\varepsilon)\subseteq V,$$
and
$$\dist(V_0,V_1)\geq \dist(W_0,W_1)-4\delta \geq \delta>0,$$
hence property \eqref{l:iii} holds.

Finally, let us fix $g\in U(f_0,\delta)$, we need to prove for property \eqref{l:iv} that $V_0\cap \ran(g)$ separates $p$ and $q$,
and of course the same argument works for $V_1\cap \ran(g)$. Let us define the compact set $K=\bigcup_{k=1}^{n_0} g(I_k)$.
Then $\bigcup_{k=1}^{n_0} f_0(I_k)=W_0$ yields $K\subseteq U(W_0,\delta)\subseteq V_0$, thus $K\subseteq V_0\cap \ran(g)$.
Therefore it is enough to prove that $K$ separates $p$ and $q$. Consider $0=x_0<x_1<\dots<x_{n_0}=1$ such that
$w_0\left([x_{k-1},x_k]\right)=E_k$ for every $k\in \{1,\dots, n_0\}$.
Then $g\in U(f_0,\delta)\subseteq U(f_0,\alpha/2)$ and Fact~\ref{f:wire} imply that $g(I_k)\cap g(I_{k+1})\neq \emptyset$ for all
$k\in \{1,\dots, n_{0}-1\}$ and $g(I_{n_0})\cap g(I_1)\neq \emptyset$. Thus we can choose $u_k,v_k\in I_k$ for all $k\in \{1,\dots,n_0\}$ such that $g(v_k)=g(u_{k+1})$ and let $h_k \colon [x_{k-1},x_{k}] \to [u_k,v_k]$ be a homeomorphism with $h_k(x_{k-1})=u_{k}$
and $h_k(x_k)=v_k$, where we use the notations $u_{n_0+1}=u_1$ and $[u,v]=\left[\min\{u,v\},\max\{u,v\}\right]$.
Let us define $\varphi \colon [0,1]\to K$ such that $\varphi|_{[x_{k-1},x_k]}=g\circ h_k$ for all $k\in \{1,\dots,n_0\}$.
Then $\varphi$ is a closed curve with $\ran(\varphi)\subseteq K$,
so it is enough to prove that $\ran(\varphi)$ separates $p$ and $q$.
The function $WN(\varphi,\cdot)$ is constant on the components of
$\mathbb{R}^2\setminus \ran(\varphi)$ by \cite[Proposition 3.16.]{Fu},
thus it is sufficient to prove that $WN(\varphi,p)\neq WN(\varphi,q)$.
From $g\in U(f_0,\delta)$, $f_0(I_k)=E_k$ and $\diam E_k<\varepsilon/3$ it follows that $\varphi\in U(w_0,\delta+\varepsilon/3)\subseteq U(w_0,2\varepsilon/3)$.
Thus $w_0\in U(\gamma,\varepsilon/3)$ yields $\varphi \in U(\gamma,\varepsilon)$.
Hence $\varphi$ and $\gamma$ are homotopic in $U(\Gamma,\varepsilon)\subseteq \mathbb{R}^{2}\setminus \{p,q\}$,
so $WN(\varphi,p)=WN(\gamma,p)=0$ and $WN(\varphi,q)=WN(\gamma,q)=\pm 1$, see \cite[Corollary~3.8.]{Fu}.
The proof is complete.
\end{proof}

\begin{lemma} \label{l:br4} Assume $p,q\in \mathbb{R}^2$ and $\varepsilon>0$. Let $f\in \mathcal{C}^2$ be a standard curve and
let $V_1,\dots,V_m \subseteq \mathbb{R}^2$ be pairwise disjoint standard open sets such that $V_k\cap \ran(f)$ separates $p$ and $q$ for all $k\in \{1,\dots,m\}$.
Then there exist a non-empty open set $\mathcal{V}\subseteq U(f,\varepsilon)$ and standard open sets $V_{kj}\subseteq \mathbb{R}^2$
such that for all $k\in \{1,\dots,m\}$ and $j\in \{0,1\}$
\begin{enumerate}[(i)]
\item \label{e:1} $V_{kj}\subseteq V_k$ and $\dist(V_{k0},V_{k1})>0$,
\item \label{e:2} $V_{kj}\cap \ran(g)$ separates $p$ and $q$ if $g\in \mathcal{V}$.
\end{enumerate}
\end{lemma}

\begin{proof} Let $g_0=f$ and $\delta_0=\varepsilon$.
Now for $k\in \{0,\dots,m-1\}$ we construct by induction a standard curve
$g_{k}\in \mathcal{C}^2$, a $\delta_{k}>0$, and for all $i\in \{1,\dots,k\}$ and $j\in \{0,1\}$ standard open sets $V_{ij}$ such that for all $i\in \{1,\dots,k\}$ and $j\in \{0,1\}$
\begin{enumerate}[(1)]
\item \label{e:a}  $V_{ij}\subseteq V_i$ and $\dist(V_{i0},V_{i1})>0$,
\item  \label{e:b} $V_{ij}\cap \ran(g)$ separates $p$ and $q$ if $g\in U(g_k,\delta_k)$,
\item  \label{e:c} $U(g_k,\delta_k)\subseteq U(f,\varepsilon)$,
\item  \label{e:d} $g_k=f$ on $[0,1]\setminus f^{-1}(V_1\cup\dots \cup V_k)$.
\end{enumerate}

The case $k=0$ is done, since \eqref{e:c} and \eqref{e:d} obviously hold,
while  \eqref{e:a} and \eqref{e:b} hold vacuously, since there are no sets
$V_{ij}$.

Let us now take up the inductive step. Since $g_k=f$ on
$$f^{-1}(V_{k+1})\subseteq [0,1]\setminus f^{-1}(V_1\cup\dots \cup V_k),$$
we obtain that $V_{k+1}\cap \ran(g_k)$ separates $p$ and $q$.
Lemma~\ref{l:br3} applied to the standard curve $g_k$, standard open set $V_{k+1}$ and $\delta_k>0$ yields that we can satisfy properties \eqref{e:a}-\eqref{e:d} for $k+1$.

Finally, the non-empty open set $\mathcal{V}=U(g_m,\delta_m)\subseteq U(f,\varepsilon)$ and the constructed standard open sets
$V_{ij}$ satisfy properties \eqref{e:1}-\eqref{e:2}.
\end{proof}

Now we are ready to prove Theorem \ref{t:br2}.

\begin{proof}[Proof of Theorem \ref{t:br2}] Let $\mathcal{F}$ be the set of functions $f\in \mathcal{C}^2$ such that if $\ran(f)$
separates a pair of points
$p$ and $q$ then for every continuum $C\subseteq \mathbb{R}^2$ with $p,q\in C$ the intersection $C \cap \ran(f)$ has cardinality continuum.
We need to show that $\mathcal{F}$ is co-meager in $\mathcal{C}^2$.
For fixed $p,q\in \mathbb{R}^2$ let $\mathcal{F}_{p,q}$ be the set of functions $f\in \mathcal{C}^2$ such that if $\ran(f)$
separates $p$ and
$q$ then for every continuum $C\subseteq \mathbb{R}^2$ with $p,q\in C$ the intersection $C \cap \ran(f)$ has cardinality continuum. As
$$\mathcal{F}=\bigcap _{p,q\in \mathbb{Q}\times \mathbb{Q}} \mathcal{F}_{p,q},$$
it is enough to show that the sets $\mathcal{F}_{p,q}$ are co-meager in $\mathcal{C}^2$. Let $p,q\in \mathbb{R}^2$, $p\neq q$ be arbitrarily fixed.
In order to prove that $\mathcal{F}_{p,q}$ is co-meager,
we play the Banach-Mazur game in the metric space $\mathcal{C}^2$:
First Player I chooses a non-empty open set $\mathcal{U}_1\subseteq \mathcal{C}^2$, then Player II
chooses a non-empty open set $\mathcal{V}_1\subseteq \mathcal{U}_1$, Player I continues with a non-empty open set $\mathcal{U}_2\subseteq \mathcal{V}_1$, and so on.
By definition Player II wins this game if $\bigcap_{n=1}^{\infty} \mathcal{V}_n\subseteq \mathcal{F}_{p,q}$.
It is well-known that Player II has a winning strategy iff $\mathcal{F}_{p,q}$ is co-meager in $\mathcal{C}^2$, see \cite[Thm.~1]{Ox}
or \cite[(8.33)]{Ke}. Thus we need to prove that Player II has a winning strategy.

Now we describe the strategy of Player II. If there is an $f \in \mathcal{U}_1$ such that $p$ and $q$ are in the same component of
$\mathbb{R}^2\setminus \ran(f)$ then there is an $\varepsilon>0$ such that all functions in $U(f,\varepsilon)\subseteq \mathcal{U}_1$ have this property.
Thus $U(f,\varepsilon)\subseteq \mathcal{F}_{p,q}$, and the strategy of Player II is the following. Let $\mathcal{V}_1=U(f,\varepsilon)$, and the other moves of
Player II are arbitrary. Then clearly $\bigcap_{n=1}^{\infty} \mathcal{V}_n\subseteq \mathcal{F}_{p,q}$, so Player II wins the game.

If $\ran(f)$ separates $p$ and $q$ for all $f \in \mathcal{U}_1$ then the strategy of Player II is the following.
The functions $f$ with the property $p,q\notin \ran(f)$
form a dense open subset in $\mathcal{U}_1$, so by Fact~\ref{f:standard}
there is a standard curve $f_1\in \mathcal{U}_1$ such that $\ran(f_1)$ separates $p$ and $q$. Let $V\subseteq \mathbb{R}^2$ be a standard open set with $\ran(f_1)\subseteq V$.
Let $\varepsilon_1>0$ be so small that $U(f_1,\varepsilon_1)\subseteq \mathcal{U}_1$.
Applying Lemma~\ref{l:br4} for $\varepsilon_1$, $f_1$ and $V$ implies that there is a non-empty open
set $\mathcal{V}_1 \subseteq U(f_1,\varepsilon_1)$ and standard open sets $V_0,V_1\subseteq V$ such that $\dist(V_0,V_1)>0$ and
$V_j\cap \ran(g)$ separates $p$ and $q$ for every $j\in \{0,1\}$ and $g\in \mathcal{V}_1$.
Next we suppose that $n\in \mathbb{N}^+$
and after the $n$th move of Player II
 the non-empty open set $\mathcal{V}_n$ and standard open sets $V_{j_1 \dots j_i}\subseteq \mathbb{R}^2$
($i\in \{1,\dots,n\},~j_1,\dots, j_i\in \{0,1\}$) are already defined such that for all $j_1, \dots, j_n\in \{0,1\}$
\begin{enumerate}[(i)]
\item \label{*1} $V_{j_1\dots j_n} \subseteq V_{j_1\dots j_{n-1}}$ and $\dist(V_{j_1 \dots j_{n-1}0},V_{j_1 \dots j_{n-1}1})>0$,
\item  \label{*2} $V_{j_1 \dots j_n}\cap \ran(g)$ separates $p$ and $q$ for all $g\in \mathcal{V}_n$.

\end{enumerate}
Suppose that Player I continues with $\mathcal{U}_{n+1}\subseteq \mathcal{V}_n$. By Fact~\ref{f:standard} Player II can choose a standard curve
$f_{n+1} \in \mathcal{U}_{n+1}$. Then $f_{n+1}\in \mathcal{V}_n$ implies that \eqref{*2} holds for $f_{n+1}$.
Let $\varepsilon_{n+1}>0$ be so small that $U(f_{n+1},\varepsilon_{n+1})\subseteq \mathcal{U}_{n+1}$.
Applying Lemma~\ref{l:br4} for $\varepsilon_{n+1}$, $f_{n+1}$ and the open sets
$V_{j_1\dots j_n}$ we obtain a non-empty open set $\mathcal{V}_{n+1}\subseteq U(f_{n+1},\varepsilon_{n+1})\subseteq \mathcal{U}_{n+1}$ and standard open sets
$V_{j_1\dots j_{n+1}}\subseteq \mathbb{R}^2$ witnessing that properties \eqref{*1}-\eqref{*2} hold for $n+1$.

Finally, we need to prove that Player II wins the game with the above strategy if $\ran(f)$ separates $p$ and $q$ for all $f \in \mathcal{U}_1$.
Let $f\in \bigcap_{n=1}^{\infty} \mathcal{V}_{n}$,  we need to prove that $f\in \mathcal{F}_{p,q}$.
Let $C\subseteq \mathbb{R}^2$ be a continuum with $p,q\in C$, we need to show that $C\cap \ran(f)$ has cardinality continuum.
Property \eqref{*2} implies that we can choose $c_{j_1\dots j_n}\in C\cap (V_{j_1 \dots j_n}\cap \ran(f))$
for every $n\in \mathbb{N}^{+}$ and $j_1, \dots, j_n\in \{0,1\}$. If $\underline{j}=(j_1,j_2,\dots)\in \{0,1\}^{\mathbb{N}}$
then let $c_{\underline{j}}$ be an arbitrary limit point of the set $\left\{c_{j_1\dots j_n}: n\in \mathbb{N}^{+}\right\}$.
Then the compactness of $C$ and $\ran(f)$ implies that $c_{\underline{j}}\in C\cap \ran(f)$ for every $\underline{j}\in \{0,1\}^{\mathbb{N}}$.
Thus it is enough to prove that if $\underline{j}, \underline{k}\in \{0,1\}^{\mathbb{N}}$, $\underline{j}\neq \underline{k}$
then $c_{\underline{j}}\neq c_{\underline{k}}$. Assume that $n\in \mathbb{N}^{+}$ is the minimal number such that the $n$th
coordinate of $\underline{j}$ and $\underline{k}$ differ. As the sets $V_{j_1\dots j_i}$ form a nested sequence by property \eqref{*1},
we obtain $c_{\underline{j}}\in \cl V_{j_1\dots j_n}$
and $c_{\underline{k}}\in \cl V_{k_1\dots k_n}$, and property \eqref{*1}
also yields $\cl V_{j_1\dots j_n}\cap \cl V_{k_1\dots k_n}=\emptyset$.
Therefore $c_{\underline{j}}\neq c_{\underline{k}}$, and the proof is complete.
\end{proof}

\section{Application I: Mandelbrot's fractal percolation process}
\label{s:Mandelbrot}

In this section we take up one of the most important random fractals, the
limit set $M\subseteq \mathbb{R}^2$ of the fractal percolation process defined by Mandelbrot in
\cite{Ma2}.

His original motivation was that this model captures certain
features of turbulence, but then this random set turned out to be very interesting in its own right.
For example, $M$ serves as a very powerful tool for
calculating Hausdorff dimension. It is well known that for each $\gamma\in (0,2)$ there is a fractal percolation
$M(\gamma)$ such that $\dim_H M(\gamma)=2-\gamma$ almost surely,
provided that $M(\gamma) \neq \emptyset$. Hawkes \cite{Ha} has shown that
for a fixed closed set $C$ in the unit square $M(\gamma) \cap C = \emptyset$ almost surely if $\dim_H C<\gamma$,
and $M(\gamma) \cap C\neq \emptyset$ with positive probability if $\dim_H C>\gamma$.
Therefore $\{M(\gamma): \gamma\in (0,2)\}$ serves as a family of \emph{random test sets}:
If we know which percolation limit sets $M(\gamma)$ intersect $C$ with positive probability,
we can determine $\dim_H C$. The idea of using random test sets goes back to Taylor \cite{T}, while the use of percolation limit sets as test
sets is due to Khoshnevisan, Peres, and Xiao \cite{KPX}. In the context of trees, similar ideas were developed by Lyons \cite{L}.

Moreover, it can be shown that the range of a $d$-dimensional Brownian motion
is \emph{intersection-equivalent} to a fractal percolation, that is, they intersect all
closed subsets of the $d$-dimensional unit cube with comparable probabilities. This can be used to
deduce numerous dimension related results about Brownian motion. For more on these results
see Peres \cite{P1,P2}, see also \cite{MP}.

Let us now formally describe the fractal percolation process. Let
$p\in (0,1)$ and $n\geq 2$, $n\in \mathbb{N}$ be fixed. Set $M_0 =
M_{0}^{(p,n)} =[0,1]^{2}$. We divide the unit square into $n^{2}$
equal closed sub-squares of side-length $1/n$ in the natural way. We
keep each sub-square independently with probability $p$ (and erase it
with probability $1-p$), and denote by $M_1 = M_1^{(p,n)}$ the union
of the kept sub-squares. Then each square in $M_1$ is divided into
$n^{2}$ squares of side-length ${1}/{n^2}$, and we keep each of them
independently (and also independently of the earlier choices) with
probability $p$, etc. After $k$ steps let $M_k = M_k^{(p,n)}$ be the
union of the kept $k^{th}$ level squares with side-length
${1}/{n^k}$. Let
\begin{equation}
\label{M}
M = M^{(p,n)} = \bigcap_{k=1}^{\infty} M_{k}.
\end{equation}
The process we have just described is
called \emph{Mandelbrot's fractal percolation process}, and $M$ is called its
\emph{limit set}.

Percolation fractals are not
only interesting from the point of view of turbulence and fractal geometry, but they are also
closely related to the (usual, graph-theoretic) percolation theory.
In case of
the fractal percolation the role of the clusters is played by the connected
components. Our starting point will be the following celebrated theorem.

\begin{theorem}[Chayes-Chayes-Durrett, \cite{CH}]
\label{t:CCD}
There exists a critical probability $p_c = p_c^{(n)} \in (0,1)$ such that if
$p < p_c$
then $M$ is totally disconnected almost surely, and if $p > p_c$ then
$M$ contains a nontrivial connected component with positive probability.
\end{theorem}

They actually prove more, the most powerful version states that in the
supercritical case (i.e. when $p > p_c$) there is actually a unique unbounded
component if the process is extended to the whole plane, but we will only
concentrate on the most surprising fact that the critical probability is
strictly between 0 and 1.

The main goal of the present section will be to prove the following generalization
of the above theorem.

\begin{theorem}
\label{t:lower}
For every $d \in [0,2)$ there exists a critical probability $p_c^{(d)} = p_c^{(d, n)}  \in
(0,1)$ such that if $p < p_c^{(d)}$ then $\dim_{tH} M \le d$ almost surely,
and if $p > p_c^{(d)}$ then $\dim_{tH} M > d$ almost surely (provided $M \neq \emptyset$).
\end{theorem}

In order to see that we actually obtain a generalization, just note that
a compact space is totally disconnected iff $\dim_t M = 0$ (\cite{Eng}), also that
$\dim_{tH} M = 0$ iff $\dim_t M = 0$, and use $d = 0$.
Theorem \ref{t:CCD} basically says that certain curves show up at the
critical probability, and our proof will show that even `thick' families of curves
show up, where the word thick is related to large Hausdorff dimension.

In the rest of this section first we do some preparations in the first subsection,
then we prove the main theorem (Theorem \ref{t:lower}) in the next subsection, and
finally give an upper bound for $\dim_{tH} M$ and conclude
that $\dim_{tH} M < \dim_H M$ almost surely in the non-trivial cases.

\subsection{Preparation}

For the proofs of the statements in the next two remarks see e.g.~\cite{CH}.

\begin{remark}
\label{r:egy}
It is well-known from the theory of branching processes that $M=\emptyset$
almost surely iff $p \leq \frac{1}{n^{2}}$, so we may assume in
the following that $p > \frac{1}{n^{2}}$.

If $\frac{1}{n^{2}}< p\leq \frac{1}{\sqrt{n}}$ then $\dim_t M = 0$ almost
surely. Hence Fact~\ref{<=>} implies that $\dim_{tH} M = 0$ almost surely. (In
fact, the same holds even for $p < p_c$, see Theorem \ref{t:CCD}.)
\end{remark}

\begin{remark}
\label{r:ketto}
As for the Hausdorff dimension, for $p > \frac{1}{n^{2}}$ we have
\[
\dim_H M = 2 + \frac{\log p}{\log n}
\]
almost surely, provided $M \neq \emptyset$.

We will also need the 1-dimensional analogue of the process (intervals instead
of squares). Here $M^{(1D)} = \emptyset$ almost surely iff $p \leq \frac1n$,
and for $p > \frac1n$ we have
\[
\dim_H M^{(1D)} = 1 + \frac{\log p}{\log n}
\]
almost surely, provided $M^{(1D)} \neq \emptyset$.
\end{remark}

Now we check that the almost sure topological Hausdorff dimension of $M$
also exists.

\begin{lemma}
\label{thomog} For every $p > \frac {1}{n^{2}}$ and $n \ge 2$, $n\in
\mathbb{N}$ there exists a number $d  = d^{(p,n)} \in [0,2]$ such
that
\[
\dim_{tH} M = d
\]
almost surely, provided $M \neq \emptyset$.
\end{lemma}

\begin{proof} Let $N$ be the random number of squares in $M_1$. Let us set
$q = P(M = \emptyset)$. Then $q < 1$ by Remark~\ref{r:egy}, and \cite[Thm.~15.2]{F}
gives that $q$ is the least positive root of the polynomial
\[
f(t)=-t+\sum _{k=0}^{n^2} P(N=k)t^{k}.
\]
Let us fix an arbitrary $x\in [0,\infty)$. As $q>0$, we obtain $P(\dim_{tH} M \leq x)>0$.
First we show that $P(\dim_{tH} M \leq x)$ is a root of $f$.

If $N>0$ then let $M_{1} = \{Q_{1},\dots ,Q_{N}\}$, where the $Q_{i}$ are the
first level sub-squares. For every $i$ and $k$ let $M_k^{Q_{i}}$
be the union of those squares in $M_k$ that are in $Q_i$, and let
$M^{Q_{i}} = \bigcap_k M_k^{Q_i}$. (Note that this is not the same
as $M \cap Q_i$, since in this latter set there may be points on the
boundary of $Q_i$ `coming from squares outside of $Q_i$'.) Then
$M^{Q_{i}}$ has the same distribution as a similar copy of $M$ (this
is called statistical self-similarity), and hence for every $i$
\[
P \left( \dim_{tH} M^{Q_i} \leq x \right) =
P \left( \dim_{tH} M \leq x \right).
\]
Using the stability of the topological Hausdorff dimension for
closed sets and the fact that the $M^{Q_i}$ are
independent and have the same distribution under the condition $N=k>0$, this
implies
\begin{align*}
P \left( \dim_{tH} M \leq x \,| \, N=k \right) &= P \left( \dim_{tH}
M^{Q_i} \leq x \textrm{ for each } 1 \le i\leq k\,|
  \, N=k \right) \notag \\
&=\big(P \left( \dim_{tH} M^{Q_1} \leq x \right) \big)^{k} \\
&= \big(P \left( \dim_{tH} M \leq x \right) \big)^{k}.
\end{align*}
For $N=0$ we have $P \left( \dim_{tH} M \leq x \,| \, N=0 \right)=1$, therefore we obtain
\begin{align*} P ( \dim_{tH} M \leq x) &= \sum _{k=0}^{n^2} P(N=k)
P ( \dim_{tH} M \leq x \,| \, N=k)\\
&= \sum _{k=0}^{n^2} P(N=k) \big( P \left( \dim_{tH} M \leq x
\right) \big)^{k},
\end{align*}
and thus $P ( \dim_{tH} M \leq x)$ is indeed a root of $f$ for every $x\in [0,\infty)$.

As mentioned above, $q \neq 1$ and $q$ is also a root of
$f$. Moreover, $1$ is obviously also a root, and it is easy to see that $f$ is
strictly convex on $[0,\infty)$, hence there are at most two nonnegative roots.
Therefore $q$ and $1$ are the only nonnegative roots, thus $P(\dim_{tH} M \leq x)= q$ or $1$
for every $x\in [0,\infty)$.

If $P(\dim_{tH} M \leq 0)=1$ then we are done, so we may assume that $P(\dim_{tH} M \leq 0)=q$.
Then the distribution function
$$F(x)=P(\dim_{tH} M \leq x \, | \, M \neq \emptyset)=\frac{P(\dim_{tH} M \leq x)-q}{1-q}$$
only attains the values $0$ and $1$, and clearly $F(0) = 0$ and $F(2) = 1$. Thus there is a value $d$ where it
`jumps' from $0$ to $1$. This concludes the proof.
\end{proof}

\subsection{Proof of Theorem \ref{t:lower}; the lower estimate of $\dim_{tH} M$}

Set
\[
p_c^{(d,n)} = \sup \left\{p : \dim_{tH} M^{(p,n)} \le d \textrm{ almost
  surely}\right\}.
\]


First we need some lemmas. The following one is analogous to \cite[p. 387]{GG}.

\begin{lemma}
\label{ln2} For every $d\in \mathbb{R}$ and $n\in \mathbb{N}$, $n
\ge 2$
\[
p_c^{(d,n)} < 1 \quad \Longleftrightarrow \quad p_c^{(d,n^2)} < 1.
\]
\end{lemma}

\begin{proof} Clearly, it is enough to show that
\begin{equation}
\label{pcn}
p_c^{(d,n)} \left( 1-\left(1-p_c^{(d,n)} \right)^{\frac
{1}{n^2}}\right)\leq p_c^{(d,n^2)} \leq p_c^{(d,n)}.
\end{equation}
We say that the random construction $X$ is dominated by the random
construction $Y$ if they can be realized on the same probability
space such that $X\subseteq Y$ almost surely.

Let us first prove the second inequality in \eqref{pcn}. It clearly suffices to show that
\[
\dim_{tH} M^{(p,n^2)} \le d \textrm{ almost surely} \implies \dim_{tH}
M^{(p,n)} \le d \textrm{ almost surely}.
\]
But this is rather straightforward, since $M_{2k}^{(p,n)}$ is easily seen to
be dominated by $M_k^{(p,n^2)}$ for every $k$, hence $M^{(p,n)}$ is dominated by
$M^{(p,n^2)}$.

Let us now prove the first inequality in \eqref{pcn}. Set $\varphi(x) =
1-\left(1-x\right)^{{1}/{n^2}}$. We need to show that
\begin{equation}
\label{e:kell}
0 < p < p_c^{(d,n)} \varphi(p_c^{(d,n)}) \implies \dim_{tH}
M^{(p,n^2)} \le d \textrm{ almost surely}.
\end{equation}

Since $x\varphi(x)$ is an increasing homeomorphism of the unit interval, $p =
q \varphi(q)$ for some $q \in (0,1)$. Then clearly $q < p_c^{(d,n)}$, so $\dim_{tH}
M^{(q,n)} \le d$ almost surely. Therefore, in order to prove (\ref{e:kell}) it
suffices to check that
\begin{equation}
\label{e:kell2}
M^{(p,n^2)} \textrm{ is dominated by } M^{(q,n)}.
\end{equation}

First we check that
\begin{equation}
\label{e:domin}
M_k^{(\varphi(q),n^2)} \textrm{ is dominated by } M_k^{(q,n)}  \textrm{ for
  every } k,
\end{equation}
and consequently $M^{(\varphi(q),n^2)}$ is dominated by $M^{(q,n)}$. Indeed,
in the second case we
erase a sub-square of side length $\frac1n$ with probability $1-q$ and
keep it with probability $q$, while in the first case we \emph{completely}
erase a sub-square of side length $\frac1n$ with the same probability
$(1-\varphi(q))^{n^2}=1-q$ and hence keep \emph{at least a subset of it} with
probability $q$.

But this will easily imply (\ref{e:kell2}), which will complete the
proof.
Indeed, after each step of the processes
$M^{(\varphi(q),n^2)}$ and $M^{(q,n)}$ let us perform the
following procedures. For $M^{(\varphi(q),n^2)}$ let us keep every existing
square independently with probability $q$ and erase it with probability $1-q$
(we do not do any subdivisions in this case). For $M^{(q,n)}$ let us take one
more step of the construction of $M^{(q,n)}$. Using (\ref{e:domin}) this
easily implies that $M_k^{(q \varphi(q),n^2)}$ is
dominated by $M_{2k}^{(q,n)}$ for every $k$, hence $M^{(q \varphi(q),n^2)}$ is
dominated by $M^{(q,n)}$, but $q \varphi(q) = p$, and hence (\ref{e:kell2}) holds.
\end{proof}

\bigskip

From now on let $N$ be a fixed (large) positive integer to be chosen
later. Recall that a square of level $k$ is a set of the form
$\left[\frac{i}{N^k}, \frac{i+1}{N^k}\right] \times
\left[\frac{j}{N^k}, \frac{j+1}{N^k}\right] \subseteq [0,1]^2$.

\begin{definition}
A \emph{walk of level $k$} is a sequence $(S_1, \dots, S_l)$ of non-overlapping squares of level $k$ such
that $S_r$ and $S_{r+1}$ are abutting for every $r = 1, \dots, l-1$, moreover $S_1 \cap ( \{0\} \times [0,1])
\neq \emptyset$ and $S_l \cap (\{1\} \times [0,1]) \neq \emptyset$.
\end{definition}

In particular, the only walk of level $0$ is $([0,1]^2)$.

\begin{definition}
We say that $(S_1, \dots, S_l)$ is a \emph{turning walk (of level $1$)} if it satisfies the properties of a walk of level $1$ except that instead of
$S_l \cap (\{1\} \times [0,1]) \neq \emptyset$ we require that $S_l \cap ([0,1] \times \{ 1 \} ) \neq \emptyset$.
\end{definition}

\begin{lemma}
\label{l:through} Let $\mathcal{S}$ be a set of $N-2$ distinct
squares of level $1$ intersecting $\{ 0 \} \times [0,1]$, and let
$\mathcal{T}$ be a set of $N-2$ distinct squares of level $1$
intersecting $\{ 1 \} \times [0,1]$. Moreover, let $F^*$ be a square
of level $1$ such that the row of $F^*$ does not intersect
$\mathcal{S} \cup \mathcal{T}$. Then there exist $N-2$
non-overlapping walks of level $1$ not containing $F^*$ such that
the set of their first squares coincides with $\mathcal{S}$ and the
set of their last squares coincides with $\mathcal{T}$. (See Figure~\ref{fig:percol3}.)
\end{lemma}

\placedrawing[ht]{
\begin{picture}(92,38)
\thinlines
\drawshadebox{4.0}{16.0}{28.0}{10.0}{22.0}{34.0}{}{0.5}
\drawshadebox{34.0}{46.0}{58.0}{10.0}{22.0}{34.0}{}{0.5}
\drawshadebox{64.0}{76.0}{88.0}{10.0}{22.0}{34.0}{}{0.5}
\drawshadebox{16.0}{18.0}{20.0}{30.0}{32.0}{34.0}{}{0}
\drawshadebox{4.0}{6.0}{8.0}{26.0}{28.0}{30.0}{}{0.25}
\drawshadebox{4.0}{6.0}{8.0}{22.0}{24.0}{26.0}{}{0.25}
\drawshadebox{4.0}{6.0}{8.0}{18.0}{20.0}{22.0}{}{0.25}
\drawshadebox{24.0}{26.0}{28.0}{22.0}{24.0}{26.0}{}{0.25}
\drawshadebox{24.0}{26.0}{28.0}{18.0}{20.0}{22.0}{}{0.25}
\drawshadebox{24.0}{26.0}{28.0}{14.0}{16.0}{18.0}{}{0.25}
\drawshadebox{24.0}{26.0}{28.0}{10.0}{12.0}{14.0}{}{0.25}
\drawshadebox{4.0}{6.0}{8.0}{14.0}{16.0}{18.0}{}{0.25}
\drawshadebox{34.0}{36.0}{38.0}{30.0}{32.0}{34.0}{}{0.25}
\drawshadebox{54.0}{56.0}{58.0}{30.0}{32.0}{34.0}{}{0.25}
\drawshadebox{42.0}{44.0}{46.0}{22.0}{24.0}{26.0}{}{0}
\drawshadebox{34.0}{36.0}{38.0}{26.0}{28.0}{30.0}{}{0.25}
\drawshadebox{34.0}{36.0}{38.0}{18.0}{20.0}{22.0}{}{0.25}
\drawshadebox{54.0}{56.0}{58.0}{14.0}{16.0}{18.0}{}{0.25}
\drawshadebox{54.0}{56.0}{58.0}{10.0}{12.0}{14.0}{}{0.25}
\drawshadebox{34.0}{36.0}{38.0}{10.0}{12.0}{14.0}{}{0.25}
\drawshadebox{54.0}{56.0}{58.0}{18.0}{20.0}{22.0}{}{0.25}
\drawshadebox{84.0}{86.0}{88.0}{22.0}{24.0}{26.0}{}{0.25}
\drawshadebox{84.0}{86.0}{88.0}{14.0}{16.0}{18.0}{}{0.25}
\drawshadebox{72.0}{74.0}{76.0}{18.0}{20.0}{22.0}{}{0}
\drawshadebox{64.0}{66.0}{68.0}{14.0}{16.0}{18.0}{}{0.25}
\drawshadebox{64.0}{66.0}{68.0}{26.0}{28.0}{30.0}{}{0.25}
\drawshadebox{84.0}{86.0}{88.0}{26.0}{28.0}{30.0}{}{0.25}
\drawshadebox{84.0}{86.0}{88.0}{10.0}{12.0}{14.0}{}{0.25}
\drawshadebox{64.0}{66.0}{68.0}{30.0}{32.0}{34.0}{}{0.25}
\drawshadebox{64.0}{66.0}{68.0}{22.0}{24.0}{26.0}{}{0.25}
\drawpath{8.0}{10.0}{8.0}{32.0}
\drawpath{8.0}{32.0}{8.0}{34.0}
\drawpath{12.0}{10.0}{12.0}{34.0}
\drawpath{16.0}{10.0}{16.0}{34.0}
\drawpath{20.0}{10.0}{20.0}{32.0}
\drawpath{20.0}{34.0}{20.0}{32.0}
\drawpath{24.0}{10.0}{24.0}{34.0}
\drawpath{4.0}{30.0}{28.0}{30.0}
\drawpath{4.0}{26.0}{28.0}{26.0}
\drawpath{4.0}{22.0}{28.0}{22.0}
\drawpath{4.0}{18.0}{28.0}{18.0}
\drawpath{4.0}{14.0}{26.0}{14.0}
\drawpath{26.0}{14.0}{28.0}{14.0}
\drawpath{34.0}{30.0}{34.0}{30.0}
\drawpath{34.0}{30.0}{58.0}{30.0}
\drawpath{34.0}{26.0}{58.0}{26.0}
\drawpath{42.0}{10.0}{42.0}{34.0}
\drawpath{46.0}{10.0}{46.0}{34.0}
\drawpath{34.0}{18.0}{58.0}{18.0}
\drawpath{34.0}{14.0}{58.0}{14.0}
\drawpath{38.0}{10.0}{38.0}{34.0}
\drawpath{50.0}{10.0}{50.0}{34.0}
\drawpath{54.0}{10.0}{54.0}{34.0}
\drawpath{34.0}{22.0}{58.0}{22.0}
\drawpath{34.0}{34.0}{58.0}{34.0}
\drawpath{58.0}{34.0}{58.0}{10.0}
\drawpath{58.0}{10.0}{34.0}{10.0}
\drawpath{34.0}{10.0}{34.0}{34.0}
\drawpath{4.0}{34.0}{28.0}{34.0}
\drawpath{28.0}{34.0}{28.0}{10.0}
\drawpath{28.0}{10.0}{4.0}{10.0}
\drawpath{4.0}{10.0}{4.0}{34.0}
\drawpath{64.0}{34.0}{88.0}{34.0}
\drawpath{88.0}{34.0}{88.0}{10.0}
\drawpath{88.0}{10.0}{64.0}{10.0}
\drawpath{64.0}{10.0}{64.0}{34.0}
\drawpath{64.0}{30.0}{88.0}{30.0}
\drawpath{64.0}{26.0}{88.0}{26.0}
\drawpath{64.0}{22.0}{88.0}{22.0}
\drawpath{64.0}{18.0}{88.0}{18.0}
\drawpath{64.0}{14.0}{88.0}{14.0}
\drawpath{68.0}{34.0}{68.0}{10.0}
\drawpath{72.0}{34.0}{72.0}{10.0}
\drawpath{76.0}{34.0}{76.0}{10.0}
\drawpath{80.0}{34.0}{80.0}{10.0}
\drawpath{84.0}{34.0}{84.0}{10.0}
\Thicklines
\drawpath{4.0}{28.0}{26.0}{28.0}
\drawpath{26.0}{28.0}{26.0}{24.0}
\drawpath{26.0}{24.0}{28.0}{24.0}
\drawpath{4.0}{24.0}{22.0}{24.0}
\drawpath{22.0}{24.0}{22.0}{20.0}
\drawpath{22.0}{20.0}{28.0}{20.0}
\drawpath{4.0}{20.0}{18.0}{20.0}
\drawpath{18.0}{20.0}{18.0}{16.0}
\drawpath{18.0}{16.0}{28.0}{16.0}
\drawpath{4.0}{16.0}{14.0}{16.0}
\drawpath{14.0}{16.0}{14.0}{12.0}
\drawpath{14.0}{12.0}{28.0}{12.0}
\drawpath{28.0}{12.0}{28.0}{12.0}
\drawpath{34.0}{32.0}{58.0}{32.0}
\drawpath{54.0}{20.0}{58.0}{20.0}
\drawpath{54.0}{16.0}{58.0}{16.0}
\drawpath{54.0}{12.0}{58.0}{12.0}
\drawpath{84.0}{24.0}{88.0}{24.0}
\drawpath{84.0}{16.0}{88.0}{16.0}
\drawpath{84.0}{12.0}{88.0}{12.0}
\thinlines
\drawpath{64.0}{28.0}{82.0}{28.0}
\drawpath{82.0}{28.0}{82.0}{24.0}
\drawpath{82.0}{24.0}{84.0}{24.0}
\drawpath{64.0}{24.0}{78.0}{24.0}
\drawpath{78.0}{24.0}{78.0}{16.0}
\drawpath{78.0}{16.0}{86.0}{16.0}
\drawpath{64.0}{16.0}{74.0}{16.0}
\drawpath{74.0}{16.0}{74.0}{12.0}
\drawpath{74.0}{12.0}{84.0}{12.0}
\drawpath{34.0}{28.0}{52.0}{28.0}
\drawpath{52.0}{28.0}{52.0}{20.0}
\drawpath{52.0}{20.0}{54.0}{20.0}
\drawpath{34.0}{20.0}{48.0}{20.0}
\drawpath{48.0}{20.0}{48.0}{16.0}
\drawpath{48.0}{16.0}{54.0}{16.0}
\drawpath{34.0}{12.0}{54.0}{12.0}
\drawcenteredtext{16.0}{6.0}{Case 1.}
\drawcenteredtext{46.0}{6.0}{Case 2.}
\drawcenteredtext{76.0}{6.0}{Case 3.}
\Thicklines
\drawpath{64.0}{32.0}{86.0}{32.0}
\drawpath{86.0}{32.0}{86.0}{28.0}
\drawpath{86.0}{28.0}{88.0}{28.0}
\drawframebox{44.0}{20.0}{20.0}{20.0}{}
\drawframebox{74.0}{20.0}{20.0}{20.0}{}
\end{picture}
}
{Illustration to Lemma~\ref{l:through}}{fig:percol3}

\begin{proof}
The proof is by induction on $N$. The case $N = 2$ is obvious.

\noindent
\textbf{Case 1.} \emph{$F^*$ is in the top or bottom row.}

By simply ignoring this row it is straightforward how to construct the walks
in the remaining rows.

\noindent \textbf{Case 2.} \emph{$F^*$ is not in the top or bottom
row, and both top corners or both bottom corners are in $\mathcal{S}
\cup \mathcal{T}$.}

Without loss of generality we may suppose that both top corners are
in $\mathcal{S} \cup \mathcal{T}$. Let the straight walk connecting
these two corners be one of the walks to be constructed. Then let us
shift the remaining members of $\mathcal{T}$ to the left by one
square, and either we can apply the induction hypothesis to the
$(N-1) \times (N-1)$ many squares in the bottom left corner of the
original $N \times N$ many squares, or  $F^*$ is not among these
$(N-1) \times (N-1)$ many squares and then the argument is even
easier. Then one can see how to get the required walks.

\noindent
\textbf{Case 3.} \emph{Neither Case 1 nor Case 2 holds.}

Since there are only two squares missing on both sides, and $F^*$
cannot be the top or bottom row, we infer that both $\mathcal{S}$
and $\mathcal{T}$ contain at least one corner. Since Case 2 does not
hold, we obtain that both the top left and the bottom right corners
or both the bottom left and the top right corners are in
$\mathcal{S} \cup \mathcal{T}$. Without loss of generality we may
suppose that
 both the top left and the bottom
right corners are in $\mathcal{S} \cup \mathcal{T}$. By reflecting
the picture about the center of the unit square if necessary, we may
assume that $F^*$ is not in the rightmost column. We now construct
the first walk. Let it run straight from the top left corner to the
top right corner, and then continue downwards until it first reaches
a member of $\mathcal{T}$. Then,  as above, we can similarly apply
the induction hypothesis to the $(N-1) \times (N-1)$ many squares in
the bottom left corner, and we are done.
\end{proof}

\begin{lemma}
\label{l:turn} Let $\mathcal{S}$ be a set of $N-2$ distinct squares
of level $1$ intersecting $\{ 0 \} \times [0,1]$, and let
$\mathcal{T}$ be a set of $N-2$ distinct squares of level $1$
intersecting $[0,1] \times \{ 1 \}$ (the sets of starting and
terminal squares). Moreover, let $F^*$ be a square of level $1$ (the
forbidden square) such that the row of $F^*$ does not intersect
$\mathcal{S}$ and the column of $F^*$ does not intersect
$\mathcal{T}$. Then there exist $N-2$ non-overlapping turning walks
not containing $F^*$ such that the set of their first squares
coincides with $\mathcal{S}$ and the set of their last squares
coincides with $\mathcal{T}$.
\end{lemma}

\begin{proof}
Obvious, just take the simplest `L-shaped' walks.
\end{proof}

\placedrawing[ht]{
\begin{picture}(58,58)
\thinlines
\drawshadebox{4.0}{9.0}{14.0}{14.0}{19.0}{24.0}{}{0.4}
\drawshadebox{14.0}{19.0}{24.0}{14.0}{19.0}{24.0}{}{0.4}
\drawshadebox{14.0}{19.0}{24.0}{24.0}{29.0}{34.0}{}{0.4}
\drawshadebox{14.0}{19.0}{24.0}{34.0}{39.0}{44.0}{}{0.4}
\drawshadebox{24.0}{29.0}{34.0}{34.0}{39.0}{44.0}{}{0.4}
\drawshadebox{34.0}{39.0}{44.0}{34.0}{39.0}{44.0}{}{0.4}
\drawshadebox{34.0}{39.0}{44.0}{24.0}{29.0}{34.0}{}{0.4}
\drawshadebox{34.0}{39.0}{44.0}{14.0}{19.0}{24.0}{}{0.4}
\drawshadebox{34.0}{39.0}{44.0}{4.0}{9.0}{14.0}{}{0.4}
\drawshadebox{44.0}{49.0}{54.0}{4.0}{9.0}{14.0}{}{0.4}
\drawshadebox{14.0}{19.0}{24.0}{44.0}{49.0}{54.0}{}{0.4}
\drawshadebox{24.0}{29.0}{34.0}{44.0}{49.0}{54.0}{}{0.4}
\drawshadebox{10.0}{11.0}{12.0}{18.0}{19.0}{20.0}{}{0}
\drawshadebox{16.0}{17.0}{18.0}{20.0}{21.0}{22.0}{}{0}
\drawshadebox{22.0}{23.0}{24.0}{28.0}{29.0}{30.0}{}{0}
\drawshadebox{14.0}{15.0}{16.0}{38.0}{39.0}{40.0}{}{0}
\drawshadebox{24.0}{25.0}{26.0}{38.0}{38.0}{38.0}{}{0}
\drawshadebox{24.0}{25.0}{26.0}{36.0}{37.0}{38.0}{}{0}
\drawshadebox{36.0}{37.0}{38.0}{40.0}{41.0}{42.0}{}{0}
\drawshadebox{34.0}{35.0}{36.0}{24.0}{25.0}{26.0}{}{0}
\drawshadebox{40.0}{41.0}{42.0}{14.0}{15.0}{16.0}{}{0}
\drawshadebox{38.0}{39.0}{40.0}{10.0}{11.0}{12.0}{}{0}
\drawshadebox{44.0}{45.0}{46.0}{8.0}{9.0}{10.0}{}{0}
\drawshadebox{18.0}{19.0}{20.0}{48.0}{49.0}{50.0}{}{0}
\drawshadebox{26.0}{27.0}{28.0}{44.0}{45.0}{46.0}{}{0}
\drawpath{4.0}{54.0}{54.0}{54.0}
\drawpath{54.0}{54.0}{54.0}{4.0}
\drawpath{54.0}{4.0}{4.0}{4.0}
\drawpath{4.0}{4.0}{4.0}{54.0}
\drawpath{14.0}{54.0}{14.0}{4.0}
\drawpath{24.0}{54.0}{24.0}{4.0}
\drawpath{34.0}{54.0}{34.0}{4.0}
\drawpath{44.0}{54.0}{44.0}{4.0}
\drawpath{44.0}{4.0}{44.0}{54.0}
\drawpath{4.0}{44.0}{54.0}{44.0}
\drawpath{4.0}{34.0}{54.0}{34.0}
\drawpath{4.0}{24.0}{54.0}{24.0}
\drawpath{4.0}{14.0}{54.0}{14.0}
\drawpath{4.0}{22.0}{24.0}{22.0}
\drawpath{4.0}{20.0}{24.0}{20.0}
\drawpath{4.0}{18.0}{24.0}{18.0}
\drawpath{4.0}{16.0}{24.0}{16.0}
\drawpath{6.0}{24.0}{6.0}{14.0}
\drawpath{8.0}{24.0}{8.0}{14.0}
\drawpath{10.0}{24.0}{10.0}{14.0}
\drawpath{10.0}{24.0}{10.0}{24.0}
\drawpath{12.0}{24.0}{12.0}{14.0}
\drawpath{16.0}{44.0}{16.0}{14.0}
\drawpath{18.0}{44.0}{18.0}{44.0}
\drawpath{18.0}{44.0}{18.0}{14.0}
\drawpath{20.0}{44.0}{20.0}{44.0}
\drawpath{20.0}{44.0}{20.0}{14.0}
\drawpath{22.0}{44.0}{22.0}{14.0}
\drawpath{22.0}{44.0}{22.0}{14.0}
\drawpath{22.0}{14.0}{22.0}{44.0}
\drawpath{14.0}{42.0}{44.0}{42.0}
\drawpath{14.0}{40.0}{44.0}{40.0}
\drawpath{14.0}{38.0}{44.0}{38.0}
\drawpath{14.0}{36.0}{44.0}{36.0}
\drawpath{44.0}{36.0}{14.0}{36.0}
\drawpath{14.0}{36.0}{44.0}{36.0}
\drawpath{36.0}{44.0}{36.0}{4.0}
\drawpath{38.0}{44.0}{38.0}{4.0}
\drawpath{40.0}{44.0}{40.0}{4.0}
\drawpath{40.0}{30.0}{40.0}{30.0}
\drawpath{42.0}{4.0}{42.0}{44.0}
\drawpath{42.0}{44.0}{42.0}{4.0}
\drawpath{34.0}{6.0}{54.0}{6.0}
\drawpath{34.0}{8.0}{54.0}{8.0}
\drawpath{34.0}{10.0}{54.0}{10.0}
\drawpath{34.0}{12.0}{54.0}{12.0}
\drawpath{26.0}{44.0}{26.0}{34.0}
\drawpath{28.0}{44.0}{28.0}{34.0}
\drawpath{30.0}{44.0}{30.0}{34.0}
\drawpath{32.0}{44.0}{32.0}{34.0}
\drawpath{32.0}{34.0}{32.0}{44.0}
\drawpath{30.0}{44.0}{30.0}{34.0}
\drawpath{30.0}{34.0}{30.0}{44.0}
\drawpath{30.0}{44.0}{30.0}{34.0}
\drawpath{30.0}{44.0}{30.0}{34.0}
\drawpath{14.0}{32.0}{24.0}{32.0}
\drawpath{14.0}{28.0}{24.0}{28.0}
\drawpath{14.0}{30.0}{24.0}{30.0}
\drawpath{14.0}{26.0}{24.0}{26.0}
\drawpath{34.0}{32.0}{44.0}{32.0}
\drawpath{34.0}{30.0}{44.0}{30.0}
\drawpath{34.0}{28.0}{44.0}{28.0}
\drawpath{34.0}{26.0}{44.0}{26.0}
\drawpath{34.0}{32.0}{44.0}{32.0}
\drawpath{34.0}{22.0}{44.0}{22.0}
\drawpath{34.0}{20.0}{44.0}{20.0}
\drawpath{34.0}{18.0}{44.0}{18.0}
\drawpath{34.0}{16.0}{44.0}{16.0}
\drawpath{46.0}{14.0}{46.0}{4.0}
\drawpath{46.0}{4.0}{46.0}{14.0}
\drawpath{48.0}{14.0}{48.0}{4.0}
\drawpath{50.0}{14.0}{50.0}{4.0}
\drawpath{52.0}{14.0}{52.0}{4.0}
\drawpath{16.0}{44.0}{16.0}{54.0}
\drawpath{18.0}{44.0}{18.0}{54.0}
\drawpath{20.0}{44.0}{20.0}{54.0}
\drawpath{22.0}{44.0}{22.0}{54.0}
\drawpath{14.0}{52.0}{34.0}{52.0}
\drawpath{14.0}{50.0}{34.0}{50.0}
\drawpath{14.0}{46.0}{34.0}{46.0}
\drawpath{14.0}{48.0}{34.0}{48.0}
\drawpath{34.0}{48.0}{14.0}{48.0}
\drawpath{14.0}{52.0}{34.0}{52.0}
\drawpath{26.0}{54.0}{26.0}{44.0}
\drawpath{28.0}{54.0}{28.0}{44.0}
\drawpath{30.0}{54.0}{30.0}{44.0}
\drawpath{32.0}{54.0}{32.0}{44.0}
\Thicklines
\drawpath{4.0}{23.0}{15.0}{23.0}
\drawpath{15.0}{23.0}{15.0}{33.0}
\drawpath{15.0}{33.0}{17.0}{33.0}
\drawpath{17.0}{33.0}{17.0}{53.0}
\drawpath{17.0}{53.0}{33.0}{53.0}
\drawpath{33.0}{53.0}{33.0}{43.0}
\drawpath{33.0}{43.0}{43.0}{43.0}
\drawpath{43.0}{43.0}{43.0}{13.0}
\drawpath{43.0}{13.0}{54.0}{13.0}
\drawpath{4.0}{21.0}{13.0}{21.0}
\drawpath{13.0}{21.0}{13.0}{17.0}
\drawpath{13.0}{17.0}{19.0}{17.0}
\drawpath{19.0}{17.0}{19.0}{43.0}
\drawpath{19.0}{43.0}{21.0}{43.0}
\drawpath{21.0}{43.0}{21.0}{51.0}
\drawpath{21.0}{51.0}{31.0}{51.0}
\drawpath{31.0}{51.0}{31.0}{39.0}
\drawpath{31.0}{39.0}{41.0}{39.0}
\drawpath{41.0}{39.0}{41.0}{31.0}
\drawpath{41.0}{31.0}{39.0}{31.0}
\drawpath{39.0}{31.0}{39.0}{21.0}
\drawpath{39.0}{21.0}{37.0}{21.0}
\drawpath{37.0}{21.0}{37.0}{7.0}
\drawpath{37.0}{7.0}{54.0}{7.0}
\drawpath{4.0}{15.0}{21.0}{15.0}
\drawpath{21.0}{15.0}{21.0}{41.0}
\drawpath{21.0}{41.0}{23.0}{41.0}
\drawpath{23.0}{41.0}{23.0}{47.0}
\drawpath{23.0}{47.0}{29.0}{47.0}
\drawpath{29.0}{47.0}{29.0}{35.0}
\drawpath{29.0}{35.0}{39.0}{35.0}
\drawpath{39.0}{35.0}{39.0}{33.0}
\drawpath{39.0}{33.0}{37.0}{33.0}
\drawpath{37.0}{33.0}{37.0}{23.0}
\drawpath{37.0}{23.0}{35.0}{23.0}
\drawpath{35.0}{23.0}{35.0}{5.0}
\drawpath{35.0}{5.0}{54.0}{5.0}
\end{picture}
}
{Illustration to Lemma~\ref{walk}}{fig:percol1}

The last two lemmas will almost immediately imply the following.

\begin{lemma}
\label{walk} Let $(S_1, \dots, S_l)$ be a walk of level $k$, and
$\mathcal{F}$ a system of squares of level $k+1$ such that each
$S_r$ contains at most $1$ member of $\mathcal{F}$. Then $(S_1,
\dots, S_l)$ contains $N-2$ non-overlapping sub-walks of level $k+1$
avoiding $\mathcal{F}$. (See Figure~\ref{fig:percol1}.)
\end{lemma}

\begin{proof}
We may assume that each $S_r$ contains exactly $1$ member of
$\mathcal{F}$. Let us denote the member of $\mathcal{F}$ in $S_r$ by
$F^*_r$. The sub-walks will be constructed separately in each $S_r$,
using an appropriately rotated or reflected version of either Lemma~\ref{l:turn} or
Lemma~\ref{l:through}. It suffices to construct
$\mathcal{S}_r$ and $\mathcal{T}_r$ for every $r$ (compatible with
$F^*_r$) so that for every member of $\mathcal{T}_r$ there is an
abutting member of $\mathcal{S}_{r+1}$. (Of course we also have to
make sure that every member of $\mathcal{S}_1$ intersects $\{ 0 \}
\times [0,1]$ and every member of $\mathcal{T}_l$ intersects $\{ 1
\} \times [0,1]$.)  For example, the construction of $\mathcal{T}_r$
for $r < l$ is as follows. The squares $S_r$ and $S_{r+1}$ share a
common edge $E$. Assume for simplicity that $E$ is horizontal. Then
$\mathcal{T}_r$ will consist of those sub-squares of $S_r$ of level
$k+1$ that intersect $E$ and whose column differs from that of
$F^*_r$ and $F^*_{r+1}$. If these two columns happen to coincide
then we can arbitrarily erase one more square. The remaining
constructions are similar and the details are left to the reader.
\end{proof}

\bigskip

\begin{definition} We say that a square in $M_k$ is
\emph{$1$-full} if it contains at least $N^2-1$ many sub-squares from
$M_{k+1}$. We say that it is \emph{${m}$-full}, if it contains at
least $N^2-1$ many $m-1$-full sub-squares from $M_{k+1}$. We call $M$
\emph{full} if $M_{0}$ is $m$-full for every $m\in \mathbb{N}^{+}$.
\end{definition}

The following lemma was the key realization in \cite{CH}.

\begin{lemma}
\label{lfull} There is a $p^{(N)}<1$ so that $P \left(M^{(p,N)} \textrm{ is full}\right)>0$ for all $p>p^{(N)}$.
\end{lemma}

See \cite{CH} or \cite[Prop. 15.5]{F} for the proof.

\bigskip

\begin{definition}
Let $L\leq N$ be positive integers. A compact set $K \subseteq
[0,1]$ is called $(L,N)$-regular if it is of the form $K =
\bigcap_{i \in \mathbb{N}} K_i$, where $K_0 = [0,1]$, and $K_{k+1}$
is obtained by dividing every interval $I$ in $K_k$ into $N$ many
non-overlapping closed intervals of length $1/N^{k+1}$, and choosing
$L$ many of them for each $I$.
\end{definition}

The following fact is well-known, see e.g. the more general \cite[Thm.~9.3]{F}.

\begin{fact}
\label{LN}
An $(L,N)$-regular compact set has Hausdorff dimension $\frac{\log L}{\log N}$.
\end{fact}

\bigskip

Next we prove the main result of the present subsection.

\begin{proof}[Proof of Theorem~\ref{t:lower}]
Let $d \in [0, 2)$ be arbitrary. First we verify that, for
sufficiently large $N$, if $M = M^{(p,N)}$ is full then $\dim_{tH} M
> d$. The strategy is as follows. We define a collection $\mathcal{G}$ of
disjoint connected subsets of $M$ such that if a set intersects each
member of $\mathcal{G}$ then its Hausdorff dimension is larger than
$d-1$. Then we show that for every countable open basis
$\mathcal{U}$ of $M$ the union of the boundaries, $\bigcup_{U\in
\mathcal{U}}
\partial_M U$ intersects each member of $\mathcal{G}$, which clearly
implies $\dim_{tH} M
> d$.

Let us fix an integer $N$ such that
\begin{equation}
\label{e:Nbig}
N \ge 6 \textrm{ and } \frac {\log (N-2)}{\log N} > d-1,
\end{equation}
and let us assume that $M$ is full.
Using Lemma~\ref{lfull} at each step we can choose
$N-2$ non-overlapping walks of level $1$ in $M_1$, then $N-2$
non-overlapping walks of level $2$ in $M_2$ in each of the above walks, etc.
Let us denote the obtained system at step $k$ by
\[
\mathcal{G}_k = \left\{ \Gamma_{i_1,\dots ,i_k} : (i_1,\dots,i_k)
\in \{1, \dots, N-2\}^{k} \right\},
\]
where $\Gamma_{i_1,\dots ,i_k}$ is the union of the squares of the
corresponding walk. (Set $\mathcal{G}_0 = \{ \Gamma_\emptyset \} =
\{ [0,1]^2 \}$.) Let us also put
\[
C_{k}= \left\{ y \in [0,1] : (0,y) \in \bigcup
\mathcal{G}_{k}\right\}
\]
and define
\[
C = \bigcap_{k\in \mathbb{N}} C_k.
\]
Then clearly $C$ is an $(N-2,N)$-regular compact set, therefore Fact~\ref{LN} yields that $\dim_H C = \frac{\log (N-2)}{\log N} > d-1$. As $C\setminus \mathbb{Q}\neq \emptyset$,
we have $\dim_H (C \setminus \mathbb{Q})=\dim_H C> d-1$.

For every $y\in C \setminus \mathbb{Q}$ and every $k \in \mathbb{N}$
there is a unique $(i_1, \dots, i_k)$ such that $(0,y)\in \Gamma_{
i_1, \dots, i_k }$. (For a $y$ of the form $\frac{i}{N^l}$ there may
be two such $(i_1, \dots, i_k)$, and we would like to avoid this
complication.) Put $\Gamma_k(y) = \Gamma_{ i_1, \dots, i_k }$ and
$\Gamma(y) = \bigcap_{k=1}^{\infty} \Gamma_{k}(y)$. Since
$\Gamma(y)$ is a decreasing intersection of compact connected sets,
it is itself connected (\cite{Eng}). (Actually, it is a continuous
curve, but we will not need this here.) It is also easy to see that
it intersects $\{0\}\times [0,1]$ and $\{1\}\times [0,1]$.

We can now define
\[
\mathcal{G} = \{ \Gamma(y) :  y \in C \setminus \mathbb{Q} \}.
\]

Next we prove that $\mathcal{G}$ consists of disjoint sets. Let
$y,y'\in C\setminus \mathbb{Q}$ be distinct. Pick $l\in \mathbb{N}$
so large such that $|y-y'|>\frac{6}{N^{l}}$. Then there are at least
$5$ intervals of level $l$ between $y$ and $y'$. Since we always
chose $N-2$ intervals out of $N$ along the construction, there can
be at most $4$ consecutive non-selected intervals, therefore there
is a $\Gamma_{ i_1, \dots, i_l }$ separating $y$ and $y'$. But then
this also separates $\Gamma_l(y)$ and $\Gamma_l(y')$, hence
$\Gamma(y)$ and $\Gamma(y')$ are disjoint.

Now we check that for every $y\in C\setminus\mathbb{Q}$ and every
countable open basis $\mathcal{U}$ of $M$ the set $\bigcup_{U\in \mathcal{U}}
\partial_M U$ intersects $\Gamma(y)$. Let $z_0\in \Gamma(y)$ and
$U_0\in \mathcal{U}$ be such that $z_0\in U_0$ and $\Gamma(y)
\nsubseteq U_0$. Then $\partial_M U_0$ must intersect $\Gamma(y)$,
since otherwise $\Gamma(y)=(\Gamma(y)\cap U_0)\cup (\Gamma(y)\cap
\inter_M (M \setminus U_0))$, hence a connected set would be the
union of two non-empty disjoint relatively open sets, a
contradiction.

Thus, as explained in the first paragraph of the proof, it is sufficient to
prove that if a set $Z$ intersects every $\Gamma(y)$ then $\dim_H Z > d-1$.
This is easily seen to hold if we can construct an onto Lipschitz map
\[
\varphi \colon  \bigcup \mathcal{G} \to C \setminus \mathbb{Q}
\]
that is constant on every member of $\mathcal{G}$, since Lipschitz
maps do not increase Hausdorff dimension, and $\dim_H (C \setminus
\mathbb{Q}) > d - 1$. Define
\[
\varphi (z) = y \textrm{ if } z \in \Gamma(y),
\]
which is well-defined by the disjointness of the members of
$\mathcal{G}$.

Let us now prove that this map is Lipschitz. Let $y, y' \in C
\setminus \mathbb{Q}$, $z \in \Gamma(y)$, and $z' \in \Gamma(y')$.
Choose $l \in \mathbb{N}^{+}$ such that $\frac{1}{N^l} < |y-y'| \leq
\frac{1}{N^{l-1}}$. Then using $N \ge 6$ we obtain $|y-y'| >
\frac{6}{N^{l+1}}$, thus, as above, there is a walk of level $l+1$
separating $z$ and $z'$. Therefore $|z - z'| \geq \frac
{1}{N^{l+1}}$, and hence
\[
|\varphi(z) - \varphi(z')| = |y-y'| \leq \frac{1}{N^{l-1}} = N^{2} \frac
{1}{N^{l+1}} \leq N^{2} |z - z'|,
\]
therefore  $\varphi$ is Lipschitz with Lipschitz constant at most $N^{2}$.

To finish the proof, let $n$ be given as in Theorem~\ref{t:lower}
and pick $k\in \mathbb{N}$ so large that $N = n^{2^{k}}$ satisfies
(\ref{e:Nbig}). If $p > p^{(N)}$ then using Lemma~\ref{lfull} we
deduce that
\[
P \left( \dim_{tH} M^{(p,N)} > d \right)
\geq P \left( M^{(p,N)} \textrm{ is full } \right) > 0,
\]
which implies $p_c^{(d,N)} < 1$. Iterating  $k$ times  Lemma~\ref{ln2} we
infer $p_c^{(d,n)} < 1$.

Now, if $p > p_c^{(d,n)}$ then
\[
P \left( \dim_{tH} M^{(p,n)} > d \, \Big| \,
M^{(p,n)} \neq \emptyset \right) \geq P (\dim_{tH} M^{(p,n)} > d )>0.
\]
Combining this with Lemma~\ref{thomog} we deduce that
\[
P\left(\dim_{tH} M^{(p,n)} > d \, \Big| \, M^{(p,n)} \neq \emptyset \right) = 1,
\]
which completes the proof of the theorem.
\end{proof}

\begin{remark}
It is well-known and not difficult to see that $\mathop{\lim}\limits_{p\to 1} P(
M^{(p,n)} = \emptyset ) = 0$. Using this it is an easy consequence of the
previous theorem that for every integer $n > 1$, $d < 2$ and $\varepsilon > 0$
there exists a $\delta = \delta^{(n, d, \varepsilon)} > 0$ such that for all $p
> 1 - \delta$
\[
P\left(\dim_{tH} M^{(p,n)} > d \right)>1-\varepsilon.
\]
\end{remark}

\subsection{The upper estimate of $\dim_{tH} M$}

The argument of this subsection will rely on some ideas from \cite{CH}.

\begin{theorem}
\label{upper}
If $p> \frac{1}{\sqrt{n}}$ then almost surely
\[
\dim_{tH} M \leq 2 + 2 \frac{\log p}{\log n}.
\]
\end{theorem}

\begin{proof}
A segment is called a \emph{basic segment} if it is of the form
$\left[\frac{i-1}{n^{k}},\frac{i}{n^{k}}\right] \times
\{\frac{j}{n^k}\}$ or $\{\frac{j}{n^k} \} \times
\left[\frac{i-1}{n^{k}},\frac{i}{n^{k}}\right]$, where $k\in
\mathbb{N}^{+}$, $i \in\{ 1,...,n^{k} \}$ and $j \in\{ 1,...,n^{k} -
1 \}$.

It suffices to show that for every basic segment $S$ and for every
$\varepsilon>0$ there exists (almost surely, a random)  arc $\gamma
\subseteq [0,1]^2$ connecting the endpoints of $S$ in the
$\varepsilon$-neighborhood of $S$ such that $\dim_H \left( M \cap
\gamma \right) \le 1+2\frac{\log p}{\log n}$. Indeed, we can almost
surely construct the analogous arcs for all basic segments, and
hence obtain a basis of $M$ consisting of `approximate squares'
whose boundaries are of Hausdorff dimension at most $1+2\frac{\log
p}{\log n}$, therefore $\dim_{tH} M \leq 2+2\frac{\log p}{\log n}$
almost surely.

\placedrawing[ht]{
\begin{picture}(100,42)
\thinlines
\drawpath{410.0}{58.0}{446.0}{58.0}
\drawpath{418.0}{58.0}{418.0}{94.0}
\drawpath{390.0}{94.0}{354.0}{94.0}
\drawpath{328.0}{94.0}{328.0}{58.0}
\thicklines
\drawpath{302.0}{26.0}{316.0}{26.0}
\thinlines
\drawshadebox{278.0}{280.0}{282.0}{28.0}{30.0}{32.0}{}{0.47}
\drawshadebox{258.0}{260.0}{262.0}{28.0}{30.0}{32.0}{}{0.47}
\thicklines
\drawthickdot{34.0}{26.0}
\thinlines
\drawpath{34.0}{46.0}{34.0}{4.0}
\drawpath{70.0}{46.0}{70.0}{4.0}
\drawpath{70.0}{26.0}{34.0}{26.0}
\drawcenteredtext{78.0}{26.0}{$(\frac{i}{n^k},\frac{j}{n^k})$}
\drawshadebox{34.0}{36.0}{38.0}{26.0}{28.0}{30.0}{}{0.47}
\drawshadebox{38.0}{40.0}{42.0}{26.0}{28.0}{30.0}{}{0.47}
\drawshadebox{38.0}{40.0}{42.0}{22.0}{24.0}{26.0}{}{0.47}
\drawshadebox{58.0}{60.0}{62.0}{26.0}{28.0}{30.0}{}{0.47}
\drawshadebox{58.0}{60.0}{62.0}{22.0}{24.0}{26.0}{}{0.47}
\thicklines
\path(58.0,26.0)(58.0,26.0)(57.86,26.06)(57.75,26.15)(57.62,26.22)(57.5,26.29)(57.4,26.38)(57.26,26.43)(57.16,26.52)(57.04,26.58)
\path(57.04,26.58)(56.91,26.65)(56.8,26.72)(56.68,26.77)(56.55,26.83)(56.44,26.9)(56.3,26.95)(56.19,27.02)(56.08,27.06)(55.94,27.11)
\path(55.94,27.11)(55.83,27.18)(55.72,27.22)(55.58,27.27)(55.47,27.31)(55.36,27.36)(55.22,27.4)(55.11,27.43)(55.0,27.5)(54.86,27.52)
\path(54.86,27.52)(54.75,27.56)(54.62,27.61)(54.5,27.63)(54.4,27.68)(54.26,27.7)(54.16,27.74)(54.04,27.75)(53.9,27.79)(53.79,27.79)
\path(53.79,27.79)(53.68,27.83)(53.54,27.86)(53.43,27.88)(53.29,27.9)(53.18,27.9)(53.08,27.93)(52.93,27.93)(52.83,27.95)(52.72,27.97)
\path(52.72,27.97)(52.58,27.97)(52.47,27.97)(52.36,27.99)(52.24,27.99)(52.11,27.99)(52.0,28.0)(51.88,27.99)(51.75,27.99)(51.63,27.99)
\path(51.63,27.99)(51.52,27.97)(51.4,27.97)(51.27,27.97)(51.15,27.95)(51.04,27.93)(50.9,27.93)(50.79,27.9)(50.68,27.9)(50.54,27.88)
\path(50.54,27.88)(50.43,27.86)(50.31,27.83)(50.18,27.79)(50.08,27.79)(49.95,27.75)(49.83,27.74)(49.72,27.7)(49.59,27.68)(49.47,27.63)
\path(49.47,27.63)(49.36,27.61)(49.24,27.56)(49.11,27.52)(49.0,27.5)(48.88,27.43)(48.75,27.4)(48.63,27.36)(48.52,27.31)(48.4,27.27)
\path(48.4,27.27)(48.27,27.22)(48.15,27.18)(48.04,27.11)(47.9,27.06)(47.79,27.02)(47.68,26.95)(47.54,26.9)(47.43,26.83)(47.31,26.77)
\path(47.31,26.77)(47.18,26.72)(47.08,26.65)(46.95,26.58)(46.83,26.52)(46.72,26.43)(46.59,26.38)(46.47,26.29)(46.36,26.22)(46.22,26.15)
\path(46.22,26.15)(46.11,26.06)(46.0,26.0)(46.0,26.0)
\path(42.0,26.0)(42.0,26.0)(42.04,25.95)(42.08,25.9)(42.11,25.88)(42.15,25.84)(42.18,25.81)(42.22,25.77)(42.27,25.72)(42.29,25.7)
\path(42.29,25.7)(42.36,25.65)(42.4,25.63)(42.43,25.59)(42.47,25.56)(42.52,25.54)(42.54,25.5)(42.58,25.49)(42.61,25.45)(42.68,25.43)
\path(42.68,25.43)(42.72,25.4)(42.75,25.38)(42.79,25.36)(42.83,25.33)(42.86,25.31)(42.9,25.29)(42.93,25.27)(43.0,25.25)(43.04,25.22)
\path(43.04,25.22)(43.08,25.2)(43.11,25.18)(43.15,25.15)(43.18,25.15)(43.22,25.13)(43.27,25.11)(43.29,25.11)(43.36,25.09)(43.4,25.09)
\path(43.4,25.09)(43.43,25.06)(43.47,25.06)(43.52,25.04)(43.54,25.04)(43.58,25.04)(43.61,25.02)(43.68,25.02)(43.72,25.0)(43.75,25.0)
\path(43.75,25.0)(43.79,25.0)(43.83,25.0)(43.86,25.0)(43.9,25.0)(43.93,25.0)(44.0,25.0)(44.04,25.0)(44.06,25.0)(44.11,25.0)
\path(44.11,25.0)(44.15,25.0)(44.18,25.0)(44.22,25.0)(44.27,25.0)(44.29,25.02)(44.36,25.02)(44.38,25.04)(44.43,25.04)(44.47,25.04)
\path(44.47,25.04)(44.52,25.06)(44.54,25.06)(44.58,25.09)(44.61,25.09)(44.68,25.11)(44.7,25.11)(44.75,25.13)(44.79,25.15)(44.83,25.15)
\path(44.83,25.15)(44.86,25.18)(44.9,25.2)(44.93,25.22)(45.0,25.25)(45.02,25.27)(45.06,25.29)(45.11,25.31)(45.15,25.33)(45.18,25.34)
\path(45.18,25.34)(45.22,25.38)(45.27,25.4)(45.29,25.43)(45.34,25.45)(45.38,25.49)(45.43,25.5)(45.47,25.54)(45.5,25.56)(45.54,25.59)
\path(45.54,25.59)(45.58,25.63)(45.61,25.65)(45.65,25.7)(45.7,25.72)(45.75,25.77)(45.79,25.79)(45.83,25.84)(45.86,25.88)(45.9,25.9)
\path(45.9,25.9)(45.93,25.95)(45.97,26.0)(46.0,26.0)
\path(34.0,26.0)(34.0,26.0)(34.04,25.95)(34.08,25.9)(34.11,25.88)(34.15,25.84)(34.18,25.81)(34.22,25.77)(34.27,25.72)(34.29,25.7)
\path(34.29,25.7)(34.36,25.65)(34.4,25.63)(34.43,25.59)(34.47,25.56)(34.52,25.54)(34.54,25.5)(34.58,25.49)(34.61,25.45)(34.68,25.43)
\path(34.68,25.43)(34.72,25.4)(34.75,25.38)(34.79,25.36)(34.83,25.33)(34.86,25.31)(34.9,25.29)(34.93,25.27)(35.0,25.25)(35.04,25.22)
\path(35.04,25.22)(35.08,25.2)(35.11,25.18)(35.15,25.15)(35.18,25.15)(35.22,25.13)(35.27,25.11)(35.29,25.11)(35.36,25.09)(35.4,25.09)
\path(35.4,25.09)(35.43,25.06)(35.47,25.06)(35.52,25.04)(35.54,25.04)(35.58,25.04)(35.61,25.02)(35.68,25.02)(35.72,25.0)(35.75,25.0)
\path(35.75,25.0)(35.79,25.0)(35.83,25.0)(35.86,25.0)(35.9,25.0)(35.93,25.0)(36.0,25.0)(36.04,25.0)(36.06,25.0)(36.11,25.0)
\path(36.11,25.0)(36.15,25.0)(36.18,25.0)(36.22,25.0)(36.27,25.0)(36.29,25.02)(36.36,25.02)(36.38,25.04)(36.43,25.04)(36.47,25.04)
\path(36.47,25.04)(36.52,25.06)(36.54,25.06)(36.58,25.09)(36.61,25.09)(36.68,25.11)(36.7,25.11)(36.75,25.13)(36.79,25.15)(36.83,25.15)
\path(36.83,25.15)(36.86,25.18)(36.9,25.2)(36.93,25.22)(37.0,25.25)(37.02,25.27)(37.06,25.29)(37.11,25.31)(37.15,25.33)(37.18,25.34)
\path(37.18,25.34)(37.22,25.38)(37.27,25.4)(37.29,25.43)(37.34,25.45)(37.38,25.49)(37.43,25.5)(37.47,25.54)(37.5,25.56)(37.54,25.59)
\path(37.54,25.59)(37.58,25.63)(37.61,25.65)(37.65,25.7)(37.7,25.72)(37.75,25.77)(37.79,25.79)(37.83,25.84)(37.86,25.88)(37.9,25.9)
\path(37.9,25.9)(37.93,25.95)(37.97,26.0)(38.0,26.0)
\thinlines
\drawcenteredtext{26.0}{26.0}{$(\frac{i-1}{n^k},\frac{j}{n^k})$}
\drawthickdot{70.0}{26.0}
\drawshadebox{62.0}{64.0}{66.0}{26.0}{28.0}{30.0}{}{0.47}
\drawshadebox{66.0}{68.0}{70.0}{22.0}{24.0}{26.0}{}{0.47}
\thicklines
\path(62.0,26.0)(62.0,26.0)(62.04,25.95)(62.08,25.9)(62.11,25.88)(62.16,25.84)(62.19,25.81)(62.22,25.77)(62.26,25.72)(62.3,25.7)
\path(62.3,25.7)(62.36,25.65)(62.37,25.63)(62.44,25.59)(62.47,25.56)(62.5,25.54)(62.55,25.5)(62.58,25.49)(62.62,25.45)(62.66,25.43)
\path(62.66,25.43)(62.72,25.4)(62.75,25.38)(62.8,25.36)(62.83,25.33)(62.86,25.31)(62.91,25.29)(62.94,25.27)(63.0,25.25)(63.04,25.22)
\path(63.04,25.22)(63.08,25.2)(63.11,25.18)(63.16,25.15)(63.19,25.15)(63.22,25.13)(63.26,25.11)(63.3,25.11)(63.36,25.09)(63.37,25.09)
\path(63.37,25.09)(63.44,25.06)(63.47,25.06)(63.5,25.04)(63.55,25.04)(63.58,25.04)(63.62,25.02)(63.66,25.02)(63.72,25.0)(63.75,25.0)
\path(63.75,25.0)(63.8,25.0)(63.83,25.0)(63.86,25.0)(63.91,25.0)(63.94,25.0)(64.0,25.0)(64.04,25.0)(64.08,25.0)(64.11,25.0)
\path(64.11,25.0)(64.12,25.0)(64.19,25.0)(64.22,25.0)(64.26,25.0)(64.3,25.02)(64.36,25.02)(64.37,25.04)(64.44,25.04)(64.47,25.04)
\path(64.47,25.04)(64.5,25.06)(64.55,25.06)(64.58,25.09)(64.62,25.09)(64.66,25.11)(64.72,25.11)(64.75,25.13)(64.79,25.15)(64.83,25.15)
\path(64.83,25.15)(64.86,25.18)(64.91,25.2)(64.94,25.22)(65.0,25.25)(65.04,25.27)(65.08,25.29)(65.11,25.31)(65.12,25.33)(65.19,25.34)
\path(65.19,25.34)(65.22,25.38)(65.26,25.4)(65.3,25.43)(65.36,25.45)(65.37,25.49)(65.41,25.5)(65.47,25.54)(65.5,25.56)(65.55,25.59)
\path(65.55,25.59)(65.58,25.63)(65.62,25.65)(65.66,25.7)(65.72,25.72)(65.75,25.77)(65.79,25.79)(65.83,25.84)(65.86,25.88)(65.91,25.9)
\path(65.91,25.9)(65.94,25.95)(66.0,26.0)(66.0,26.0)
\path(66.0,26.0)(66.0,26.0)(66.04,26.02)(66.08,26.06)(66.11,26.11)(66.16,26.15)(66.19,26.18)(66.22,26.22)(66.26,26.25)(66.3,26.29)
\path(66.3,26.29)(66.36,26.31)(66.37,26.36)(66.44,26.38)(66.47,26.4)(66.5,26.43)(66.55,26.47)(66.58,26.5)(66.62,26.52)(66.66,26.54)
\path(66.66,26.54)(66.72,26.58)(66.75,26.61)(66.8,26.63)(66.83,26.65)(66.86,26.68)(66.91,26.68)(66.94,26.72)(67.0,26.75)(67.04,26.75)
\path(67.04,26.75)(67.08,26.77)(67.11,26.79)(67.16,26.81)(67.19,26.83)(67.22,26.84)(67.26,26.86)(67.3,26.88)(67.36,26.88)(67.37,26.9)
\path(67.37,26.9)(67.44,26.9)(67.47,26.93)(67.5,26.93)(67.55,26.93)(67.58,26.93)(67.62,26.95)(67.66,26.97)(67.72,26.97)(67.75,26.97)
\path(67.75,26.97)(67.8,26.99)(67.83,26.99)(67.86,26.99)(67.91,26.99)(67.94,26.99)(68.0,27.0)(68.04,26.99)(68.08,26.99)(68.11,26.99)
\path(68.11,26.99)(68.12,26.99)(68.19,26.99)(68.22,26.97)(68.26,26.97)(68.3,26.97)(68.36,26.95)(68.37,26.93)(68.44,26.93)(68.47,26.93)
\path(68.47,26.93)(68.5,26.93)(68.55,26.9)(68.58,26.9)(68.62,26.88)(68.66,26.88)(68.72,26.86)(68.75,26.84)(68.79,26.83)(68.83,26.81)
\path(68.83,26.81)(68.86,26.79)(68.91,26.77)(68.94,26.75)(69.0,26.75)(69.04,26.72)(69.08,26.68)(69.11,26.68)(69.12,26.65)(69.19,26.63)
\path(69.19,26.63)(69.22,26.61)(69.26,26.58)(69.3,26.54)(69.36,26.52)(69.37,26.5)(69.41,26.47)(69.47,26.43)(69.5,26.4)(69.55,26.38)
\path(69.55,26.38)(69.58,26.36)(69.62,26.31)(69.66,26.29)(69.72,26.25)(69.75,26.22)(69.79,26.18)(69.83,26.15)(69.86,26.11)(69.91,26.06)
\path(69.91,26.06)(69.94,26.02)(70.0,26.0)(70.0,26.0)
\drawpath{38.0}{26.0}{42.0}{26.0}
\drawpath{58.0}{26.0}{62.0}{26.0}
\drawpath{38.0}{26.0}{42.0}{26.0}
\drawshadebox{34.0}{36.0}{38.0}{30.0}{32.0}{34.0}{}{0.2}
\drawshadebox{34.0}{36.0}{38.0}{34.0}{36.0}{38.0}{}{0.2}
\drawshadebox{38.0}{40.0}{42.0}{34.0}{36.0}{38.0}{}{0.2}
\drawshadebox{38.0}{40.0}{42.0}{30.0}{32.0}{34.0}{}{0.2}
\drawshadebox{42.0}{44.0}{46.0}{34.0}{36.0}{38.0}{}{0.2}
\drawshadebox{42.0}{44.0}{46.0}{30.0}{32.0}{34.0}{}{0.2}
\drawshadebox{42.0}{44.0}{46.0}{26.0}{28.0}{30.0}{}{0.2}
\drawshadebox{34.0}{36.0}{38.0}{22.0}{24.0}{26.0}{}{0.2}
\drawshadebox{42.0}{44.0}{46.0}{22.0}{24.0}{26.0}{}{0.2}
\drawshadebox{34.0}{36.0}{38.0}{18.0}{20.0}{22.0}{}{0.2}
\drawshadebox{38.0}{40.0}{42.0}{18.0}{20.0}{22.0}{}{0.2}
\drawshadebox{42.0}{44.0}{46.0}{18.0}{20.0}{22.0}{}{0.2}
\drawshadebox{34.0}{36.0}{38.0}{14.0}{16.0}{18.0}{}{0.2}
\drawshadebox{38.0}{40.0}{42.0}{14.0}{16.0}{18.0}{}{0.2}
\drawshadebox{42.0}{44.0}{46.0}{14.0}{16.0}{18.0}{}{0.2}
\drawshadebox{58.0}{60.0}{62.0}{30.0}{32.0}{34.0}{}{0.2}
\drawshadebox{58.0}{60.0}{62.0}{34.0}{36.0}{38.0}{}{0.2}
\drawshadebox{62.0}{64.0}{66.0}{34.0}{36.0}{38.0}{}{0.2}
\drawshadebox{62.0}{64.0}{66.0}{30.0}{32.0}{34.0}{}{0.2}
\drawshadebox{66.0}{68.0}{70.0}{34.0}{36.0}{38.0}{}{0.2}
\drawshadebox{66.0}{68.0}{70.0}{30.0}{32.0}{34.0}{}{0.2}
\drawshadebox{66.0}{68.0}{70.0}{26.0}{28.0}{30.0}{}{0.2}
\drawshadebox{62.0}{64.0}{66.0}{22.0}{24.0}{26.0}{}{0.2}
\drawshadebox{58.0}{60.0}{62.0}{18.0}{20.0}{22.0}{}{0.2}
\drawshadebox{62.0}{64.0}{66.0}{18.0}{20.0}{22.0}{}{0.2}
\drawshadebox{66.0}{68.0}{70.0}{18.0}{20.0}{22.0}{}{0.2}
\drawshadebox{58.0}{60.0}{62.0}{14.0}{16.0}{18.0}{}{0.2}
\drawshadebox{62.0}{64.0}{66.0}{14.0}{16.0}{18.0}{}{0.2}
\drawshadebox{66.0}{68.0}{70.0}{14.0}{16.0}{18.0}{}{0.2}
\drawthickdot{34.0}{26.0}
\drawthickdot{70.0}{26.0}
\path(34.0,26.0)(34.0,26.0)(34.04,25.95)(34.08,25.9)(34.11,25.88)(34.15,25.84)(34.18,25.81)(34.22,25.77)(34.27,25.72)(34.31,25.7)
\path(34.31,25.7)(34.36,25.65)(34.4,25.63)(34.43,25.59)(34.47,25.56)(34.5,25.54)(34.54,25.5)(34.59,25.49)(34.63,25.45)(34.68,25.43)
\path(34.68,25.43)(34.72,25.4)(34.75,25.38)(34.79,25.36)(34.83,25.33)(34.86,25.31)(34.9,25.29)(34.95,25.27)(35.0,25.25)(35.04,25.22)
\path(35.04,25.22)(35.08,25.2)(35.11,25.18)(35.15,25.15)(35.18,25.15)(35.22,25.13)(35.27,25.11)(35.31,25.11)(35.36,25.09)(35.4,25.09)
\path(35.4,25.09)(35.43,25.06)(35.47,25.06)(35.5,25.04)(35.54,25.04)(35.59,25.04)(35.63,25.02)(35.68,25.02)(35.72,25.0)(35.75,25.0)
\path(35.75,25.0)(35.79,25.0)(35.83,25.0)(35.86,25.0)(35.9,25.0)(35.95,25.0)(36.0,25.0)(36.04,25.0)(36.08,25.0)(36.11,25.0)
\path(36.11,25.0)(36.15,25.0)(36.18,25.0)(36.22,25.0)(36.27,25.0)(36.31,25.02)(36.36,25.02)(36.4,25.04)(36.43,25.04)(36.47,25.04)
\path(36.47,25.04)(36.5,25.06)(36.54,25.06)(36.59,25.09)(36.63,25.09)(36.68,25.11)(36.72,25.11)(36.75,25.13)(36.79,25.15)(36.83,25.15)
\path(36.83,25.15)(36.86,25.18)(36.9,25.2)(36.95,25.22)(37.0,25.25)(37.04,25.27)(37.08,25.29)(37.11,25.31)(37.15,25.33)(37.18,25.34)
\path(37.18,25.34)(37.22,25.38)(37.27,25.4)(37.31,25.43)(37.36,25.45)(37.4,25.49)(37.43,25.5)(37.47,25.54)(37.5,25.56)(37.54,25.59)
\path(37.54,25.59)(37.59,25.63)(37.63,25.65)(37.68,25.7)(37.72,25.72)(37.75,25.77)(37.79,25.79)(37.83,25.84)(37.86,25.88)(37.9,25.9)
\path(37.9,25.9)(37.95,25.95)(38.0,26.0)(38.0,26.0)
\path(42.0,26.0)(42.0,26.0)(42.04,26.02)(42.08,26.06)(42.11,26.11)(42.15,26.15)(42.18,26.18)(42.22,26.22)(42.27,26.25)(42.31,26.29)
\path(42.31,26.29)(42.36,26.31)(42.4,26.36)(42.43,26.38)(42.47,26.4)(42.5,26.45)(42.54,26.47)(42.59,26.5)(42.63,26.52)(42.68,26.56)
\path(42.68,26.56)(42.72,26.59)(42.75,26.61)(42.79,26.63)(42.83,26.65)(42.86,26.68)(42.9,26.7)(42.95,26.72)(43.0,26.75)(43.04,26.75)
\path(43.04,26.75)(43.08,26.77)(43.11,26.79)(43.15,26.81)(43.18,26.84)(43.22,26.84)(43.27,26.86)(43.31,26.88)(43.36,26.88)(43.4,26.9)
\path(43.4,26.9)(43.43,26.9)(43.47,26.93)(43.5,26.93)(43.54,26.95)(43.59,26.95)(43.63,26.95)(43.68,26.97)(43.72,26.97)(43.75,26.97)
\path(43.75,26.97)(43.79,26.99)(43.83,26.99)(43.86,26.99)(43.9,26.99)(43.95,26.99)(44.0,27.0)(44.04,26.99)(44.08,26.99)(44.11,26.99)
\path(44.11,26.99)(44.15,26.99)(44.18,26.99)(44.22,26.97)(44.27,26.97)(44.31,26.97)(44.36,26.95)(44.4,26.95)(44.43,26.95)(44.47,26.93)
\path(44.47,26.93)(44.5,26.93)(44.54,26.9)(44.59,26.9)(44.63,26.88)(44.68,26.88)(44.72,26.86)(44.75,26.84)(44.79,26.84)(44.83,26.81)
\path(44.83,26.81)(44.86,26.79)(44.9,26.77)(44.95,26.75)(45.0,26.75)(45.04,26.72)(45.08,26.7)(45.11,26.68)(45.15,26.65)(45.18,26.63)
\path(45.18,26.63)(45.22,26.61)(45.27,26.59)(45.31,26.56)(45.36,26.52)(45.4,26.5)(45.43,26.47)(45.47,26.45)(45.5,26.4)(45.54,26.38)
\path(45.54,26.38)(45.59,26.36)(45.63,26.31)(45.68,26.29)(45.72,26.25)(45.75,26.22)(45.79,26.18)(45.83,26.15)(45.86,26.11)(45.9,26.06)
\path(45.9,26.06)(45.95,26.02)(46.0,26.0)(46.0,26.0)
\path(62.0,26.0)(62.0,26.0)(62.01,25.95)(62.08,25.9)(62.11,25.88)(62.16,25.84)(62.19,25.81)(62.23,25.77)(62.26,25.72)(62.3,25.7)
\path(62.3,25.7)(62.36,25.65)(62.37,25.63)(62.44,25.59)(62.47,25.56)(62.51,25.54)(62.55,25.5)(62.58,25.49)(62.62,25.45)(62.66,25.43)
\path(62.66,25.43)(62.72,25.4)(62.75,25.38)(62.8,25.36)(62.83,25.33)(62.87,25.31)(62.91,25.29)(62.94,25.27)(63.0,25.25)(63.01,25.22)
\path(63.01,25.22)(63.08,25.2)(63.11,25.18)(63.16,25.15)(63.19,25.15)(63.23,25.13)(63.26,25.11)(63.3,25.11)(63.36,25.09)(63.37,25.09)
\path(63.37,25.09)(63.44,25.06)(63.47,25.06)(63.51,25.04)(63.55,25.04)(63.58,25.04)(63.62,25.02)(63.66,25.02)(63.72,25.0)(63.75,25.0)
\path(63.75,25.0)(63.8,25.0)(63.83,25.0)(63.87,25.0)(63.91,25.0)(63.94,25.0)(64.0,25.0)(64.01,25.0)(64.08,25.0)(64.11,25.0)
\path(64.11,25.0)(64.16,25.0)(64.19,25.0)(64.23,25.0)(64.26,25.0)(64.3,25.02)(64.36,25.02)(64.37,25.04)(64.44,25.04)(64.47,25.04)
\path(64.47,25.04)(64.51,25.06)(64.55,25.06)(64.58,25.09)(64.62,25.09)(64.66,25.11)(64.72,25.11)(64.75,25.13)(64.8,25.15)(64.83,25.15)
\path(64.83,25.15)(64.87,25.18)(64.91,25.2)(64.94,25.22)(65.0,25.25)(65.01,25.27)(65.08,25.29)(65.11,25.31)(65.16,25.33)(65.19,25.34)
\path(65.19,25.34)(65.23,25.38)(65.26,25.4)(65.3,25.43)(65.36,25.45)(65.37,25.49)(65.44,25.5)(65.47,25.54)(65.51,25.56)(65.55,25.59)
\path(65.55,25.59)(65.58,25.63)(65.62,25.65)(65.66,25.7)(65.72,25.72)(65.75,25.77)(65.8,25.79)(65.83,25.84)(65.87,25.88)(65.91,25.9)
\path(65.91,25.9)(65.94,25.95)(66.0,26.0)(66.0,26.0)
\path(66.0,26.0)(66.0,26.0)(66.05,26.02)(66.15,26.06)(66.23,26.11)(66.3,26.15)(66.37,26.18)(66.44,26.22)(66.54,26.25)(66.61,26.29)
\path(66.61,26.29)(66.68,26.31)(66.76,26.36)(66.83,26.38)(66.9,26.4)(66.97,26.45)(67.04,26.47)(67.11,26.5)(67.16,26.52)(67.23,26.56)
\path(67.23,26.56)(67.3,26.59)(67.37,26.61)(67.44,26.63)(67.5,26.65)(67.55,26.68)(67.62,26.7)(67.68,26.72)(67.75,26.75)(67.8,26.75)
\path(67.8,26.75)(67.86,26.77)(67.91,26.79)(67.98,26.81)(68.04,26.84)(68.08,26.84)(68.15,26.86)(68.19,26.88)(68.25,26.88)(68.3,26.9)
\path(68.3,26.9)(68.36,26.9)(68.41,26.93)(68.44,26.93)(68.51,26.95)(68.55,26.95)(68.58,26.95)(68.65,26.97)(68.69,26.97)(68.73,26.97)
\path(68.73,26.97)(68.79,26.99)(68.83,26.99)(68.87,26.99)(68.91,26.99)(68.94,26.99)(69.0,27.0)(69.01,26.99)(69.05,26.99)(69.11,26.99)
\path(69.11,26.99)(69.15,26.99)(69.18,26.99)(69.22,26.97)(69.26,26.97)(69.29,26.97)(69.3,26.95)(69.36,26.95)(69.37,26.95)(69.41,26.93)
\path(69.41,26.93)(69.44,26.93)(69.48,26.9)(69.51,26.9)(69.51,26.88)(69.55,26.88)(69.58,26.86)(69.61,26.84)(69.62,26.84)(69.66,26.81)
\path(69.66,26.81)(69.68,26.79)(69.69,26.77)(69.72,26.75)(69.75,26.75)(69.76,26.72)(69.76,26.7)(69.8,26.68)(69.8,26.65)(69.83,26.63)
\path(69.83,26.63)(69.83,26.61)(69.87,26.59)(69.87,26.56)(69.87,26.52)(69.91,26.5)(69.91,26.47)(69.93,26.45)(69.94,26.4)(69.94,26.38)
\path(69.94,26.38)(69.94,26.36)(69.94,26.31)(69.97,26.29)(69.98,26.25)(69.98,26.22)(69.98,26.18)(69.98,26.15)(69.98,26.11)(69.98,26.06)
\path(69.98,26.06)(69.98,26.02)(70.0,26.0)(70.0,26.0)
\drawpath{58.0}{26.0}{62.0}{26.0}
\drawpath{62.0}{26.0}{58.0}{26.0}
\end{picture}
}
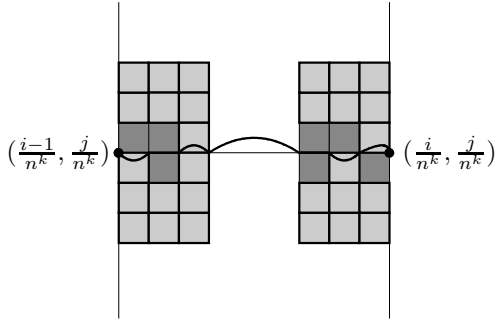
{Construction of the arc $\gamma$
connecting the endpoints of $S$}{fig:percol2}

Let us now construct such an arc $\gamma$ for $S$ and $\varepsilon >
0$. We may assume that $S$ is horizontal, hence it is of the form $S
= \left[\frac{i-1}{n^{k}},\frac{i}{n^{k}}\right] \times
\{\frac{j}{n^k}\}$ for some $k\in \mathbb{N}^{+}$, $i \in\{
1,...,n^{k} \}$ and $j \in\{ 1,...,n^{k} - 1 \}$.

We divide $S$ into $n$ subsegments of length $\frac{1}{n^{k+1}}$,
and we call a subsegment
$\left[\frac{m-1}{n^{k+1}},\frac{m}{n^{k+1}}\right] \times
\{\frac{j}{n^k} \}$ \emph{bad} if both the adjacent squares
$\left[\frac{m-1}{n^{k+1}},\frac{m}{n^{k+1}}\right] \times [
\frac{j}{n^k} - \frac{1}{n^{k+1}}, \frac{j}{n^k}]$ and
$\left[\frac{m-1}{n^{k+1}},\frac{m}{n^{k+1}}\right] \times
[\frac{j}{n^k}, \frac{j}{n^k} + \frac{1}{n^{k+1}}]$ are in
$M_{k+1}$. Otherwise we say that the subsegment is \emph{good}. Let
$B_1$ denote the union of the bad segments. Then inside every bad
segment we repeat the same procedure, and obtain $B_2$ and so on. It
is easy to see that this process is (a scaled copy of) the
1-dimensional fractal percolation with $p$ replaced by $p^2$. Let $B
= \bigcap_l B_l$ be its limit set. Then by Remark~\ref{r:ketto}
(note that $p^2 > \frac1n$) we obtain $\dim_H B = 1 + \frac{\log
p^2}{\log n} = 1 + 2 \frac{\log p}{\log n}$ or $B = \emptyset$
almost surely. So it suffices to construct a $\gamma$ connecting the
endpoints of $S$ in the $\varepsilon$-neighborhood of $S$ such that
$\gamma \cap M = B$ (except perhaps some endpoints, but all the
endpoints form a countable set, hence a set of Hausdorff dimension
0). (See Figure~\ref{fig:percol2}.)

But this is easily done. Indeed, for every good subsegment $I$ let
$\gamma_I$ be an arc connecting the endpoints of $I$ in a small neighborhood
of $I$ such that $\gamma$ is
disjoint from $M$ apart from the endpoints (this is possible, since either the
top or the bottom square was erased from $M$). Then $\gamma = \left( \bigcup_{I \textrm{
 is good}} \ \gamma_I \right) \cup B$ works.
\end{proof}

Using Remarks~\ref{r:egy} and \ref{r:ketto} this easily implies

\begin{corollary}
\label{c:tH<H}
Almost surely
\[
\dim_{tH} M < \dim_{H} M \textrm{ or } M = \emptyset.
\]
\end{corollary}

\begin{remark}
Calculating the exact value of $\dim_{tH} M$ seems to be difficult, since it would provide
the value of the critical probability $p_c$ of Chayes, Chayes and
Durrett (where the phase transition occurs, see above), and this is a
long-standing open problem.
\end{remark}

\section{Application II: The Hausdorff dimension of the level sets of
the generic continuous function}\label{s:application}

Now we return to Problem \ref{p:Buczo}.  The main goal is to find
analogues to Kirchheim's theorem, that is, to determine the Hausdorff
dimension of the level sets of the generic continuous function defined on a
compact metric space $K$.

Let us first note that the case $\dim_t K =0$, that is, when there is a basis
consisting of clopen sets is trivial because of the following well-known and easy
fact. For a short proof see \cite[Lemma~2.6]{BBE2}.

\begin{fact} \label{fact} If $K$ is a  compact metric space with $\dim_t K =0$
  then the generic continuous function is one-to-one on $K$.
\end{fact}

\begin{corollary}
 If $K$ is a  compact metric space with $\dim_t K = 0$
 then every non-empty level set of the generic continuous function is of Hausdorff dimension $0$.
\end{corollary}

Hence from now on we can restrict our attention to the case of positive
topological dimension.

In the first part of this section we prove Theorem~\ref{ft} and Corollary~\ref{c:ft},
 our main theorems concerning level
sets of the generic function defined on an arbitrary compact metric space,
then we use this to derive conclusions about homogeneous and self-similar
spaces in Theorem~\ref{t:hom} and Corollary~\ref{c:self}.

\subsection{Arbitrary compact metric spaces}

The goal of this subsection is to prove Theorem~\ref{ft}. In order to do this we
will need two equivalent definitions of the topological Hausdorff
dimension.

Let us fix a compact metric space $K$ with $\dim_t K > 0$, and let
$C(K)$ denote the space of continuous real-valued functions
equipped with the supremum norm. Since this is a complete
metric space, we can use Baire category arguments.

\begin{definition} \label{diml} Define
\begin{align*} P_{l}=\{&d\geq 1: \exists G\subseteq K \textrm{ such that } \dim_H G\leq d-1 \textrm{ and}  \\
&\textrm{the generic } f\in C(K) \textrm{ is one-to-one on } K\setminus G\}.
\end{align*}
\end{definition}

\begin{definition} \label{dimn} We say that a continuous function $f$ is \emph{$d$-level
narrow}, if there exists a dense set $S_{f}\subseteq \mathbb{R}$
such that $\dim_{H} f^{-1}(y)\leq d-1$ for every $y \in S_f$. Let
$\mathcal{N}_{d}$ be the set of $d$-level narrow functions. Define
$$P_{n}=\left\{ d: \mathcal{N}_{d} \textrm{ is somewhere dense in }C(K)\right\}.$$
\end{definition}

Now we repeat the definition of the topological Hausdorff dimension.

\begin{definition} \label{PtH} Let $\dim_{tH}K=\inf P_{tH}$, where
$$P_{tH}=\left\{d: K \textrm{ has a
basis } \mathcal{U} \textrm{ such that }  \dim_{H} \partial {U} \leq
d-1 \textrm{ for every } U\in \mathcal{U}   \right\}.$$
We assume that by definition $\infty \in P_n, P_l, P_{tH}$.
\end{definition}

Now we show the following theorem.

\begin{theorem} \label{t:3P}
If $K$ is a compact metric space with $\dim_{t}K>0$ then
$$P_{tH}=P_{l}=P_{n}.$$
\end{theorem}

Theorem~\ref{t:3P} and Corollary~\ref{c:inf=min} immediately yield two new equivalent
definitions for the topological Hausdorff dimension.

\begin{theorem} \label{t:3=}
If $K$ is a compact metric space with $\dim_{t}K>0$ then
$$\dim_{tH}K=\min P_{l}=\min P_{n}.$$
\end{theorem}

Before proving Theorem~\ref{t:3P} we need the following well-known lemma.
For the readers convenience we give its short proof.

\begin{lemma} \label{lem1}

Let $K_1\subseteq K_2$ be compact metric spaces and
$$R \colon C(K_2)\rightarrow
 C(K_1), \quad R(f)=f|_{K_1}.$$
If $\mathcal{F}\subseteq C(K_1)$ is co-meager then so is
$R^{-1}(\mathcal{F})\subseteq C(K_2)$.
\end{lemma}

\begin{proof}
The map $R$ is clearly continuous. Using the Tietze Extension
Theorem it is not difficult to see  that it is also open. We may
assume that $\mathcal{F}$ is a dense $G_{\delta}$ set in $C(K_1)$.
The continuity of $R$ implies that $R^{-1}(\mathcal{F})$ is also
$G_{\delta}$, thus it is enough to prove that $R^{-1}(\mathcal{F})$
is dense in $C(K_2)$. Let $\mathcal{U} \subseteq C(K_2)$ be
non-empty open, then $R(\mathcal{U}) \subseteq C(K_1)$ is also
non-empty open, hence $R(\mathcal{U}) \cap \mathcal{F} \neq
\emptyset$, and therefore $\mathcal{U} \cap R^{-1}(\mathcal{F}) \neq
\emptyset$.
\end{proof}

Next we prove Theorem~\ref{t:3P}. The proof will consist of three lemmas.

\begin{lemma} $P_{tH}\subseteq P_{l}$.
\end{lemma}

\begin{proof} Assume $d\in P_{tH}$ and $d<\infty$. Let $\mathcal{U}$ be a countable basis of $K$ such that $\dim_H \partial U\leq d-1$ for all $U\in \mathcal{U}$. Now the assumption $\dim_t K \ge 1$  and Theorem~\ref{<} yield $d\geq \dim_{tH} K\geq 1$.
Let $F=\bigcup_{U\in \mathcal{U}} \partial U$. The countable stability of Hausdorff dimension implies $\dim_H F\leq d-1$.
Then there exists a $G_{\delta}$ set $G\subseteq K$ such that $F\subseteq G$ and $\dim_H G=\dim_H F\leq d-1$. The above definitions clearly imply $\dim_t(K\setminus G)\leq \dim_t(K\setminus F)\leq 0$.

As $K\setminus G$ is $F_{\sigma}$, we can choose compact sets $K_n$ such that $K\setminus G=\bigcup_{n=1}^{\infty} K_n$ and $K_n\subseteq K_{n+1}$ for all $n\in \mathbb{N}^+$. Let $\mathcal{F}_n=\{f\in C(K_n): f \textrm{ is one-to-one}\}$ and let us define $R_n\colon C(K)\to C(K_n)$ as $R_n(f)=f|_{K_n}$ for all $n\in \mathbb{N}^+$. Since $\dim_t K_n\leq \dim_t(K\setminus G)\leq 0$, Fact~\ref{fact} implies that the sets $\mathcal{F}_n\subseteq C(K_n)$ are co-meager. Lemma~\ref{lem1} yields that $R_n^{-1}(\mathcal{F}_n)\subseteq C(K)$ are co-meager, too. As a countable intersection of co-meager sets $\mathcal{F}=\bigcap_{n=1}^{\infty} R_n^{-1}(\mathcal{F}_n)\subseteq C(K)$ is also co-meager. Clearly, every $f\in \mathcal{F}$ is one-to-one on all $K_n$, so $K_n\subseteq K_{n+1}$ $(n\in \mathbb{N}^+)$ yields that $f$ is one-to-one on $\bigcup_{n=1}^{\infty} K_n=K\setminus G$. Hence $d\in P_l$.
\end{proof}

\begin{lemma} $P_{l}\subseteq P_{n}$.
\end{lemma}

\begin{proof}  Assume $d\in P_{l}$ and $d<\infty$.
By the definition of $P_{l}$, there exists $G\subseteq K$ such that $\dim_H G\leq d-1$ and for the generic $f\in C(K)$ for all $y\in \mathbb{R}$
we have $\#(f^{-1}(y)\setminus G)\leq 1$. Then $\dim_H G\leq d-1$ and $d\geq 1$ yield $\dim_{H}f^{-1}(y)\leq d-1$, so
$\mathcal{N}_{d}$ is co-meager, thus (everywhere) dense. Hence $d\in P_{n}$.
\end{proof}

\begin{lemma}
$P_{n}\subseteq P_{tH}$.
\end{lemma}

\begin{proof}  Assume $d\in P_{n}$ and $d<\infty$.
Let us fix $x_{0}\in K$ and $r>0$. To verify
$d\in P_{tH}$ we need to find an open set $U$ such that
 $x_{0}\in U \subseteq
U(x_{0},r)$ and $\dim_{H} \partial U\leq d-1$. We may assume $\partial
U(x_{0},r) \neq \emptyset$, otherwise we are done.

By $d\in P_{n}$ we obtain that $\mathcal{N}_{d}$ is dense in a ball
$B(f_{0},6\varepsilon)$, $\varepsilon>0$. By decreasing $r$ if
necessary, we may assume that $\diam f_{0}(U(x_{0},r))\leq 3\varepsilon$. Then Tietze's
Extension Theorem provides an $f\in B(f_{0},6\varepsilon)$ such that
$f(x_{0})=f_{0}(x_{0})$ and $f|_{\partial U(x_{0},r)} (x) =
f_{0}(x_{0})+3\varepsilon$ for every $x \in
\partial U(x_{0},r)$. Since
$\mathcal{N}_{d}$ is dense in $B(f_{0},6\varepsilon)$, we can choose
$g\in \mathcal{N}_{d}$ such that $||f-g||\leq \varepsilon$.  By the
construction of $g$ it follows that $g(x_{0})<\min\{ g(\partial
U(x_{0},r))\}$. Hence  in the dense set $S_{g}$ (see Definition~\ref{dimn}) there is an $s\in S_{g}$ such that
 \begin{equation} \label{sdef} g(x_{0})<s<\min\left\{ g\left(\partial
U(x_{0},r)\right)\right\}. \end{equation}
Let
\[
U=g^{-1}\left((-\infty,s)\right)\cap U(x_{0},r),
\]
then clearly $x_{0}\in U \subseteq U(x_{0},r)$. By \eqref{sdef}
we have $\partial g^{-1}\left((-\infty,s)\right) \cap \partial U(x_{0},r) =
\emptyset$, therefore
$\partial U\subseteq \partial g^{-1}\left((-\infty,s)\right) \subseteq
g^{-1}(s)$. Using $s\in S_{g}$ we infer that $\dim_{H}\partial U\leq
\dim_{H}g^{-1}(s)\leq d-1$. \end{proof}

This concludes the proof of Theorem~\ref{t:3P}.

\bigskip

Now we are ready to describe the Hausdorff dimension of
the level sets of generic continuous functions.

As already mentioned above, if $\dim_{t}K=0$ then every level set of a
generic continuous function on $K$ consists of at most one point.

\begin{theorem} \label{ft} Let $K$ be a compact metric space with $\dim_{t}
K>0$. Then for the generic $f\in C(K)$
\begin{enumerate}[(i)]
\item $\dim_{H} f^{-1} (y)\leq \dim_{tH} K-1$ for every $y\in \mathbb{R}$,
\item for every $d<\dim_{tH} K$ there exists a non-degenerate
interval $I_{f,d}$ such that $\dim_{H} f^{-1} (y)\geq
d- 1 $ for every $y\in I_{f,d}$.
\end{enumerate}
\end{theorem}

Note that this theorem is sharp in general, see the last but one paragraph of the Introduction.
Theorem~\ref{ft} actually readily follows from the following more precise, but slightly technical version.

\begin{theorem}  Let $K$ be a compact metric space with $\dim_{t}
K>0$. Then there exists a $G_\delta$ set $G\subseteq K$ with $\dim_H G = \dim_{tH} K-1$ such that for the generic $f\in C(K)$
\begin{enumerate}[(i)]
\item $f$ is one-to-one on $K\setminus G$, hence $\dim_{H} f^{-1} (y)\leq
  \dim_{tH} K-1$ for every $y\in \mathbb{R}$,
\item for every $d<\dim_{tH} K$ there exists a non-degenerate
interval $I_{f,d}$ such that $\dim_{H} f^{-1} (y)\geq
d- 1 $ for every $y\in I_{f,d}$.
\end{enumerate}
\end{theorem}

\begin{proof} Let us first prove $(i)$. Theorem~\ref{t:3=} implies $\dim_{tH}K=\min P_l$.
Thus there exists a set $G\subseteq K$ such that
$\dim_H G=\dim_{tH} K-1$ and the generic $f\in C(K)$ is one-to-one on $K\setminus G$. By taking a
$G_\delta$ hull of the same Hausdorff dimension we can assume that $G$ is $G_\delta$.
As $\dim_{tH} K=\min P_l\geq 1$, we have $\dim_H G=\dim_{tH} K-1 \geq 0$.
Hence $\dim_H f^{-1}(y)\leq \dim_H G=\dim_{tH} K-1$ for the generic $f\in C(K)$ and for all $y\in \mathbb{R}$. Thus $(i)$ holds.

Let us now prove $(ii)$. Let us choose a sequence $d_k \nearrow \dim_{tH} K$. Theorem~\ref{t:3=} yields
$d_k < \dim_{tH}K=\min P_{n}$ for every $k\in \mathbb{N}^{+}$.
Hence $\mathcal{N}_{d_k}$ is nowhere dense by the
definition of $P_{n}$. It follows from the definition of
$\mathcal{N}_d$ that for every $f\in C(K)\setminus
\mathcal{N}_{d_k}$ there exists a non-degenerate
interval $I_{f,d_k}$ such that $\dim_{H}f^{-1}(y)\geq d_k-1$ for every $y\in I_{f,d_k}$. But then $(ii)$
holds for every $ f \in C(K)\setminus (\bigcup_{k\in \mathbb{N}^{+}}
\ \mathcal{N}_{d_k})$, and this latter set is
clearly co-meager, which concludes the proof of the theorem.
\end{proof}

This immediately implies

\begin{corollary}
\label{c:ft} If $K$ is a compact metric space with $\dim_t K > 0$ then for the generic $f \in C(K)$
$$\sup \left\{ \dim_{H}f^{-1}(y) : y \in \mathbb{R} \right\} = \dim_{tH} K - 1.$$
\end{corollary}

\subsection{Homogeneous and self-similar compact metric spaces}

In this subsection we show that if the compact metric space is sufficiently
homogeneous, e.g. self-similar (see \cite{F} or \cite{Ma}) then we can say
much more.

\begin{theorem} \label{t:hom}Let $K$ be a compact metric space with $\dim_{t}K>0$
such that $\dim_{tH} B(x,r) = \dim_{tH} K$ for every $x \in K$ and $r>0$. Then
for the generic $f\in C(K)$ for the generic $y\in f(K)$
$$\dim_{H} f^{-1}(y)=\dim_{tH}K-1.$$
\end{theorem}

\begin{remark} In fact, the authors show in \cite{BBE2} that the condition in Theorem~\ref{t:hom} is also necessary:
If $K$ is a compact metric space with $\dim_{t}K>0$ such that $\dim_{H} f^{-1}(y)=\dim_{tH}K-1$
for the generic $f\in C(K)$ and for the generic $y\in f(K)$ then $\dim_{tH} B(x,r) = \dim_{tH} K$ for every $x \in K$ and $r>0$.
\end{remark}

Before turning to the proof of this theorem we formulate a corollary.
Recall that $K$ is \emph{self-similar} if there are injective contractive similitudes
$\varphi_1, \dots , \varphi_k : K \to K$ such that $K = \bigcup_{i=1}^k \varphi_i
(K)$. The sets of the form $\varphi_{i_1} \circ \varphi_{i_2} \circ \dots \circ
\varphi_{i_m} (K)$ are called the \emph{elementary pieces} of $K$. It is easy to see
that every ball in $K$ contains an elementary piece. Moreover, by Corollary~\ref{c:bi-Lip} the topological Hausdorff dimension of every elementary piece
is $\dim_{tH} K$. Hence, using monotonicity as well, we obtain that if $K$ is
self-similar then $\dim_{tH} B(x,r) = \dim_{tH} K$ for every $x \in K$ and
$r>0$. This yields the following.

\begin{corollary} \label{c:self}
Let $K$ be a self-similar compact metric space with
  $\dim_{t}K>0$. Then
for the generic $f\in C(K)$ for the generic $y\in f(K)$
$$\dim_{H} f^{-1}(y)=\dim_{tH}K-1.$$
\end{corollary}

\begin{proof}[Proof of Theorem~\ref{t:hom}]
Theorem~\ref{ft} implies that for the generic $f\in C(K)$ for
\emph{every} $y\in \mathbb{R}$ we have $\dim_{H}f^{-1}(y)\leq
\dim_{tH} K-1$, so we only have to prove the opposite inequality.

Let us consider a sequence $0<d_k \nearrow \dim_{tH} K$. For $f\in C(K)$ and $k\in \mathbb{N}^{+}$ let
\[
L_{f,k} = \left\{y\in f(K): \dim_{H}f^{-1}(y)\geq d_k-1 \right\}.
\]
First we show that it suffices to construct for every $k\in \mathbb{N}^{+}$ a co-meager set
$\mathcal{F}_k \subseteq C(K)$
such that for every $f\in \mathcal{F}_k$ the set
$L_{f,k}$ is co-meager in $f(K)$. Indeed, then the set
$\mathcal{F} = \bigcap_{k \in \mathbb{N}^{+}} \
\mathcal{F}_{k}\subseteq C(K)$ is co-meager, and for every
$f\in \mathcal{F}$ the set $L_f=\bigcap _{k\in \mathbb{N}^{+}}
\ L_{f,k}\subseteq f(K)$ is also co-meager.
 Since
for every $y \in L_f$ clearly $\dim_{H}f^{-1}(y)\geq \dim_{tH} K-1$,
this finishes the proof.

Let us now construct such an $\mathcal{F}_k$ for a fixed
$k \in \mathbb{N}^{+}$. Let $\{ B_n \}_{n \in \mathbb{N}}$ be a
countable basis of $K$ consisting of closed balls, and for all $n
\in \mathbb{N}$ let $R_n \colon C(K)\to C(B_n)$ be defined as
\[
R_n (f)=f|_{B_n}.
\]
 Let us also define
\[
\mathcal{B}_n =\left\{f\in C( B_n ): \exists I_{f,n}
\textrm{
  ~s.~t.~}
\forall y\in I_{f,n}~ \dim_{H}f^{-1}(y)\geq d_k-
1 \right\},
\]
(where $I_{f,n}$ is understood to be a non-degenerate
interval). Finally, let us define
\[
\mathcal{F}_{k}=\bigcap_{n \in \mathbb{N}} R_n^{-1}
(\mathcal{B}_n).
\]
First we show that $\mathcal{F}_k$ is co-meager. By our
assumption $\dim_{tH} B_n = \dim_{tH} K > d_k$
(which also implies $\dim_t B_n > 0$ by Fact~\ref{<=>}, since
$d_k>0$), thus Theorem~\ref{ft}
yields that $\mathcal{B}_n$ is co-meager in $C(B_n)$.
Lemma~\ref{lem1} implies that $R_n^{-1}(\mathcal{B}_n)$ is co-meager
in $C(K)$ for all $n \in \mathbb{N}$, thus $\mathcal{F}_k$
is also co-meager.

It remains to show that for every $f\in \mathcal{F}_k$ the
set $L_{f,k}$ is co-meager in $f(K)$. Let us fix $f\in
\mathcal{F}_k$. We will actually show that
$L_{f,k}$ contains an open set in $\mathbb{R}$ which is a
dense subset of $f(K)$. So let $U\subseteq \mathbb{R}$ be an open
set in $\mathbb{R}$ such that $f(K)\cap U\neq \emptyset$. It is
enough to prove that $L_{f,k}\cap U$ contains an interval.
Since the sets $B_n$ form a basis, the continuity of $f$ implies that
there exists an $n \in \mathbb{N}$ such that $f(B_n) \subseteq U$.
It is easy to see using the definition of $\mathcal{F}_k$
that $f|_{B_n} \in \mathcal{B}_n$, so there exists a non-degenerate
interval $I_{f|_{B_n},k}$ such that for all $y\in
I_{f|_{B_n},k}$ we have
\[
\dim _{H}f^{-1}(y)\geq \dim_{H}\left( f|_{B_n}\right)^{-1}(y)\geq
d_k - 1 .
\]
Thus $I_{f|_{B_n},k} \subseteq L_{f,k}$. Then
$d_k-1> -1$ implies $\left( f|_{B_n}\right)^{-1}(y) \neq \emptyset$
for every $y\in I_{f|_{B_n},k}$, thus
$I_{f|_{B_n},k}\subseteq f(B_n)$. But it follows from
$f(B_n)\subseteq U$ that $I_{f|_{B_n},k}\subseteq U$.
Hence $I_{f|_{B_n},k}\subseteq L_{f,k}\cap U$
and this completes the proof.
\end{proof}

\section{Open Problems}
\label{c:problems}

First let us recall the most interesting open problem.

\begin{w}
Determine the almost sure topological Hausdorff
dimension of the range of the $d$-dimensional Brownian motion for $d=2$ and $d=3$. Equivalently, determine the smallest $c \ge 0$ such that the range can be decomposed into a totally disconnected set and a set of Hausdorff dimension at most $c-1$ almost surely.
\end{w}

Now we collect a few more open problems. The first one concerns a certain
Darboux property.

\begin{problem}
Let $B \subseteq \mathbb{R}^d$ be a Borel set and $1 \le c <
\dim_{tH} B$ be arbitrary. Does there exist a Borel set $B' \subseteq
B$ with $\dim_{tH} B' = c$?
\end{problem}

The following problem is motivated by the proof of Theorem~\ref{t:lower}.
The idea is to look for some structural reason behind large
topological Hausdorff dimension.

\begin{problem}
Is it true that a compact metric space $K$ satisfies
$\dim_{tH} K \ge c$ iff it contains a family of \emph{disjoint}
non-degenerate continua such that each set meeting all members
of this family is of Hausdorff dimension at least $c-1$?
\end{problem}

The following remark shows that by dropping disjointness the problem becomes
rather simple.

\begin{remark} If $K$ is a non-empty compact metric space and $\mathcal{S}$ is the collection of subsets of $K$
intersecting every non-degenerate continuum then
$$\dim_{tH} K=\min\{ \dim_H S+1: S\in \mathcal{S}\}.$$
Indeed, first let $S\in \mathcal{S}$ be arbitrary, and we prove that $\dim_{tH} K\leq \dim_{H} S+1$. We may assume that $S$ is $G_{\delta}$, since we can
take $G_{\delta}$ hulls with the same Hausdorff dimension. Then $K\setminus S$ is $\sigma$-compact. A compact subset of $K\setminus S$
does not contain non-degenerate continua by definition, so it is totally disconnected, thus it has topological dimension at most zero. Therefore the countable stability of topological dimension zero for closed sets \cite[1.3.1.]{Eng} yields that $\dim_t (K\setminus S)\leq 0$. Hence Theorem~\ref{t:tHdecomp} implies $\dim_{tH} K\leq \dim_{H} S+1$.

Now we prove that there exists $S\in \mathcal{S}$ with $\dim_{tH} K= \dim_{H} S+1$. Theorem~\ref{t:tHdecomp} yields that there is a set $S\subseteq K$ with $\dim_H S=\dim_{tH}K-1$ and $\dim_{t}(K\setminus S)\leq 0$. Then $K\setminus S$ cannot contain any non-degenerate continuum, so $S\in \mathcal{S}$.
\end{remark}

Notice that the above remark does not apply even to $G_{\delta}$ subspaces of
Euclidean spaces, see Example~\ref{ex:Maz}.

\bigskip

Finally, we consider other notions of dimension.

\begin{problem}
What is the right notion to describe the packing, lower box, or upper box dimension of the level sets of the generic continuous function $f\in C(K)$?
\end{problem}

\begin{remark} We can analogously define topological packing, or lower box, or upper box dimension, respectively. However, one can show that these definitions and some natural modifications of them do not solve the above problem. The reason why these concepts behave differently is that box dimensions are not even countable stable, and packing dimension does not admit $G_{\delta}$ hulls: It is easy to see that
every $G_{\delta}$ hull of $\mathbb{Q}$ has packing dimension $1$.
\end{remark}

\subsection*{Acknowledgments}
We are indebted to U. B. Darji and A. M\'ath\'e for some illuminating
discussions, in particular to U. B. Darji for suggesting ideas that led to the
results of Section~\ref{s:equivalent} and hence to the simplification of
certain proofs. We thank the anonymous referee for several helpful suggestions.

\end{document}